\documentclass[11pt]{amsart}
\usepackage{amssymb}
\usepackage{amsmath}
\usepackage{mathtools} 
\usepackage{graphicx}
\usepackage{tikz}
\usepackage{fancyhdr}
\usepackage[top=1.1in, bottom=0.9in, left=1in, right=1in]{geometry}

\usepackage{enumerate}

\numberwithin{equation}{section}

\title{Clique decompositions of multipartite graphs and completion of Latin squares}

\date{\today}
\author{Ben Barber, Daniela K\"uhn, Allan Lo, Deryk Osthus and Amelia Taylor}
\thanks{The research leading to these results was partially supported by the European Research Council
under the European Union's Seventh Framework Programme (FP/2007--2013) / ERC Grant
Agreement no. 258345 (B.~Barber and D.~K\"uhn) and 306349 (D.~Osthus). The research was also partially supported by the EPSRC, grant no. EP/M009408/1 (D.~K\"uhn and D.~Osthus), and by the Royal Society and the Wolfson Foundation (D.~K\"uhn).}

\newtheorem{firstthm}{Proposition}[section]
\newtheorem{theorem}[firstthm]{Theorem}
\newtheorem{prop}[firstthm]{Proposition}
\newtheorem{lemma}[firstthm]{Lemma}
\newtheorem{cor}[firstthm]{Corollary}

\newtheorem{conj}[firstthm]{Conjecture}
\newtheorem{fact}[firstthm]{Fact}

\newdimen\margin   
\def\textno#1&#2\par{%
   \margin=\hsize
   \advance\margin by -4\parindent
          \setbox1=\hbox{\sl#1}%
   \ifdim\wd1 < \margin
      $$\box1\eqno#2$$%
   \else
      \bigbreak
      \hbox to \hsize{\indent$\vcenter{\advance\hsize by -3\parindent
      \it\noindent#1}\hfil#2$}%
      \bigbreak
   \fi}

\newcommand{\claim}[2]{\medskip \noindent \textbf{#1: } \emph{#2}}
\newcommand{\pclaim}[1]{\medskip \noindent \textit{#1.}}

\begin{document}

\def\COMMENT#1{}
\def\TASK#1{}
\let\TASK=\footnote

\def\eps{{\varepsilon}}
\newcommand{\ex}{\mathbb{E}}
\newcommand{\pr}{\mathbb{P}}
\newcommand{\N}{\mathbb{N}}
\newcommand{\cP}{\mathcal{P}}
\newcommand{\cF}{\mathcal{F}}
\newcommand{\cM}{\mathcal{M}}
\newcommand{\cA}{\mathcal{A}}
\newcommand{\cC}{\mathcal{C}}
\newcommand{\cD}{\mathcal{D}}
\newcommand{\cK}{\mathcal{K}}
\newcommand{\cH}{\mathcal{H}}
\newcommand{\cS}{\mathcal{S}}
\newcommand{\cL}{\mathcal{L}}
\newcommand{\cI}{\mathcal{I}}
\newcommand{\cQ}{\mathcal{Q}}
\newcommand{\cT}{\mathcal{T}}

\newcommand{\defn}{\emph}

\begin{abstract}  \noindent
Our main result essentially reduces the problem of finding an edge-decomposition of a balanced $r$-partite graph of large minimum degree
into $r$-cliques to the problem of finding a fractional $r$-clique decomposition or an approximate one.
Together with very recent results of Bowditch and Dukes as well as Montgomery on fractional decompositions into triangles and cliques 
respectively, this gives the best known bounds on the minimum degree which ensures an edge-decomposition of an $r$-partite graph into $r$-cliques
(subject to trivially necessary divisibility conditions).
The case of triangles translates into the setting of partially completed Latin squares 
and more generally the case of $r$-cliques translates into the setting of partially completed mutually orthogonal Latin squares.
\end{abstract}

\maketitle

\section{Introduction}\label{sec:intro}

A \defn{$K_r$-decomposition} of a graph $G$ is a partition of its edge set $E(G)$ into cliques of order $r$.
If $G$ has a $K_r$-decomposition, then certainly $e(G)$ is divisible by $\binom r 2$ and the degree of every vertex is divisible by $r-1$.
A classical result of Kirkman~\cite{kirk} asserts that, when $r=3$, these two conditions  ensure that $K_n$ has a triangle decomposition
(i.e.~Steiner triple systems exist).
This was generalized to arbitrary $r$ (for large $n$) by Wilson~\cite{Wilson} and to hypergraphs by Keevash~\cite{keev}.
Recently, there has been much progress in extending this from decompositions of complete host graphs to decompositions of graphs which are allowed to be far from complete (see the final paragraphs in Section~\ref{sec:subintro}).
In this paper, we investigate this question in the $r$-partite setting.
This is of particular interest as it implies results on the completion of partial Latin squares 
and more generally partial mutually orthogonal Latin squares.

\subsection{Clique decompositions of $r$-partite graphs} \label{sec:subintro}
Our main result (Theorem~\ref{thm:main}) states that if $G$ is (i) balanced $r$-partite, (ii) satisfies the necessary divisibility conditions and 
(iii) its minimum degree is at least a little larger than the minimum degree which guarantees an approximate decomposition into $r$-cliques,
then $G$ in fact has a decomposition into $r$-cliques.
(Here an approximate decomposition is a set of edge-disjoint copies of $K_r$ which cover almost all edges of $G$.)
To state this result precisely, we need the following definitions.

We say that a graph or multigraph $G$ on $(V_1, \dots, V_r)$ is \emph{$K_r$-divisible} if $G$ is $r$-partite with vertex classes $V_1, \dots, V_r$ and for all $1\leq j_1,j_2\leq r$ and every $v\in V(G)\setminus(V_{j_1}\cup V_{j_2})$, $$d(v,V_{j_1})=d(v,V_{j_2}).$$
Note that in this case, for all $1\leq j_1, j_2, j_3, j_4\leq r$ with $j_1\neq j_2$, $j_3\neq j_4$, we automatically have $e(V_{j_1}, V_{j_2})=e(V_{j_3}, V_{j_4})$. In particular, $e(G)$ is divisible by $e(K_r)=\binom{r}{2}$.

Let $G$ be an $r$-partite graph on $(V_1, \dots, V_r)$ with $|V_1|=\dots=|V_r|=n$. 
Let $$\hat{\delta}(G):=\min \{d(v, V_j): 1\leq j\leq r,\; v\in V(G)\setminus V_j\}.$$
An 
\emph{$\eta$-approximate} $K_r$-decomposition of $G$ is a set of edge-disjoint copies of $K_r$ covering all but at most $\eta n^2$ edges of $G$. We define $\hat{\delta}^\eta_{K_r}(n)$ to be the infimum over all  $\delta$ such that every $K_r$-divisible graph $G$ on $(V_1, \dots, V_r)$ with $|V_1|=\dots=|V_r|=n$ and $\hat{\delta}(G)\geq \delta n$ has an $\eta$-approximate $K_r$-decomposition.
Let $\hat{\delta}^\eta_{K_r}:=\limsup_{n\rightarrow \infty} \hat{\delta}^\eta_{K_r}(n)$.
So if $\eps>0$ and $G$ is sufficiently large, $K_r$-divisible and $\hat\delta(G)>(\hat{\delta}^\eta_{K_r}+\eps)n$, then $G$ has an $\eta$-approximate $K_r$-decomposition. Note that it is important here that $G$ is $K_r$-divisible. Take, for example, the complete $r$-partite graph with vertex classes of size $n$ and remove $\lceil\eta n\rceil$ edge-disjoint perfect matchings between one pair of vertex classes. The resulting graph $G$ satisfies $\hat\delta(G)=n-\lceil\eta n\rceil$, yet has no $\eta$-approximate $K_r$-decomposition whenever $r\geq 3$.

\begin{theorem}\label{thm:main}
For every $r \ge 3$ and every $\eps>0$ there exists an $n_0 \in \mathbb{N}$ and an $\eta>0$ such that the following holds for all $n \ge n_0$.
Suppose $G$ is a $K_r$-divisible graph on $(V_1, \dots, V_r)$ with $|V_1|=\dots=|V_r|=n$. 
If $\hat{\delta}(G) \geq (\hat{\delta}^\eta_{K_r}+\eps)n$, then $G$ has a $K_r$-decomposition.
\end{theorem}
By a result of Haxell and R\"odl~\cite{HaxellRodl}, the existence of an approximate decomposition follows from that of a fractional decomposition.
So together with very recent results of Bowditch and Dukes~\cite{Dukes2} as well as Montgomery~\cite{mont} on fractional decompositions into 
triangles and cliques respectively, Theorem~\ref{thm:main} implies the following explicit bounds.
We discuss this derivation in Section~\ref{fractional}.
\begin{theorem} \label{thm:explicit}
For every $r \ge 3$ and every $\eps>0$ there exists an $n_0 \in \mathbb{N}$ such that the following holds for all $n \ge n_0$.
Suppose $G$ is a $K_r$-divisible graph on $(V_1, \dots, V_r)$ with $|V_1|=\dots=|V_r|=n$. 
\begin{itemize}
\item[(i)] If $r=3$ and $\hat{\delta}(G) \geq \left( \frac{24}{25}+\eps \right)n$, then $G$ has a $K_3$-decomposition.
\item[(ii)] If $r\geq 4$ and $\hat{\delta}(G) \geq \left(1- \frac{1}{10^6r^3}+\eps \right)n$, then $G$ has a $K_r$-decomposition.
\end{itemize}
\end{theorem}
If $G$ is the complete $r$-partite graph, this corresponds to a theorem of Chowla, Erd\H{o}s and Straus~\cite{cherst}.
A bound of $(1-1/(10^{16}r^{29}))n$ was claimed by Gustavsson~\cite{Gustavsson}.
The following conjecture seems natural (and is implicit in~\cite{Gustavsson}).
\begin{conj} \label{conj:decomp}
For every $r \ge 3$ there exists an $n_0 \in \mathbb{N}$ such that the following holds for all $n \ge n_0$.
Suppose $G$ is a $K_r$-divisible graph on $(V_1, \dots, V_r)$ with $|V_1|=\dots=|V_r|=n$. 
If $\hat{\delta}(G) \geq (1-1/(r+1))n$, then $G$ has a $K_r$-decomposition.
\end{conj}
A construction which matches the lower bound in Conjecture~\ref{conj:decomp} is described in Section~\ref{sec:extremal} (this construction also gives a similar lower bound on $\hat\delta^\eta_{K_r}$).
In the non-partite setting, the triangle case is a long-standing conjecture by Nash-Williams~\cite{NashWilliams}
that every graph $G$ on $n$ vertices with minimum degree at least $3n/4$ has a triangle decomposition
(subject to divisibility conditions).
Barber, K\"uhn, Lo and Osthus~\cite{mindeg} recently reduced its asymptotic version to proving an approximate or fractional version.
Corresponding results on fractional triangle decompositions were proved by Yuster~\cite{YusterFractKr}, Dukes~\cite{Dukes}, Garaschuk~\cite{Garaschuk} and Dross~\cite{dross}.

More generally~\cite{mindeg} also gives results for arbitrary graphs, and corresponding 
fractional decomposition results have been obtained by
Yuster~\cite{YusterFractKr}, Dukes~\cite{Dukes} as well as Barber, K\"uhn, Lo, Montgomery and Osthus~\cite{bklmo}.
Further results on $F$-decompositions of non-partite graphs (leading on from \cite{mindeg}) have been obtained by Glock, K\"uhn,
Lo, Montgomery and Osthus~\cite{GKLMO}. Amongst others, for any bipartite graph~$F$, they asymptotically determine the minimum degree
threshold which guarantees an $F$-decomposition.
Finally, Glock, K\"uhn, Lo and Osthus~\cite{GKLO} gave a new (combinatorial) proof of the existence of designs.
The results in~\cite{GKLO} generalize those in~\cite{keev}, in particular, they imply a resilience version and a decomposition result for hypergraphs of large minimum degree.

\subsection{Mutually orthogonal Latin squares and $K_r$-decompositions of $r$-partite graphs} \label{latin}

A \defn{Latin square} $\mathcal T$ of order $n$ is an $n \times n$ grid of cells, each containing a symbol from $[n]$, such that no symbol appears twice in any row or column.
It is easy to see that $\mathcal T$  corresponds to a $K_3$-decomposition of the complete tripartite graph $K_{n,n,n}$ with vertex classes 
consisting of the rows, columns and symbols.%
\COMMENT{
Indeed, $|\mathcal T| = n^2$ and $e(K_{n,n,n}) = 3n^2$, so $\mathcal T$ contains the right number of edges in total.
And by the conditions defining a Latin square, specifying any two of $i$, $j$ and $L_{ij}$ determines the third uniquely, so the triangles in $\mathcal T$ are edge-disjoint.
Similarly, each $K_3$-decomposition of $K_{n,n,n}$ defines a Latin square in the natural way, so this correspondence is bijective.}

Now suppose that we have a partial Latin square; that is, a partially filled in grid of cells satisfying the conditions defining a Latin square.
When can it be completed to a Latin square?%
\COMMENT{Colbourn has a paper showing that this problem in NP-complete.
It's referenced in Bartlett's thesis, but I've only seen a preview of the first 20 pages, and Colbourn has about 5000 papers about Latin squares.}
This natural question has received much attention.
For example, a classical theorem of Smetaniuk~\cite{smet} as well as Anderson and Hilton~\cite{andhil} 
states that this is always possible if at most $n-1$ entries have been made
(this bound is best possible).
The case $r=3$ of Conjecture~\ref{conj:decomp} implies that,
provided we have used each row, column and symbol at most $n/4$ times, it should also still be possible to complete a partial Latin square.
This was conjectured by Daykin and H\"aggkvist~\cite{DH}.
(For a discussion of constructions which match this conjectured bound, see Wanless~\cite{Wanless}.)
Note that the conjecture of  Daykin and H\"aggkvist corresponds to the special case of Conjecture~\ref{conj:decomp} when $r=3$
and the condition of $G$ being $K_r$-divisible is replaced by that of $G$ being obtained from $K_{n,n,n}$ by deleting edge-disjoint triangles.

More generally, we say that two Latin squares $R$ (red) and $B$ (blue) drawn in the same $n \times n$ grid of cells are \defn{orthogonal} if no two cells contain the same combination of red symbol and blue symbol.
In the same way that a Latin square corresponds to a $K_3$-decomposition of $K_{n,n,n}$, 
a pair of orthogonal Latin squares corresponds to a $K_4$-decomposition of $K_{n,n,n,n}$ 
where the vertex classes are rows, columns, red symbols and blue symbols.
More generally, there is a natural bijection between sequences of $r-2$ \defn{mutually orthogonal} 
Latin squares (where every pair from the sequence are orthogonal) and $K_r$-decompositions of complete 
$r$-partite graphs with vertex classes of equal size. Sequences of mutually orthogonal Latin squares are also known as \defn{transversal designs}.
Theorem~\ref{thm:explicit} can be used to show the following (see Section~\ref{sec:latinproof} for details).
\begin{theorem}\label{thm:mols}
For every $r \ge 3$ and every $\eps>0$ there exists an $n_0 \in \mathbb{N}$ such that the following holds for all $n \ge n_0$. Let
	\begin{equation*}
	c_r:=
	\begin{cases}
	\frac{1}{25} & \text{if}\ r=3,\\
	\frac{9}{10^7r^3} & \text{if}\ r\geq 4.
	\end{cases}
	\end{equation*}
Let $\cT_1, \dots, \cT_{r-2}$ be a sequence of mutually orthogonal partial $n\times n$ Latin squares (drawn in the same $n\times n$ grid).
Suppose that each row and each column of the grid contains at most $(c_r-\eps)n$ non-empty cells and each coloured symbol is used at most $(c_r-\eps)n$ times.
Then $\cT_1, \dots, \cT_{r-2}$ can be completed to a sequence of mutually orthogonal Latin squares.
\end{theorem}
Here, by a non-empty cell we mean a cell containing at least one symbol (in at least one of the colours).
The best previous bound for the triangle case $r=3$ is due to Bartlett~\cite{bart},
who obtained a minimum degree bound of $(1- 10^{-4})n$. 
This improved an earlier bound of Chetwynd and H\"aggkvist~\cite{chehag} as well as the one claimed by Gustavsson~\cite{Gustavsson}. We are not aware of any previous upper or lower bounds  for $r \ge 4$.

\subsection{Fractional and approximate decompositions} \label{fractional}

A \defn{fractional $K_r$-decomposition} of a graph $G$ is a non-negative weighting of the copies of $K_r$ in $G$ such that the total weight of all the copies of $K_r$ containing any fixed edge of $G$ is exactly $1$.
Fractional decompositions are of particular interest to us because of the following result of Haxell and R\"odl, of which we state only a very special case (see~\cite{Yuster} for a shorter proof).

\begin{theorem}[Haxell and R\"odl \cite{HaxellRodl}] \label{HRthm}
For every $r\geq 3$ and every $\eta > 0$ there exists an $n_0 \in \N$ such that the following holds.
Let $G$ be a graph on $n \geq n_0$ vertices that has a fractional $K_r$-decomposition.
Then $G$ has an $\eta$-approximate $K_r$-decomposition.
\end{theorem}

We define $\hat{\delta}^*_{K_r}(n)$ to be the infimum over all $\delta$ such that every $K_r$-divisible graph $G$ on $(V_1, \dots, V_r)$ with $|V_1|=\dots=|V_r|=n$ and $\hat{\delta}(G)\geq \delta n$ has a fractional $K_r$-decomposition.
Let $\hat{\delta}^*_{K_r}:=\limsup_{n\rightarrow \infty} \hat{\delta}^*_{K_r}(n)$.
Theorem~\ref{HRthm} implies that, for every $\eta > 0$, $\hat \delta^\eta_{K_r} \leq \hat \delta^*_{K_r}$.
Together with Theorem~\ref{thm:main}, this yields the following.
\begin{cor}\label{thm:mainfrac}
For every $r \ge 3$ and every $\eps>0$ there exists an $n_0 \in \mathbb{N}$ such that the following holds for all $n \ge n_0$.
Suppose $G$ is a $K_r$-divisible graph on $(V_1, \dots, V_r)$ with $|V_1|=\dots=|V_r|=n$. 
If $\hat{\delta}(G) \geq (\hat{\delta}^*_{K_r}+\eps)n$, then $G$ has a $K_r$-decomposition.
\end{cor}
In particular, to prove Conjecture~\ref{conj:decomp} asymptotically, it suffices to show that $\hat \delta^*_{K_r} \leq 1 - 1/(r+1)$.%
\COMMENT{Extremal example can also be made to work fractionally.}
Similarly, improved bounds on $\hat \delta^*_{K_r}$ would lead to improved bounds in Theorem~\ref{thm:mols}
(see Corollary~\ref{cor:fracmols}).

For triangles, the best bound on the `fractional decomposition threshold' is due to Bowditch and Dukes~\cite{Dukes2}.
\begin{theorem}[Bowditch and Dukes~\cite{Dukes2}]
$\hat \delta^*_{K_3} \le \frac{24}{25}$.
\end{theorem}
\noindent For arbitrary cliques, Montgomery obtained the following bound. Somewhat weaker bounds (obtained by different methods) are also proved in \cite{Dukes2}.
\begin{theorem}[Montgomery~\cite{mont}]
For every $r\ge 3$, $\hat \delta^*_{K_r} \le 1-\frac{1}{10^6r^3}$.
\end{theorem}
\noindent Note that together with Corollary~\ref{thm:mainfrac}, these results immediately imply Theorem~\ref{thm:explicit}. 

This paper is organised as follows. In Section~\ref{sec:notat} we introduce some notation and tools which will be used throughout this paper. In Section~\ref{sec:extremlatin} we give extremal constructions which support the bounds in Conjecture~\ref{conj:decomp} and we provide a proof of Theorem~\ref{thm:mols}. Section~\ref{sec:sketch} outlines the proof of Theorem~\ref{thm:main} and guides the reader through the remaining sections in this paper.

\section{Notation and tools}\label{sec:notat}

Let $G$ be a graph and let $\cP=\{U^1, \dots, U^k\}$ be a partition of $V(G)$. We write $G[U^1]$ for the subgraph of $G$ induced by $U^1$ and $G[U^1, U^2]$ for the bipartite subgraph of $G$ induced by the vertex classes $U^1$ and $U^2$. We will also sometimes write $G[U^1, U^1]$ for $G[U^1]$. We write $G[\cP]:=G[U^1, \dots, U^k]$ for the $k$-partite subgraph of $G$ induced by the partition $\cP$. We write $U^{<i}$ for $U^1\cup \dots \cup U^{i-1}$. We say the partition $\cP$ is \emph{equitable} if its parts differ in size  by at most one. Given a set $U\subseteq V(G)$, we write $\cP[U]$ for the restriction of $\cP$ to $U$.

Let $G$ be a graph and let $U, V\subseteq V(G)$. We write $N_G(U,V):=\{v\in V: xv\in E(G)\text{ for all } x\in U\}$ and $d_G(U,V):=|N_G(U,V)|$. For $v\in V(G)$, we write $N_G(v,V)$ for $N_G(\{v\},V)$ and $d_G(v,V)$ for $d_G(\{v\},V)$. If $U$ and $V$ are disjoint, we let $e_G(U,V):=e(G[U,V])$. 

Let $G$ and $H$ be graphs. We write $G-H$ for the graph with vertex set $V(G)$ and edge set $E(G)\setminus E(H)$. We write $G\setminus H$ for the subgraph of $G$ induced by the vertex set $V(G)\setminus V(H)$. We call a vertex-disjoint collection of copies of $H$ in $G$ an \emph{$H$-matching}. If the $H$-matching covers all vertices in $G$, we say that it is \emph{perfect}.

Throughout this paper, we consider a partition $V_1, \dots, V_r$ of a vertex set $V$ such that $|V_j|=n$ for all $1\leq j\leq r$. Given a set $U\subseteq V$, we write
$$U_j:=U\cap V_j.$$ 
A \emph{$k$-partition} of $V$ is a partition $\cP=\{U^1, \dots, U^k\}$ of $V$ such that the following hold:
\begin{enumerate}[({Pa}1)]
	\item for each $1\leq j \leq r$, $\{U^i_j: 1\leq i\leq k\}$ is an equitable partition of $V_j$;\label{Pa1}
	\item for each $1\leq i\leq k$, $|U^i_1|=\dots=|U^i_r|$.\label{Pa2}%
	\COMMENT{just for convenience, can then assume graphs in each iteration have equal class sizes}
\end{enumerate}
If $G$ is an $r$-partite graph on $(V_1, \dots, V_r)$, we sometimes also refer to a $k$-partition of $G$ (instead of a $k$-partition of $V(G)$).
We write $K_r(k)$ for the complete $r$-partite graph with vertex classes of size $k$.
We say that an $r$-partite graph $G$ on $(V_1, \dots, V_r)$ is \emph{balanced} if $|V_1|=\dots=|V_r|$.

We use the symbol $\ll$ to denote hierarchies of constants, for example $1/n\ll a\ll b<1$, where the constants are chosen from right to left. The notation $a \ll b$ means that there exists an increasing function $f$ for which the result holds whenever $a \leq f(b)$.

Let $m,n,N\in \N$ with $m,n<N$. The \emph{hypergeometric distribution} with parameters $N$, $n$ and $m$ is the distribution of the random variable $X$ defined as follows. Let $S$ be a random subset of $\{1,2, \dots, N\}$ of size $n$ and let $X:=|S\cap \{1,2,\dots, m\}|$. We will frequently use the following bounds, which are simple forms of Hoeffding's inequality.

\begin{lemma}[see {\cite[Remark 2.5 and Theorem 2.10]{JLR}}] \label{lem:chernoff}
Let $X\sim B(n,p)$ or let $X$ have a hypergeometric distribution with parameters $N,n,m$.
Then
$\pr(|X - \ex(X)| \geq t) \leq 2e^{-2t^2/n}$.
\end{lemma}

\begin{lemma}[see {\cite[Corollary 2.3 and Theorem 2.10]{JLR}}] \label{lem:chernoff2}
Suppose that $X$ has binomial or hypergeometric distribution and $0<a<3/2$. Then $\pr(|X-\ex(X)|\geq a\ex(X))\leq 2e^{-a^2\ex(X)/3}$.
\end{lemma}

\section{Extremal graphs and completion of Latin squares}\label{sec:extremlatin}

\subsection{Extremal graphs}\label{sec:extremal}

The following proposition shows that the minimum degree bound conjectured in Conjecture~\ref{conj:decomp} would be best possible. It also provides a lower bound on the approximate decomposition threshold $\hat\delta^{\eta}_{K_r}$ (and thus on the fractional decomposition threshold $\hat\delta^*_{K_r}$).

\begin{prop}\label{prop:extrem}
Let $r\in \N$ with $r\geq 3$ and let $\eta > 0$. For infinitely many $n$, there exists a $K_r$-divisible graph $G$ on $(V_1, \dots, V_r)$ with $|V_1|=\dots=|V_r|=n$ and $\hat\delta(G)=\lceil (1-1/(r+1))n\rceil -1$ which does not have a $K_r$-decomposition.
Moreover, $\hat\delta^{\eta}_{K_r} \geq 1-1/(r+1) - \eta$.
\end{prop}

\begin{proof}
Let $m\in \N$ with $1/m\ll \eta$%
\COMMENT{$(r+1)(r-2)r(r-1)(\eta n-1)/4>\eta n$ so $\eta n>6/5$ would do.}
and let $n:=(r-1)m$. Let $\{U^1, \dots, U^{r-1}\}$ be a partition of $V_1\cup \dots \cup V_r$ such that, for each $1\leq i\leq r-1$ and each $1\leq j\leq r$, $U^i_j=U^i\cap V_j$ has size $m$. 

Let $G_0$ be the intersection of the complete $r$-partite graph on $(V_1, \ldots, V_r)$ and the complete $(r-1)$-partite graph on $(U^1, \ldots, U^{r-1})$.%
\COMMENT{Hall's theorem implies any regular bipartite graph has a perfect matching, so $K_{m,m}$ decomposes into $m$ perfect matchings. Take $q$ edge-disjoint perfect matchings between each pair $(U^i_{j_1}, U^i_{j_2})$.}
For each $1\leq q \leq m$ and each  $1\leq i\leq r-1$,  let $H^i_q$ be a graph formed by starting with the empty graph on $U^i$ and including a $q$-regular bipartite graph with vertex classes $(U^i_{j_1}, U^i_{j_2})$ for each $1\leq j_1<j_2\leq r$. 
Let $H_q := H^1_q \cup \cdots \cup H^{r-1}_q$ and let $G_q := G_0 \cup H_q$.
Observe that $G_q$ is regular, $K_r$-divisible and $$\hat \delta(G_q) = (r-2)m + q.$$

Now $G_0$ is $(r-1)$-partite, so every copy of $K_r$ in $G_q$ contains at least one edge of $H_q$.
Therefore, any collection of edge-disjoint copies of $K_r$ in $G$ will leave at least
\begin{align*}
\ell(G_q):=e(G_q)-e(H_q)\binom{r}{2}&=\big( (r-2)m+q-\binom{r}{2}q\big)\binom{r}{2}n\\
&=(m-(r+1)q/2)(r-2)\binom{r}{2}n
\end{align*}
edges of $G_q$ uncovered.
Let $q_0:= \lceil 2m/(r+1) \rceil -1$. Then $\ell(G_{q_0})>0$, so $G_{q_0}$ does not have a $K_r$-decomposition. Also, 
$$\hat \delta(G_{q_0})=(r-2)m +\lceil 2m/(r+1) \rceil-1 = \lceil (1-1/(r+1))n\rceil -1.$$

Now let $q_\eta:=\lceil 2m/(r+1)-\eta n \rceil$. We have $\hat \delta(G_{q_\eta})\geq (1-1/(r+1)-\eta)n$ and
\begin{align*}
\ell(G_{q_\eta})&\geq (m-(2m/(r+1)-\eta n+1)(r+1)/2)(r-2)\binom{r}{2}n\\
&=(\eta n-1)(r+1)(r-2)r(r-1)n/4\geq 6(\eta n-1)n>\eta n^2.
\end{align*}
Thus, $\hat\delta^{\eta}_{K_r} \geq 1-1/(r+1) - \eta$.
\end{proof}

\subsection{Completion of mutually orthogonal Latin squares}\label{sec:latinproof}

In this section, we give a proof of Theorem~\ref{thm:mols}. We also discuss how better bounds on the fractional decomposition threshold would
immediately lead to better bounds on $c_r$. For any $r$-partite graph $H$ on $(V_1, \dots, V_r)$, we let $\overline H$ denote the $r$-partite complement of $H$ on $(V_1, \dots, V_r)$.

\begin{proofof}{Theorem~\ref{thm:mols}}
By making $\eps$ smaller if necessary, we may assume that $\eps\ll 1$.
Let $n_0\in\mathbb{N}$ be such that $1/n_0\ll \eps, 1/r$.
Use $\cT_1, \dots, \cT_{r-2}$ to construct a balanced $r$-partite graph $G$ with vertex classes $V_j=[n]$ for $1\leq j\leq r$ as follows.
For each $1\leq i,j,k\leq n$ and each $1\leq m\leq r-2$, if in $\cT_m$ the cell $(i,j)$ contains the symbol $k$, include a $K_3$ on the vertices
$i\in V_{r-1}$, $j\in V_r$ and $k\in V_{m}$. (If the cell $(i,j)$ is filled in different $\cT_m$, this leads to multiple edges between
$i\in V_{r-1}$ and $j\in V_r$, which we disregard.) For each $1\leq i,j,k,k'\leq n$ and each $1\leq m<m'\leq r-2$ such that
the cell $(i,j)$ contains symbol $k$ in $\cT_m$ and symbol $k'$ in $\cT_{m'}$, add an edge between the vertices $k\in V_{m}$ and $k'\in V_{m'}$. 

If $r=3$, then $G$ is an edge-disjoint union of copies of $K_3$, so $G$ is $K_3$-divisible. Then $\overline G$ is also $K_3$-divisible and $\hat\delta(\overline G)\geq (24/25+\eps)n$. So we can apply Theorem~\ref{thm:explicit} to find a $K_3$-decomposition of $\overline G$ which we can then use to complete $\cT_1$ to a Latin square.

Suppose now that $r\geq 4$. Observe that $G$ consists of an edge-disjoint union of cliques $H_1, \dots, H_q$ such that,
for each $1\leq i\leq q$, $H_i$ contains an edge of the form $xy$ where $x\in V_{r-1}$ and $y\in V_r$. We have $q\leq (c_r-\eps)n^2$. 
We now show that we can extend $G$ to a graph of small maximum degree which can be decomposed into $q$ copies of~$K_r$.  We will do this by greedily extending each $H_i$ in turn to a copy $H_i'$ of $K_r$.
Suppose that $1\leq p\leq q$ and we have already found edge-disjoint $H_1', \dots, H_{p-1}'$. Given $v\in V(G)$, let $s(v, p-1)$ be the number of graphs in $\{H_1', \dots, H_{p-1}'\}\cup \{H_p, \dots, H_q\}$ which contain $v$. 
Suppose inductively that $s(v,p-1)\leq 10(c_r-\eps^2)n/9$ for all $v\in V(G)$. (This holds when $p=1$ by our assumption that
each row and each column of the grid contains at most $(c_r-\eps)n$ non-empty cells and each coloured symbol is used at most $(c_r-\eps)n$ times.)
For each $1\leq j\leq r$, let $B_j:=\{v\in V_j: s(v,p-1)\geq 10(c_r-\eps)n/9\}$. We have
\begin{equation}
|B_j|\leq \frac{q}{10(c_r-\eps)n/9}\leq \frac{9n}{10}.\label{eq:latin:a}
\end{equation}
Let $G_{p-1}:=G\cup \bigcup_{i=1}^{p-1}(H_i'-H_i)$. Note that
\begin{equation}
\hat\delta(\overline G_{p-1})\geq (1-10(c_r-\eps^2)/9)n,\label{eq:latin:b}
\end{equation}
by our inductive assumption.
We will extend $H_p$ to a copy of $K_r$ as follows. Let $\{j_1, \dots, j_m\}=\{j:1\leq j\leq r \text{ and }V(H_p)\cap V_j=\emptyset\}$. For each $j_i$ in turn, starting with $j_1$, choose one vertex $x_{j_i}$ from the set $N_{\overline G_{p-1}}(V(H_p)\cup \{x_{j_1}, \dots, x_{j_{i-1}}\}, V_{j_i}\setminus B_{j_i})$.  This is possible since \eqref{eq:latin:a} and \eqref{eq:latin:b} imply 
\begin{align*}
d_{\overline G_{p-1}}(V(H_p)\cup \{x_{j_1}, \dots, x_{j_{i-1}}\}, V_{j_i}\setminus B_{j_i})\geq (1/10-(r-1)10(c_r-\eps^2)/9)n>0.
\end{align*}
\COMMENT{for any $c_r\leq 9/10^2(r-1)$.}
 Let $H_p'$ be the copy of $K_r$ with vertex set $V(H_p)\cup \{x_j: 1\leq j\leq r \text{ and }V(H_p)\cap V_j=\emptyset\}$.  By construction, for every $v\in V(G)$, the number $s(v,p)$ of graphs in $\{H_1', \dots, H_{p}'\}\cup \{H_{p+1}, \dots, H_q\}$ which contain $v$ satisfies $s(v, p)\leq 10(c_r-\eps^2)n/9$.

Continue in this way to find edge-disjoint $H_1', \dots, H_q'$ such that $s(v,q)\leq 10(c_r-\eps^2)n/9$. Let $G_q:=\bigcup_{1\leq i\leq q}H_i'$. We have 
$\hat\delta(\overline G_q)\geq (1-10(c_r-\eps^2)/9)n= (1-1/10^6r^3+10\eps^2/9)n$
and, since $G_q$ is an edge-disjoint union of copies of $K_r$, we know that $\overline G_q$ is $K_r$-divisible. So we can apply Theorem~\ref{thm:explicit} to find a $K_r$-decomposition $\cF$ of $\overline G_q$. Note that $\cF':=\cF\cup \bigcup_{1\leq i\leq q}H_i'$ is a $K_r$-decomposition of the complete $r$-partite graph. Since $H_i\subseteq H_i'$ for each $1\leq i \leq q$, we can use $\cF'$ to complete $\cT_1, \dots, \cT_{r-2}$ to a sequence of mutually orthogonal Latin squares.
\end{proofof}

The proof of Theorem~\ref{thm:mols} also shows how better bounds for the fractional 
decomposition threshold $\hat{\delta}^*_{K_r}$ lead to better bounds on $c_r$.
More precisely, by replacing the `10/9' in the above inductive upper bound on $s(v,p-1)$ by `2'
and making the obvious adjustments to the calculations we obtain the following result.
\begin{cor} \label{cor:fracmols}
For all $r \ge 3$ and $n \in \mathbb{N}$, define $\beta_r(n)$ to be the supremum over all
$\beta$ so that the following holds:
Let $\cT_1, \dots, \cT_{r-2}$ be a sequence of mutually orthogonal partial $n\times n$ Latin squares (drawn in the same $n\times n$ grid).
Suppose that each row and each column of the grid contains at most $\beta n$ non-empty cells and each coloured symbol is used at most $\beta n$ times.
Then $\cT_1, \dots, \cT_{r-2}$ can be completed to a sequence of mutually orthogonal Latin squares.

Let $\beta_r:= \liminf_{n \to \infty} \beta_r(n)/n$. Also, for every $r \ge 3$, let
	\begin{equation*}
	\beta'_r:=
	\begin{cases}
	1-\hat{\delta}^*_{K_3} & \text{if}\ r=3,\\
	(1-\hat{\delta}^*_{K_r})/4 & \text{if}\ r\geq 4.
	\end{cases}
	\end{equation*}
	Then $\beta_r \ge \beta_r'$.
\end{cor}

\COMMENT{
\begin{proof}
Choose $0<\eps\ll 1$ and consider any $n\in\mathbb{N}$ with $1/n\ll \eps, 1/r$.
Let $\cT_1, \dots, \cT_{r-2}$ be a sequence of mutually orthogonal partial $n\times n$ Latin squares. 
Suppose that each row and each column of the grid contains at most $(\beta_r' -\eps) n$ non-empty cells and each coloured symbol is used at most $(\beta'_r -\eps)n$ times. We will show that
$\cT_1, \dots, \cT_{r-2}$ can be extended to a sequence of mutually orthogonal $n\times n$ Latin squares.
To do this, we first construct a balanced $r$-partite graph $G$ with vertex classes $V_j=[n]$ for $1\leq j\leq r$ as follows.
For each $1\leq i,j,k\leq n$ and each $1\leq m\leq r-2$, if in $\cT_m$ the cell $(i,j)$ contains the symbol $k$, include a $K_3$ on the vertices
$i\in V_{r-1}$, $j\in V_r$ and $k\in V_{m}$. (If the cell $(i,j)$ is filled in different $\cT_m$, this leads to multiple edges between
$i\in V_{r-1}$ and $j\in V_r$, which we disregard.) For each $1\leq i,j,k,k'\leq n$ and each $1\leq m<m'\leq r-2$ such that
the cell $(i,j)$ contains symbol $k$ in $\cT_m$ and symbol $k'$ in $\cT_{m'}$, add an edge between the vertices $k\in V_{m}$ and $k'\in V_{m'}$. 
If $r=3$, then $G$ is an edge-disjoint union of copies of $K_3$, so $G$ is $K_3$-divisible. Then $\overline G$ is also $K_3$-divisible and $\hat\delta(\overline G)\geq (\hat\delta^*_{K_3}+\eps)n$. So we can apply Theorem~\ref{thm:mainfrac} to find a $K_3$-decomposition of $\overline G$ which we can then use to complete $\cT_1$ to a Latin square.
Suppose now that $r\geq 4$. Observe that $G$ consists of an edge-disjoint union of cliques $H_1, \dots, H_q$ such that,
for each $1\leq i\leq q$, $H_i$ contains an edge of the form $xy$ where $x\in V_{r-1}$ and $y\in V_r$. We have $q\leq (\beta'_r-\eps)n^2$. 
We now show that we can extend $G$ to a graph of small maximum degree which can be decomposed into $q$ copies of~$K_r$.  We will do this by greedily extending each $H_i$ in turn to a copy $H_i'$ of $K_r$.
Suppose that $1\leq p\leq q$ and we have already found edge-disjoint $H_1', \dots, H_{p-1}'$. Given $v\in V(G)$, let $s(v, p-1)$ be the number of graphs in $\{H_1', \dots, H_{p-1}'\}\cup \{H_p, \dots, H_q\}$ which contain $v$. 
Suppose inductively that $s(v,p-1)\leq 2(\beta'_r-\eps^2)n$ for all $v\in V(G)$.
For each $1\leq j\leq r$, let $B_j:=\{v\in V_j: s(v,p-1)\geq 2(\beta'_r-\eps)n\}$. We have
\begin{equation}
|B_j|\leq \frac{q}{2(\beta'_r-\eps)n}\leq \frac{n}{2}.\label{eq:latin:a2}
\end{equation}
Also note that Proposition~\ref{prop:extrem} and Theorem~\ref{HRthm} together imply that
\begin{equation} \label{eq:betabound}
\beta'_r \le \frac{1}{4(r+1)}.
\end{equation} 
Let $G_{p-1}:=G\cup \bigcup_{i=1}^{p-1}(H_i'-H_i)$. Note that our inductive assumption implies that
\begin{equation}\label{eq:latin:b2}
\hat\delta(\overline G_{p-1})\geq (1-2(\beta'_r-\eps^2))n 
\stackrel{(\ref{eq:betabound})}{\ge} 1- \frac{1}{2(r+1)}.
\end{equation}
We will extend $H_p$ to a copy of $K_r$ as follows. Let $\{j_1, \dots, j_m\}=\{j:1\leq j\leq r \text{ and }V(H_p)\cap V_j=\emptyset\}$. For each $j_i$ in turn, starting with $j_1$, choose one vertex $x_{j_i}$ from the set $N_{\overline G_{p-1}}(V(H_p)\cup \{x_{j_1}, \dots, x_{j_{i-1}}\}, V_{j_i}\setminus B_{j_i})$.  This is possible since \eqref{eq:latin:a2} and \eqref{eq:latin:b2} imply 
\begin{align*}
d_{\overline G_{p-1}}(V(H_p)\cup \{x_{j_1}, \dots, x_{j_{i-1}}\}, V_{j_i}\setminus B_{j_i})\geq (\frac{1}{2}-\frac{r-1}{2(r+1)})n>0.
\end{align*}
 Let $H_p'$ be the copy of $K_r$ with vertex set $V(H_p)\cup \{x_j: 1\leq j\leq r \text{ and }V(H_p)\cap V_j=\emptyset\}$.  By construction, for every $v\in V(G)$, the number $s(v,p)$ of graphs in $\{H_1', \dots, H_{p}'\}\cup \{H_{p+1}, \dots, H_q\}$ which contain $v$ satisfies $s(v, p)\leq 2(\beta'_r-\eps^2)n$.
Continue in this way to find edge-disjoint $H_1', \dots, H_q'$ such that $s(v,q)\leq 2(\beta'_r-\eps^2)n$. Let $G_q:=\bigcup_{1\leq i\leq q}H_i'$. We have 
$\hat\delta(\overline G_q)\geq (1-2(\beta'_r-\eps^2))n \ge (\hat{\delta}^*_{K_r}+2\eps^2)n$
and, since $G_q$ is an edge-disjoint union of copies of $K_r$, we know that $\overline G_q$ is $K_r$-divisible. So we can apply Theorem~\ref{thm:mainfrac} to find a $K_r$-decomposition $\cF$ of $\overline G_q$. Note that $\cF':=\cF\cup \bigcup_{1\leq i\leq q}H_i'$ is a $K_r$-decomposition of the complete $r$-partite graph. Since $H_i\subseteq H_i'$ for each $1\leq i \leq q$, we can use $\cF'$ to complete $\cT_1, \dots, \cT_{r-2}$ to a sequence of mutually orthogonal Latin squares. This shows that $\beta_r(n)\ge (\beta'_r-\eps)n$. Since we can choose $\eps$ arbitrary small, this implies that
$\beta_r\ge \beta'_r$.
\end{proof}}

If, in addition, we know that, for each $1\leq i,j\leq n$, the entry $(i,j)$ of the grid is either filled by a symbol of every colour or it is empty, we can omit the factor $4$ in the definition of $\beta'_r$ for each $r\geq 4$. We obtain this stronger result since the graph $G$ obtained from 
$\cT_1, \dots, \cT_{r-2}$ will automatically be $K_r$-decomposable.

\section{Proof Sketch}\label{sec:sketch}

Our proof of Theorem~\ref{thm:main} builds on the proof of the main results of~\cite{mindeg}, but requires significant new ideas.
In particular, the $r$-partite setting involves a stronger notion of divisibility (the non-partite setting simply requires that $r-1$ divides the degree of each vertex of $G$ and that $\binom{r}{2}$ divides $e(G)$) and we have to work much harder to preserve it during our proof. This necessitates a delicate `balancing' argument (see Section~\ref{sec:bal}).
In addition, we use a new construction for our absorbers, which allows us to obtain the best possible version of Theorem~\ref{thm:main}.
(The construction of~\cite{mindeg} would only achieve $1-1/3(r-1)$ in place of $1-1/(r+1)$.)

The idea behind the proof is as follows.
We are assuming that we have access to a black box approximate decomposition result: given a $K_r$-divisible graph $G$ on vertex classes of size $n$ with $\hat\delta(G) \geq (\hat\delta^\eta_{K_r}+\eps) n$ we can obtain an approximate $K_r$-decomposition that leaves only $\eta n^2$ edges uncovered.
We would like to obtain an exact decomposition by `absorbing' this small remainder.
By an \defn{absorber} for a $K_r$-divisible graph $H$ we mean a graph $A_H$ such that both $A_H$ and $A_H \cup H$ have a $K_r$-decomposition.
For any fixed $H$ we can construct an absorber $A_H$.
But there are far too many possibilities for the remainder $H$ to allow us to reserve individual absorbers for each in advance.

To bridge the gap between the output of the approximate result and the capabilities of our absorbers, we use an iterative absorption approach (see also \cite{mindeg} and \cite{kko}).
Our guiding principle is that, since we have no control on the remainder if we apply the approximate decomposition result all in one go, we should apply it more carefully.
More precisely, we begin by partitioning $V(G)$ at random into a large number of parts $U^1, \ldots, U^k$.
Since $k$ is large, $G[U^1, \ldots, U^k]$ still has high minimum degree, and, since the partition is random, each $G[U^i]$ also has high minimum degree.
We first reserve a sparse and well structured subgraph $J$ of $G[U^1, \ldots, U^k]$, then we obtain an approximate decomposition of $G[U^1, \ldots, U^k] - J$ leaving a sparse remainder $H$.
We then use a small number of edges from the $G[U^i]$ to cover all edges of $H \cup J$ by copies of $K_r$.
Let $G'$ be the subgraph of $G$ consisting of those edges not yet used in the approximate decomposition.
Then all edges of $G'$ lie in some $G'[U^i]$, and each $G'[U^i]$ has high minimum degree, so we can repeat this argument on each $G'[U^i]$.
Suppose that we can iterate in this way until we obtain a partition $W_1 \cup \cdots \cup W_m$ of $V(G)$ such that each $W_i$ has size at most some constant $M$ and all edges of $G$ have been used in the approximate decomposition except for those contained entirely within some $W_i$.
Then the remainder is a vertex-disjoint union of graphs $H_1, \ldots, H_m$, with each $H_i$ contained within $W_i$.
At this point we have already achieved that the total leftover $H_1\cup \dots \cup H_m$ has only $O(n)$ edges. More importantly, the set of all possibilities for the graphs $H_i$ has size at most $2^{M^2}m = O(n)$, which is a small enough number that we are able to reserve special purpose absorbers for each of them in advance (i.e. right at the start of the proof).

The above sketch passes over one genuine difficulty. Recall that $H\subseteq G[U^1, \ldots, U^k]$ denotes the sparse remainder obtained
from the approximate decomposition, which we aim to `clean up' using a well structured graph $J$ set aside at the beginning of the proof,
i.e.~we aim to cover all edges of $H\cup J$ with copies of $K_r$ by using a few additional edges from the $G[U^i]$. So consider any vertex $v\in U_1^1$ (recall that $U^i_j=U^i\cap V_j$). In order to cover the edges in $H\cup J$ between $v$ and $U^2$, we would like to find a perfect $K_{r-1}$-matching in $N(v)\cap U^2$. However,
for this to work, the number of neighbours of $v$ inside each of $U^2_2,\dots,U^2_r$ must be the same, and the analogue must hold with $U^2$ replaced by
any of $U^3,\dots,U^k$. (This is in contrast to \cite{mindeg}, where one only needs that the number of leftover edges between $v$ and
any of the parts $U^i$ is divisible by~$r$, which is much easier to achieve.) We ensure this balancedness condition by constructing a `balancing graph' which can be used to transfer a surplus of edges or degrees from one part to another. This `balancing graph' will be the main ingredient of $J$.
Another difficulty is that whenever we apply the approximate decomposition result, we need to ensure that the graph is $K_r$-divisible.
This means that we need to `preprocess' the graph at each step of the iteration.

The rest of this paper is organised as follows.
In Section~\ref{sec:emb}, we present general purpose embedding lemmas that allow us to find a wide range of desirable structures within our graph.
In Section~\ref{sec:absorb}, we detail the construction of our absorbers.
In Section~\ref{sec:partn}, we prove some basic properties of random subgraphs and partitions.
In Section~\ref{sec:degred}, we show how we can assume that our approximate decomposition result produces a remainder with low maximum degree rather than simply a small number of edges.
In Section~\ref{sec:pseud}, we clean up the edges in the remainder using a few additional edges from inside each part of the current partition.
However, we assume in this section that our remainder is balanced in the sense described above.
In Section~\ref{sec:bal}, we describe the balancing operation which ensures that we can make this assumption.
Finally, in Section~\ref{sec:proof} we put everything together to prove Theorem~\ref{thm:main}.


\section{Embedding lemmas}\label{sec:emb}

Let $G$ be an $r$-partite graph on $(V_1,\dots, V_r)$ and let $\cP=\{U^1, U^2, \dots, U^k\}$ be a partition of $V(G)$. Recall that $U^i_j:=U^i\cap V_j$ for each $1\leq i \leq k$ and each $1\leq j\leq r$. We say that a graph (or multigraph) $H$ is \emph{$\cP$-labelled} if:
\begin{enumerate}[(a)]
\item every vertex of $H$ is labelled by one of: $\{v\}$ for some $v\in V(G)$; $U^i_j$ for some $1\leq i\leq k$, $1\leq j\leq r$ or $V_j$ for some $1\leq j\leq r$;\label{label:a}
\item the vertices labelled by singletons (called \emph{root vertices}) form an independent set in $H$, and each $v\in V(G)$ appears as a label $\{v\}$ at most once;\label{label:b}
\item for each $1\leq j\leq r$, the set of vertices $v\in V(H)$ such that $v$ is labelled $L$ for some $L\subseteq V_j$ forms an independent set in $H$.\label{label:c}%
\COMMENT{implies $H$ is $r$-partite (and gives vtx classes)}
\end{enumerate}
Any vertex which is not a root vertex is called a \emph{free vertex}. Throughout this paper, we will always have the situation that all the sets $U^i_j$ are large, so there will be no ambiguity between the labels of the form $\{v\}$ and $U^i_j$ in (b).

Let $H$ be a $\cP$-labelled graph and let $H'$ be a copy of $H$ in $G$. We say that $H'$ is \emph{compatible with its labelling} if each vertex of $H$ gets mapped to a vertex in its label.

Given a graph $H$ and $U \subseteq V(H)$ with $e (H[U]) = 0$, we define the \emph{degeneracy of $H$ rooted at $U$} to be the least $d$ for which there is an ordering $v_1, \ldots, v_b$ of the vertices of $H$ such that
\begin{itemize}
	\item there is an $a$ such that $U  = \{v_1, \ldots, v_a \}$ (the ordering of $U$ is unimportant);
	\item for $a < j \leq b$, $v_j$ is adjacent to at most $d$ of the $v_i$ with $1\leq i < j$.
\end{itemize}
The \emph{degeneracy of a $\cP$-labelled graph $H$} is the degeneracy of $H$ rooted at $U$, where $U$ is the set of root vertices of $H$.

In the proof of Lemma~\ref{lem:mover2}, we use the following special case of Lemma~5.1%
\COMMENT{not a `\textbackslash ref' - this is number in \cite{mindeg}}
 from \cite{mindeg} to find copies of labelled graphs inside a graph $G$, provided their degeneracy is small. Moreover, this lemma allows us to assume that the subgraph of $G$ used to embed these graphs has low maximum degree.

\begin{lemma} \label{lma:finding}
Let $1/ n \ll \eta \ll \eps, 1/d, 1/b \leq 1$ and let $G$ be a graph on $n$ vertices.
Suppose that:
\begin{enumerate}[\rm(i)]
\item  for each $S \subseteq V(G)$ with $|S| \leq d$, $d_G(S,V(G)) \geq \eps n$.
\end{enumerate}
Let $m \leq \eta n^2$ and let $H_1, \ldots, H_m$ be labelled graphs such that, for every $1\leq i\leq m$, every vertex of $H_i$ is labelled $\{v\}$ for some $v\in V(G)$ or labelled by $V(G)$ and that property \eqref{label:b} above holds for $H_i$. Moreover, suppose that:
\begin{enumerate}[\rm(i)]
\setcounter{enumi}{1}
	\item for each $1 \leq i \leq m$, $|H_i| \leq b$;
	\item for each $1\leq i \leq m$, the degeneracy of $H_i$ (rooted at the set of vertices labelled by singletons) is at most $d$;
	\item for each $v\in V(G)$, there are at most $\eta n$ graphs $H_i$ with some vertex labelled $\{v\}$.
\end{enumerate} 
Then there exist edge-disjoint embeddings $\phi(H_1), \ldots, \phi(H_m)$ of $H_1, \ldots, H_m$ compatible with their labellings such that the subgraph $H := \bigcup_{i=1}^m \phi(H_i)$ of $G$ satisfies $\Delta(H) \leq \eps n$.
\end{lemma}

We will also use the following partite version of the lemma to find copies of $\cP$-labelled graphs in an $r$-partite graph~$G$. We omit the proof since it is very similar to the proof of Lemma~5.1%
\COMMENT{not a `\textbackslash ref' - this is number in \cite{mindeg}}
 in \cite{mindeg}.
(See \cite[Lemma 4.5.2]{thesis} for a complete proof.)

\begin{lemma}\label{lem:emb}
Let $1/n \ll \eta \ll \eps, 1/d, 1/b, 1/k, 1/r\leq 1$ and let $G$ be an $r$-partite graph on $(V_1, \dots, V_r)$ where $|V_1|=\dots=|V_r|=n$. Let $\cP=\{U^1, \dots, U^k\}$ be a $k$-partition of $V(G)$. Suppose that:
\begin{enumerate}[\rm(i)]
\item for each $1\leq i \leq k$ and each $1\leq j \leq r$, if $S\subseteq V(G)\setminus V_j$ with $|S|\leq d$ then $d_G(S, U^i_j)\geq \eps |U^i_j|$.\label{lem:emb:1}
\end{enumerate}
Let $m\leq \eta n^2$ and let $H_1, \dots, H_m$ be $\cP$-labelled graphs such that the following hold:
\begin{enumerate}[\rm(i)]
\setcounter{enumi}{1}
	\item for each $1\leq i \leq m$, $|H_i|\leq b$;\label{lem:emb:2}
	\item for each $1\leq i \leq m$, the degeneracy of $H_i$ is at most $d$;\label{lem:emb:3}
	\item for each $v\in V(G)$, there are at most $\eta n$ graphs $H_i$ with some vertex labelled $\{v\}$.\label{lem:emb:4}
\end{enumerate}
Then there exist edge-disjoint embeddings $\phi(H_1), \dots, \phi(H_m)$ of $H_1, \dots, H_m$ in $G$ which are compatible with their labellings such that $H:=\bigcup_{1\leq i\leq m}\phi(H_i)$ satisfies $\Delta(H)\leq \eps n$. \qed
\end{lemma}
\COMMENT{proof - see appendix}


\section{Absorbers}\label{sec:absorb}

Let $H$ be any $r$-partite graph on the vertex set $V=(V_1, \dots , V_r)$. An \emph{absorber for $H$} is a graph $A$ such that both $A$ and $A \cup H$ have $K_r$-decompositions.

Our aim is to find an absorber for each small $K_r$-divisible graph $H$ on $V$.
The construction develops ideas in \cite{mindeg}. In particular, we will build the absorber in stages using transformers, introduced below, to move between $K_r$-divisible graphs.

Let $H$ and $H'$ be vertex-disjoint graphs.
An \emph{$(H,H')_{r}$-transformer} is a graph $T$ which is edge-disjoint from $H$ and $H'$ and is such that both $T\cup H$ and $T\cup H'$ have $K_r$-decompositions.
Note that if $H'$ has a $K_r$-decomposition, then $T\cup H'$ is an absorber for $H$. So the idea is that we can use a transformer to transform a given $H$ into a new graph $H'$, then into $H''$ and so on, until finally we arrive at a graph which has a $K_r$-decomposition.

Let $V=(V_1, \dots, V_r)$. 
Throughout this section, given two $r$-partite graphs $H$ and $H'$ on $V$, we say that $H'$ is a \emph{partition-respecting copy} of $H$
if there is an isomorphism $f: H\rightarrow H'$ such that $f(v)\in V_j$ for every vertex $v\in V(H)\cap V_j$.

Given $r$-partite graphs $H$ and $H'$ on $V$, we say that $H'$ is \emph{obtained from $H$ by identifying vertices}%
\COMMENT{Basically - if we can identify vertices in the same vertex class and no vertex is allowed to jump into a different vertex class.}
if there exists a sequence of $r$-partite graphs $H_0, \dots, H_{s}$ on $V$ such that $H_0=H$, $H_s=H'$ and the following holds.
For each $0\leq i < s$, there exists $1\leq j_i\leq r$ and vertices $x_i,y_i \in V(H_i)\cap V_{j_i}$ satisfying the following:
\begin{enumerate}[\rm(i)]
		\item $N_{H_i}(x_i) \cap N_{H_i}(y_i) = \emptyset$.\label{id:1}
		\item $H_{i+1}$ is the graph which has vertex set $V(H_i)\setminus \{y_i\}$ and edge set $E(H_i\setminus \{y_i\})\cup \{vx_i: vy_i\in E(H_i)\}$ (i.e., $H_{i+1}$ is obtained from $H_i$ by identifying the vertices $x_i$ and $y_i$).
\end{enumerate}
Condition~\eqref{id:1} ensures that the identifications do not produce multiple edges.
Note that if $H$ and $H'$ are $r$-partite graphs on $V$ and $H'$ is a partition-respecting copy of a graph obtained from $H$ by identifying vertices then there exists a graph homomorphism $\phi: H \rightarrow  H'$ that is edge-bijective and maps vertices in $V_j$ to vertices in $V_j$ for each $1\leq j\leq r$.

In the following lemma, we find a transformer between a pair of $K_r$-divisible graphs $H$ and $H'$ whenever $H'$ can be obtained from $H$ by identifying vertices.

\begin{lemma}\label{lem:transformer}
Let $r\geq 3$ and $1/n \ll \eta \ll 1/s\ll \eps, 1/b, 1/r\leq 1$. Let $G$ be an $r$-partite graph on $V=(V_1, \dots, V_r)$ with $|V_1|=\dots=|V_r|=n$. Suppose that $\hat\delta(G)\geq (1-1/(r+1)+\eps)n$. Let $H$ and $H'$ be vertex-disjoint $K_r$-divisible graphs on $V$ with $|H|\leq b$. Suppose further that $H'$ is a partition-respecting copy of a graph obtained from $H$ by identifying vertices. Let $B\subseteq V$ be a set of at most $\eta n$ vertices.
Then $G$ contains an $(H, H')_r$-transformer $T$ such that $V(T)\cap B\subseteq V(H\cup H')$ and $|T|\leq s^2$.
\end{lemma}

In our proof of Lemma~\ref{lem:transformer}, we will use the following multipartite asymptotic version of the Hajnal--Szemer\'edi theorem.

\begin{theorem}[\cite{kemy} and \cite{loma}]\label{thm:hajszem}
Let $r\geq 2$ and let $1/n\ll \eps,1/r$. Suppose that $G$ is an $r$-partite graph on $(V_1, \dots, V_r)$ with $|V_1|=\dots=|V_r|=n$ and $\hat{\delta}(G)\geq (1-1/r+\eps)n$. Then $G$ contains a perfect $K_r$-matching. 
\end{theorem}

\begin{proofof}{Lemma~\ref{lem:transformer}}
Let $\phi: H \to H'$ be a graph homomorphism from $H$ to $H'$ that is edge-bijective and maps vertices in $V_j$ to $V_j$ for each $1\leq j\leq r$.

Let $T$ be any graph defined as follows:
\begin{enumerate}[(a)]
\item For each $xy \in E(H)$, $Z^{xy}:= \{ z_j^{xy}: 1\leq j\leq r \text{ and } x,y\notin V_j\}$ is a set of $r-2$ vertices. For each $x\in V(H)$, let $Z^x:=\bigcup_{y\in N_H(x)}Z^{xy}$.\label{transf0}
\item 
For each $x\in V(H)$, $S^{x}$ is a set of $(r-1)s$ vertices.
\item For all distinct $e,e'\in E(H)$ and all distinct $x,x'\in V(H)$, the sets $Z^e$, $Z^{e'}$, $S^x$, $S^{x'}$ and $V(H\cup H')$ are disjoint.
\item $V(T) := V(H) \cup V(H')  \cup \bigcup_{e \in E(H)} Z^e\cup \bigcup_{x\in V(H)}S^x$.\label{transf1}
\item $E_H := \{ x z: x \in V(H)$ and $z\in Z^x\}$.\label{transf2}
\item $E_{H'} := \{ \phi(x) z  : x \in V(H)$ and $z\in Z^x\}$.\label{transf4}
\item $E_Z := \{ wz : e \in E(H)\text{ and } w,z\in Z^{e}\}$.\label{transf3}
\item $E_S:= \{ x v : x\in V(H) \text{ and } v\in S^x\}$.\label{transf5}
\item $E_S':= \{ \phi(x) v  : x\in V(H) \text{ and } v\in S^x\}$.\label{transf6}
\item For each $x\in V(H)$, $F^x_1$ is a perfect $K_{r-1}$-matching on $S^x\cup Z^x$.\label{transf7}
\item For each $x\in V(H)$, $F^x_2$ is a perfect $K_{r-1}$-matching on $S^x$.\label{transf8}
\item For each $x\in V(H)$, $F^x_1$ and $F^x_2$ are edge-disjoint.\label{transf9}
\item For each $x\in V(H)$, $Z^x$ is independent in $F^x_1$.\label{transf10}
\item $E(T):=E_H\cup E_{H'}\cup E_Z\cup E_S\cup E_S'\cup \bigcup_{x\in V(H)} E(F^x_1\cup F^x_2)$.\label{transf12}
\end{enumerate}
Then
$$|T| =|H| + |H'|+\sum_{e\in E(H)}|Z^e| +\sum_{x\in V(H)}|S^x|= |H| + |H'| + (r-2)e(H)+(r-1)s|H|\leq s^2.$$
Let $T_1$ be the subgraph of $T$ with edge set $E_H\cup E_{H'} \cup E_Z$ and let $T_2:=T-T_1$. So $E(T_2)=E_S\cup E_S'\cup\bigcup_{x\in V(H)} E(F^x_1\cup F^x_2)$. In what follows, we will often identify certain subsets of the edge set of $T$ with the subgraphs of $T$ consisting of these edges. For example, we will write $E_S[\{x\}, S^x]$ for the subgraph of $T$ consisting of all the edges in $E_S$ between $x$ and $S^x$.
Note that there are several possibilities for $T$ as we have several choices for the perfect $K_{r-1}$-matchings in \eqref{transf7} and \eqref{transf8}.

Lemma~\ref{lem:transformer} will follow from Claims~1~and~2 below.

\begin{figure}[t]
\includegraphics[scale=0.8]{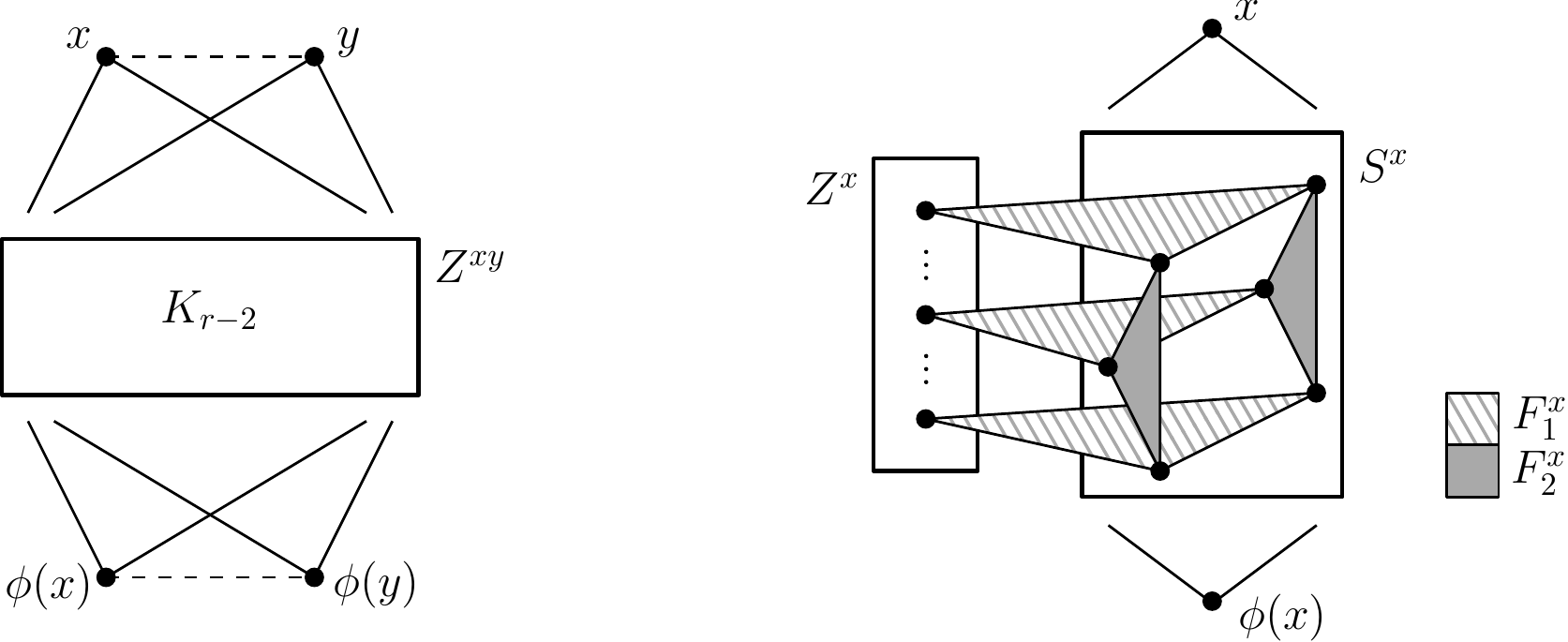}
\caption{Left: Subgraph of $T_1$ associated with $xy\in E(H)$. Right: Subgraph of $T_2$ associated with $x\in V(H)$ in the case when $r=4$.}
\end{figure}

\medskip
\claim{Claim 1}{If $T$ satisfies \eqref{transf0}--\eqref{transf12}, then $T$ is an $(H, H')_r$-transformer.}

\pclaim{Proof of Claim 1}
Note that $H \cup E_H \cup E_Z$ can be decomposed into $e(H)$ copies of $K_r$, where each copy of $K_r$ has vertex set $\{x, y\} \cup Z^{xy}$ for some edge $xy \in E(H)$.
Similarly, $H' \cup E_{H'} \cup E_Z$ can be decomposed into $e(H)$ copies of $K_r$.%
\COMMENT{where each $K_r$ has vertex set $\{\phi(x), \phi(y) \} \cup Z^{xy}$ for some edge $xy \in E(H)$.}

For each $x \in V(H)$, note that $(E_{H'}\cup E_S')[\{\phi(x)\}, S^x\cup Z^x] \cup F^x_1$ and $E_S[\{x\}, S^x]\cup F^x_2$ are edge-disjoint and have $K_r$-decompositions.
Since
$$T_2 \cup E_{H'}=\bigcup_{x\in V(H)}\big((E_{H'}\cup E_S')[\{\phi(x)\}, S^x\cup Z^x] \cup F^x_1\big)
\cup \bigcup_{x\in V(H)}\big(E_S[\{x\}, S^x]\cup F^x_2\big),$$
it follows that $T_2 \cup E_{H'}$ has a $K_r$-decomposition.
Similarly, for each $x \in V(H)$, $(E_H\cup E_S)[\{x\}, S^x\cup Z^x] \cup F^x_1$ and $E_S'[\{\phi(x)\}, S^x]\cup F^x_2$ are edge-disjoint and have $K_r$-decompositions, so $T_2 \cup E_H$ has a $K_r$-decomposition.

To summarise, $H \cup E_H \cup E_Z$, $H' \cup E_{H'} \cup E_Z$, $T_2\cup E_H$ and $T_2\cup E_{H'}$ all have $K_r$-decompositions. Therefore, 
$T \cup H = (H \cup E_H \cup E_Z) \cup (T_2\cup E_{H'})$ has a $K_r$-decomposition, as does $T \cup H' = (H' \cup E_{H'} \cup E_Z) \cup (T_2\cup E_H)$.
Hence $T$ is an $(H,H')_r$-transformer.

\medskip
\claim{Claim 2}{$G$ contains a graph $T$ satisfying \eqref{transf0}--\eqref{transf12} such that $V(T)\cap B\subseteq V(H\cup H')$.}

\pclaim{Proof of Claim 2}
We begin by finding a copy of $T_1$ in $G$. It will be useful to note that, for any graph $T$ which satisfies \eqref{transf0}--\eqref{transf12}, $T_1$ is $r$-partite with vertex classes $(V(H\cup H')\cap V_j)\cup \{z^{xy}_j: xy\in E(H) \text{ and } x,y\notin V_j\}$ where $1\leq j\leq r$.
Also, $T[V(H\cup H')]$ is empty and every vertex $z \in V(T_1) \setminus V(H \cup H')$ satisfies
\begin{align}
	d_{T_1}(z) =2 + (r-3)+2 = r+1. \label{eqn:regularT1}
\end{align}
So $T_1$ has degeneracy $r+1$ rooted at $V(H\cup H')$. Since $\hat\delta(G)\geq (1-1/(r+1)+\eps/2)n+|B|$, we can find a
copy of $T_1$ in $G$ such that $V(T_1)\cap B\subseteq V(H\cup H')$.

We now show that, after fixing $T_1$, we can extend $T_1$ to $T$ by finding a copy of $T_2$. Consider any ordering $x_1, \dots, x_{|H|}$ on the vertices of $H$. Suppose we have already chosen $S^{x_1}, \dots, S^{x_{q-1}}$, $F^{x_1}_1, \dots, F^{x_{q-1}}_1$ and $F^{x_1}_2, \dots, F^{x_{q-1}}_2$ and we are currently embedding $S^{x_{q}}$.
Let $B':=B\cup V(T_1)\cup \bigcup_{i=1}^{q-1}S^{x_i}$; that is, $B'$ is the set of vertices that are unavailable for $S^{x_{q}}$, either because they have been used previously or they lie $B$. Note that $|B'|\leq |T|+|B|\leq 2\eta n$.
We will choose suitable vertices for  $S^{x_{q}}$ in the common neighbourhood of $x_q$ and $\phi(x_q)$.

To simplify notation, we write $x:=x_q$ and assume that $x\in V_1$ (the argument is identical in the other cases). Choose a set $V'\subseteq (N_G(x)\cap N_G(\phi(x)))\setminus B'$ which is maximal subject to $|V_2'|=\dots =|V_r'|$ (recall that $V_j'=V'\cap V_j$). Note that for each $2\leq j\leq r$, we have
$$|V_{j}'| \geq (1-1/(r+1)+\eps)n-(1/(r+1)-\eps)n-|B'|\geq (1-2/(r+1))n.$$
Let $n':=|V_2'|$. For every $2\leq j\leq r$ and every $v\in V(G)\setminus V_j$, we have
\begin{equation}
d_G(v, V_j')\geq n'-(1/(r+1)-\eps) n\geq (1-1/(r-1)+\eps)n'.\label{eq:V'deg}
\end{equation}

Roughly speaking, we will choose $S^{x}$ as a random subset of $V'$. For each $2\leq j\leq r$, choose each vertex of $V_{j}'$ independently with probability $p:=(1+\eps/8)s/n'$ and let $S_{j}'$ be the set of chosen vertices. Note that, for each $j$, $\ex(|S_{j}'|)=n'p=(1+\eps/8)s$. We can apply Lemma~\ref{lem:chernoff2} to see that
\begin{align}
\pr(||S_{j}'|-(1+\eps/8)s|\geq \eps s/8)&\leq \pr(||S_{j}'|-(1+\eps/8)s|\geq \eps\ex(|S_{j}'|)/10)\nonumber\\
&\leq 2e^{-\eps^2s/300}\leq 1/4(r-1).\label{eq:prob1}
\end{align}

Given a vertex $v\in V(G)$ and $2\leq j\leq r$ such that $v\notin V_{j}$, note that
$$\ex(d_G(v, S'_{j})) \stackrel{(\ref{eq:V'deg})}{\geq} (1-1/(r-1)+\eps)n'p> (1-1/(r-1)+\eps)s.$$
We will say that a vertex $v\in V(G)$ is \emph{bad} if there exists $2\leq j\leq r$ such that $v\notin V_{j}$ and $d_G(v, S'_{j})<(1-1/(r-1)+3\eps/4)s$, that is, the degree of $v$ in $S'_{j}$ is lower than expected.
We can again apply Lemma~\ref{lem:chernoff2} to see that
\begin{align*}
\pr(d_G(v, S'_{j})\leq (1-1/(r-1)+3\eps/4)s)&\leq \pr(|d_G(v, S'_{j})-\ex(d_G(v, S'_{j}))|\geq \eps s/4)\\
&\leq \pr(|d_G(v, S'_{j})-\ex(d_G(v, S'_{j}))|\geq \eps \ex(d_G(v, S'_{j}))/10)\\
&\leq 2e^{-\eps^2 s/600}.
\end{align*}
So $\pr(v\text{ is bad})\leq 2(r-1)e^{-\eps^2 s/600}\leq e^{-s^{1/2}}$.
Let $S':=\bigcup_{j=2}^rS'_j$. We say that the set $S'$ is \emph{bad} if $S'\cup Z^x$ contains a bad vertex. We have
\begin{align}
\pr(S'\text{ is bad})&\leq \sum_{v\in V'}\pr(v\in S' \text{ and } v\text{ is bad})+\sum_{v \in Z^x}\pr(v\text{ is bad})\nonumber\\
&=\sum_{v\in V'}\pr(v\in S')\pr(v\text{ is bad})+\sum_{v \in Z^x}\pr(v\text{ is bad})\nonumber\\
&\leq (n'p+(b-1)(r-2))e^{-s^{1/2}}\leq 2se^{-s^{1/2}}\leq 1/4.\label{eq:prob2}
\end{align}
\COMMENT{used that $|Z^x|\leq \Delta(H)(r-2)\leq (b-1)(r-2)$.}

We apply \eqref{eq:prob1} and \eqref{eq:prob2} to see that with probability at least $1/2$, the set $S'$ chosen in this way is not bad and, for each $2\leq j\leq r$, we have $s\le |S_{j}'|\le (1+\eps/4)s$. Choose one such set $S'$. 
Delete at most $\eps s/4$ vertices  from each $S_{j}'$ to obtain sets $S^x_{j}$ satisfying $|S^x_2|=\dots =|S^x_r|= s$. Let $S^x:=\bigcup_{j=2}^r S^x_{j}$.
Since $S'$ was not bad, for each $2\leq j\leq r$ and each vertex $v\in (S^x\cup Z^x)\setminus V_j$,
\begin{equation}\label{eq:Sdeg}
d_G(v, S^x_j)\geq (1-1/(r-1)+3\eps/4)s-\eps s/4=(1-1/(r-1)+\eps/2)s.
\end{equation}

We now show that we can find $F^x_1$ and $F^x_2$ satisfying \eqref{transf7}--\eqref{transf10}. Let $G^x:=G[Z^x\cup S^x]-G[Z^x]$. Note that $G^x$ is a balanced $(r-1)$-partite graph with vertex classes of size $n_x$ where $s\le n_x\leq s+(r-2)(b-1)/(r-1)<s+b$.
Using \eqref{eq:Sdeg}, we see that
$$\hat\delta(G^x)\geq (1-1/(r-1)+\eps/2)s\geq (1-1/(r-1)+\eps/3)n_x.$$
So, using Theorem~\ref{thm:hajszem}, we can find a perfect $K_{r-1}$-matching $F^x_1$ in $G^x$.
Finally, let $G':=G-F^x_1$ and use \eqref{eq:Sdeg} to see that
$$\hat\delta(G'[S^x])\geq (1-1/(r-1)+\eps/3)s.$$
So we can again apply Theorem~\ref{thm:hajszem}, to find a perfect $K_{r-1}$-matching $F^x_2$ in $G'[S^x]$.
In this way, we find a copy of $T$ satisfying \eqref{transf0}--\eqref{transf12} such that $V(T)\cap B\subseteq V(H\cup H')$.

This completes the proof of Lemma~\ref{lem:transformer}.
\end{proofof}

We now construct our absorber by combining several suitable transformers.

\newcommand{\Hatt}{H_{\text{att}}}
\newcommand{\Hexp}{H_{\text{exp}}}

Let $H$ be an $r$-partite multigraph on $(\tilde V_1,\dots,\tilde V_r)$ with $\tilde V_i \subset V_i$ for each $1 \leq i \leq r$, and let $xy\in E(H)$. A \emph{$K_r$-expansion of $xy$} is defined as follows.
Consider a copy $F_{xy}$ of $K_r$ on vertex set $\{u_1,\dots,u_r\}$ such that $u_j\in V_j\setminus V(H)$ for all $1\le j\le r$.
Let $j_1,j_2$ be such that $x\in V_{j_1}$ and $y\in V_{j_2}$.
Delete $xy$ from $H$ and $u_{j_1}u_{j_2}$ from $F_{xy}$ and add edges joining $x$ to $u_{j_2}$ and joining $y$ to $u_{j_1}$.
Let $\Hexp$ be the graph obtained by $K_r$-expanding every edge of $H$, where the $F_{xy}$ are chosen to be vertex-disjoint for
different edges $xy\in E(H)$.

\begin{fact} \label{expand}
Suppose that the graph $H'$ is obtained from a graph $H$ by $K_r$-expanding the edge $xy \in E(H)$ as above.
Then the graph obtained from $H'$ by identifying $x$ and $u_{j_1}$ is $H$ with a copy of $K_r$ attached to $x$.\qedhere
\end{fact}

Let $h\in \N$. We define a graph $M_h$ as follows. Take a copy of $K_r$ on $V$ (consisting of one vertex in each $V_j$) and replace each edge by $h$ multiedges. Let $M$ denote the resulting multigraph. Let $M_h:=M_\text{exp}$ be the graph obtained by $K_r$-expanding every edge of $M$.
We have $|M_h|=r+hr\binom{r}{2}$. Note that $M_h$ has degeneracy $r-1$. To see this, list all vertices in $V(M)$ (in any order) followed by the vertices in $V(M_h\setminus M)$ (in any order).

We will now apply Lemma~\ref{lem:transformer} twice in order to find an $(H, M_h)_r$-transformer in $G$.

\begin{lemma}\label{lem:embtransf2}
Let $r\geq 3$ and $1/n \ll \eta \ll 1/s\ll \eps, 1/b, 1/r\leq 1$. Let $G$ be an $r$-partite graph on $V=(V_1, \dots, V_r)$ with $|V_1|=\dots=|V_r|=n$. Suppose that $\hat\delta(G)\geq (1-1/(r+1)+\eps)n$. Let $H$ be a $K_r$-divisible graph on $V$ with $|H|\leq b$. Let $h:=e(H)/\binom{r}{2}$.
Let $M'_h$ be a partition-respecting copy of $M_h$ on $V$ which is vertex-disjoint from $H$. Let $B\subseteq V$ be a set of at most $\eta n$ vertices.
Then $G$ contains an $(H, M'_h)_r$-transformer $T$ such that $V(T)\cap B\subseteq V(H\cup M'_h)$ and $|T|\leq 3s^2$.
\end{lemma}

\begin{proof}
We construct a graph $\Hatt$ as follows. Start with the graph $H$. For each edge of $H$, arbitrarily choose one of it endpoints $x$ and attach a copy of $K_r$ (found in
$G\setminus ((V(H\cup M'_h)\cup B)\setminus \{x\})$) to $x$. The copies of $K_r$ should be chosen to be vertex-disjoint outside $V(H)$. Write $\Hatt$ for the resulting graph. 
Let $\Hexp'$ be a partition-respecting copy of $\Hexp$ in $G\setminus (V(\Hatt\cup M'_h)\cup B)$.
Note that we are able to find these graphs since both have degeneracy $r-1$ and 
$\hat\delta(G)\geq (1-1/(r+1)+\eps)n$.

By Fact~\ref{expand}, $\Hatt$ is a partition-respecting copy of a graph obtained from $\Hexp'$ by identifying vertices, and this
is also the case for~$M'_h$. To see the latter, for each $1\leq j\leq r$, identify all vertices of $\Hexp'$ lying in~$V_j$.
(We are able to do this since these vertices are non-adjacent with disjoint neighbourhoods.)

Apply Lemma~\ref{lem:transformer} to find an $(\Hexp',\Hatt)_r$-transformer $T'$ in $G-M'_h$ such that $V(T')\cap B \subseteq V(H)$ and $|T'|\leq s^2$.
Then apply Lemma~\ref{lem:transformer} again to find an $(\Hexp',M'_h)_r$-transformer $T''$ in $G-(\Hatt\cup T')$ such that $V(T'')\cap B \subseteq V(M'_h)$ and $|T''|\leq s^2$.

Let $T:=T'\cup T''\cup \Hexp' \cup (\Hatt-H)$. Then $T$ is edge-disjoint from $H\cup M'_h$. Note that
\begin{align*}
&T\cup H=(T'\cup \Hatt) \cup (T''\cup \Hexp')\hspace{6pt}\text{ and}\\
&T\cup M'_h=(T'\cup \Hexp')\cup (T''\cup M'_h) \cup (\Hatt-H),
\end{align*}
both of which have $K_r$-decompositions. Therefore $T$ is an $(H, M'_h)_r$-transformer. Moreover, $|T|\leq 3s^2$.
Finally, observe that $V(T)\cap B=V(T'\cup T''\cup \Hatt)\cap B\subseteq V(H\cup M'_h)$.
\end{proof}

We now have all of the necessary tools to find an absorber for $H$ in $G$.

\begin{lemma}\label{lem:absorber}
Let $r\geq 3$ and let $1/n \ll \eta \ll 1/s\ll \eps, 1/b, 1/r\leq 1$. Let $G$ be an $r$-partite graph on $V=(V_1, \dots, V_r)$ with $|V_1|=\dots=|V_r|=n$. Suppose that $\hat\delta(G)\geq (1-1/(r+1)+\eps)n$. Let $H$ be a $K_r$-divisible graph on $V$ with $|H|\leq b$. Let $B\subseteq V$ be a set of at most $\eta n$ vertices.
Then $G$ contains an absorber $A$ for $H$ such that $V(A)\cap B\subseteq V(H)$
and $|A|\leq s^3$.
\end{lemma}

\begin{proof}
Let $h:=e(H)/\binom{r}{2}$. Let $G':=G\setminus (V(H)\cup B)$. Write $hK_r$ for the graph consisting of $h$ vertex-disjoint copies of $K_r$. Since $\hat\delta(G')\geq (1-1/(r+1)+\eps/2)n$, we can choose vertex-disjoint (partition-respecting) copies of $M_h$ and $hK_r$ in $G'$
(and call these $M_h$ and $hK_r$ again).
Use Lemma~\ref{lem:embtransf2} to find an $(H,M_h)_r$-transformer $T'$ 
in $G-hK_r$ such that $V(T')\cap B \subseteq V(H)$ and $|T'|\leq 3s^2$.  Apply Lemma~\ref{lem:embtransf2} again to find an $(h K_r, M_h )_r$-transformer $T''$ in $G-(H\cup T')$ which avoids $B$ and satisfies
$|T''|\leq 3s^2$.
It is easy to see that $T:=T'\cup T'' \cup M_h$ is an $(H, hK_r)_r$-transformer.%
\COMMENT{$T$ is edge-disjoint from $H\cup hK_r$. $T\cup H=(T'\cup H)\cup (T'' \cup M_h)$ and $T\cup hK_r=(T'\cup M_h)\cup (T''\cup hK_r)$ have $K_r$-decompositions.}

Let $A:=T \cup h K_r$. Note that both $A$ and $A\cup H=(T\cup H)\cup h K_r$ have $K_r$-decompositions. So $A$ is an absorber for $H$. 
Moreover, $V(A)\cap B\subseteq V(T')\cap B\subseteq V(H)$ and $|A| \leq s^3$.
\end{proof}

\subsection{Absorbing sets}

Let $\cH$ be a collection of graphs on the vertex set $V=(V_1, \dots, V_r)$. We say that $\cA$ is an \emph{absorbing set for $\cH$} if $\cA$ is a collection of edge-disjoint graphs and, for every $H\in \cH$ and every $K_r$-divisible subgraph $H'\subseteq H$, there is a distinct $A_{H'}\in \cA$ such that $A_{H'}$ is an absorber for $H'$.

\begin{lemma}\label{lem:absorbset}
Let $r\geq 3$ and $1/n \ll \eta \ll \eps, 1/b, 1/r\leq 1$. Let $G$ be an $r$-partite graph on $V=(V_1, \dots, V_r)$ with $|V_1|=\dots=|V_r|=n$. Suppose that $\hat\delta(G)\geq (1-1/(r+1)+\eps)n$.
Let $m\leq \eta n^2$ and let $\cH$ be a collection of $m$ edge-disjoint graphs on $V=(V_1, \dots, V_r)$ such that each vertex $v\in V$ appears in at most $\eta n$ of the elements of $\cH$ and $|H|\leq b$ for each $H\in \cH$. 
Then $G$ contains an absorbing set $\cA$ for $\cH$ such that $\Delta(\bigcup \cA)\leq  \eps n$.
\end{lemma}

We repeatedly use Lemma~\ref{lem:absorber} and aim to avoid any vertices which have been used too often.

\begin{proof}
Enumerate the $K_r$-divisible subgraphs of all $H\in \cH$ as $H_1,\dots, H_{m'}$. Note that each $H\in \cH$ can have at most $2^{e(H)}\leq 2^{\binom{b}{2}}$ $K_r$-divisible subgraphs so $m'\leq 2^{\binom{b}{2}}\eta n^2$.
For each $v\in V(G)$ and each $0\leq j\leq m'$, let $s(v,j)$ be the number of indices $1\leq i\leq j$ such that $v\in V(H_i)$. Note that $s(v,j)\leq 2^{\binom{b}{2}}\eta n$.

Let $s\in \N$ be such that $\eta \ll 1/s\ll \eps, 1/b,1/r$.
Suppose that we have already found absorbers $A_1, \dots, A_{j-1}$ for $H_1, \dots, H_{j-1}$ respectively such that
$|A_i|\leq s^3$, for all $1\leq i\leq j-1$,
and, for every $v\in V(G)$,
\begin{equation}\label{eq:emb1'}
d_{G_{j-1}}(v)\leq \eta^{1/2} n+(s(v,j-1)+1)s^3,
\end{equation}
where $G_{j-1}:=\bigcup_{1\leq i\leq j-1} A_i$. We show that we can find an absorber $A_j$ for $H_j$ in $G-G_{j-1}$ which satisfies \eqref{eq:emb1'} with $j$ replacing $j-1$.

Let $B:= \{v\in V(G): d_{G_{j-1}}(v)\geq \eta^{1/2}n\}$. We have
$$|B|\leq \frac{2e(G_{j-1})}{\eta^{1/2}n}\leq \frac{2m'\binom{s^3}{2}}{\eta^{1/2}n}\leq \frac{2^{\binom{b}{2}+1}\eta n^2s^6}{\eta^{1/2}n}\leq \eta^{1/3}n.$$
We have
\begin{align*}
\hat\delta(G-G_{j-1})&\stackrel{\mathclap{\eqref{eq:emb1'}}}{\geq} (1-1/(r+1)+\eps) n- \eta^{1/2} n-(s(v,j-1)+1)s^3\\
&\geq (1-1/(r+1)+\eps) n- \eta^{1/2} n-(2^{\binom{b}{2}}\eta n+1)s^3
> (1-1/(r+1)+\eps/2)n.
\end{align*}
So we can apply Lemma~\ref{lem:absorber} (with $\eps/2$, $\eta^{1/3}$, $G-G_{j-1}$ and $H_j$ playing the roles of $\eps$, $\eta$, $G$ and $H$) to find an absorber $A_j$ for $H_j$ in $G-G_{j-1}$ such that $V(A_j)\cap B\subseteq V(H_j)$ and $|A_j|\leq s^3$.

We now check that \eqref{eq:emb1'} holds with $j$ replacing $j-1$. If $v\in V(G)\setminus B$, this is clear. Suppose then that $v\in B$. If $v\in V(A_j)$, then $v\in V(H_j)$ and $s(v,j)=s(v,j-1)+1$. So in all cases,
$$d_{G_{j}}(v)\leq \eta^{1/2} n+(s(v,j)+1)s^3,$$
as required.

Continue in this way until we have found an absorber $A_i$ for each $H_i$. Then $\cA:=\{A_i: 1\leq i\leq m'\}$ is an absorbing set. Using \eqref{eq:emb1'}, $$\Delta\big(\bigcup \cA\big)=\Delta(G_{m'})\leq \eta^{1/2}n+(2^{\binom{b}{2}}\eta n+1)s^3\leq \eps n,$$
as required.
\end{proof}


\section{Partitions and random subgraphs}\label{sec:partn}

In this section we consider a sequence $\cP_1, \dots, \cP_\ell$ of successively finer partitions which will underlie our iterative absorption process. We will also construct corresponding sparse quasirandom subgraphs $R_i$ which will be used to `smooth out' the leftover from the approximate decomposition in each step of the process.

Recall from Section~\ref{sec:notat} that a $k$-partition is a partition satisfying (Pa\ref{Pa1}) and  (Pa\ref{Pa2}). Let $G$ be an $r$-partite graph on $(V_1,\dots, V_r)$. An \emph{$(\alpha,k,\delta)$-partition} for $G$ on $(V_1,\dots, V_r)$ is a $k$-partition $\cP=\{U^1, \dots, U^k\}$ of $V(G)$ such that in the following hold:
\begin{enumerate}[({Pa}1)]
 \setcounter{enumi}{2}
	\item for each $v\in V(G)$, each $1\leq i \leq k$ and each $1\leq j\leq r$,
	$$|d_G(v, U^i_j)-d_G(v, V_j)/k|<\alpha|U^i_j|;$$\label{Pa3}
	\item for each $1\leq i \leq k$, each $1\leq j\leq r$ and each $v\notin V_j$,
	$d_G(v, U^i_j)\geq \delta|U^i_j|$.\label{Pa4}
\end{enumerate}

The following proposition guarantees a $(n^{-1/3}/2,k, \delta-n^{-1/3}/2)$-partition of any sufficiently large balanced $r$-partite graph $G$ with  $\hat{\delta}(G)\geq \delta n$. To prove this result, it suffices to consider an equitable partition $U^1_j, U^2_j, \dots, U^k_j$ of $V_j$ chosen uniformly at random (with $|U^1_j|\leq \dots \leq |U^k_j|$).

\begin{prop}\label{prop:partition}
Let $k, r\in \N$. There exists $n_0$ such that if $n\geq n_0$ and $G$ is any $r$-partite graph on $(V_1,\dots, V_r)$ with $|V_1|=\dots=|V_r|=n$ and $\hat{\delta}(G)\geq \delta n$, then $G$ has a $(\nu,k, \delta-\nu)$-partition, where  $\nu:=n^{-1/3}/2$.\qed
\end{prop}
\COMMENT{proof - see appendix}

We say that $\cP_1, \cP_2, \dots, \cP_\ell$ is an \emph{$(\alpha, k, \delta, m)$-partition sequence} for $G$ on $(V_1, \dots, V_r)$ if, writing $\cP_0:=\{V(G)\}$,
\begin{enumerate}[(S1)]
	\item for each $1\leq i \leq \ell$, $\cP_i$ refines $\cP_{i-1}$;\label{S2}
	\item for each $1\leq i \leq \ell$ and each $W\in \cP_{i-1}$, $\cP_i[W]$ is an $(\alpha,k,\delta)$-partition for $G[W]$;\label{S4}
	\item for each $1\leq i \leq \ell$, all $1\leq j_1, j_2, j_3\leq r$ with $j_1\neq j_2, j_3$, each $W\in \cP_{i-1}$, each $U\in \cP_i[W]$ and each $v\in W_{j_1}$,
	$$|d_G(v, U_{j_2})-d_G(v, U_{j_3})|<\alpha |U_{j_1}|;$$\label{S3}
	\item for each $U\in \cP_\ell$ and each $1\leq j\leq r$, $|U_j|=m$ or $m-1$.\label{S5}
\end{enumerate}
Note that (S\ref{S4}) and (Pa\ref{Pa2}) together imply that $|U_{j_1}|=|U_{j_2}|$ for each $1\leq i \leq \ell$, each $U\in \cP_i$ and all $1\leq j_1,j_2\leq r$. 

By successive applications of Proposition~\ref{prop:partition}, we immediately obtain the following result which guarantees the existence of a suitable partition sequence (for details see~\cite{thesis}). 

\begin{lemma}\label{lem:partseq}
Let $k,r\in \N$ with $k\geq 2$ and let $0<\alpha<1$. There exists $m_0$ such that, for all $m'\geq m_0$, any $K_r$-divisible  graph $G$ on $(V_1, \dots, V_r)$ with $|V_1|=\dots=|V_r|=n\geq km'$ and $\hat{\delta}(G)\geq \delta n$ has an $(\alpha, k, \delta-\alpha, m)$-partition sequence for some $m'\leq m \leq km'$.	\qed
\end{lemma}
\COMMENT{proof - see appendix}


Suppose that we are given a $k$-partition $\cP$ of $G$. The following proposition finds a quasirandom spanning subgraph $R$ of $G$ so that each vertex in $R$ has roughly the expected number of neighbours in each set $U\in \cP$. The proof is an easy application of Lemma~\ref{lem:chernoff}.

\begin{prop}\label{prop:random2}
Let $1/n\ll\alpha, \rho,1/k,1/r\leq 1$.  Let $G$ be an $r$-partite graph on $(V_1, \dots, V_r)$ with $|V_1|=\dots=|V_r|=n$. Suppose that $\cP$ is a $k$-partition for $G$. Let $\cS$ be a collection of at most $n^2$ subsets of $V(G)$. Then there exists $R\subseteq G[\cP]$ such that for all $1\leq j\leq r$, all distinct $x,y\in V(G)$, all $U\in \cP$ and all $S\in \cS$:
\begin{itemize}
\item $|d_{R}(x, U_j)-\rho d_{G[\cP]}(x, U_j)|< \alpha |U_j|$;
\item $|d_{R}(\{x,y\}, U_j)-\rho^2 d_{G[\cP]}(\{x,y\}, U_j)|< \alpha |U_j|$;
\item $|d_{G}(y, N_{R}(x, U_j))-\rho d_{G}(y, N_{G[\cP]}(x, U_j))|< \alpha |U_j|$;
\item $|d_{R}(y, S_j)-\rho d_{G[\cP]}(y, S_j)|< \alpha n$. \qed
\end{itemize}
\end{prop}
\COMMENT{proof - see appendix}

We need to reserve some quasirandom subgraphs $R_i$ of $G$ at the start of our proof, whilst the graph $G$ is still almost balanced with respect to the partition sequence. We will add the edges of $R_i$ back after finding an approximate decomposition of $G[\cP_i]$ in order to assume the leftover from this approximate decomposition is quasirandom. The next lemma gives us suitable subgraphs for $R_i$.

\begin{lemma}\label{lem:randoms}
Let $1/m\ll \alpha\ll \rho, 1/k,1/r \leq 1$. Let $G$ be an $r$-partite graph on $(V_1, \dots, V_r)$ with $|V_1|=\dots=|V_r|$. Suppose that $\cP_1, \dots, \cP_\ell$ is a $(1, k, 0, m)$-partition sequence for $G$. Let $\cP_0:=\{V(G)\}$ and, for each $0\leq q \leq \ell$, let $G_q:=G[\cP_q]$. Then there exists a sequence of graphs $R_1, \dots, R_\ell$ such that $R_q\subseteq G_q-G_{q-1}$ for each $q$ and the following holds. For all $1\leq q\leq \ell$, all $1\leq j \leq r$, all $W\in \cP_{q-1}$, all distinct $x,y\in W$ and all $U\in \cP_q[W]$:
\begin{enumerate}[\rm (i)]
\item $|d_{R_q}(x, U_j)-\rho d_{G_q}(x, U_j)|< \alpha |U_j|$;\label{lem:randoms:i}
\item $|d_{R_q}(\{x,y\}, U_j)-\rho^2 d_{G_q}(\{x,y\}, U_j)|< \alpha |U_j|$;\label{lem:randoms:ii}
\item $d_{G_{q+1}'}(y, N_{R_q}(x, U_j))\geq \rho d_{G_{q+1}}(y, N_{G_q}(x, U_j))-3\rho^2|U_j|$, where $G_{q+1}':=G_{q+1}-R_{q+1}$ if $q\leq \ell-1$, $G_{\ell+1}':=G$ and $G_{\ell+1}:=G$.\label{lem:randoms:iii}
\end{enumerate}
\end{lemma}

\begin{proof}
For $1\leq q\leq \ell$, we say that the sequence of graphs $R_1, \ldots, R_q$ is \emph{good} if $R_i\subseteq G_i-G_{i-1}$ and for all $1\leq i\leq q$, all $1\leq j\leq r$, all $W\in \cP_{i-1}$, all distinct $x,y\in W$ and all $U\in \cP_i[W]$:
\begin{enumerate}[\rm(a)]
\item \eqref{lem:randoms:i} and \eqref{lem:randoms:ii} hold (with $q$ replaced by $i$);\label{item:a}
\item $|d_{G_{i+1}}(y, N_{R_i}(x, U_j))- \rho d_{G_{i+1}}(y, N_{G_i}(x, U_j))|< \alpha |U_j|$;\label{item:b}
\item if $i\leq q-1$, $d_{R_{i+1}}(y, N_{R_{i}}(x, U_j))< \rho d_{G_{i+1}}(y, N_{R_{i}}(x, U_j))+ \alpha |U_j|$.\label{item:c}
\end{enumerate}

Suppose $1\leq q\leq\ell$ and we have found a good sequence of graphs $R_1, \ldots, R_{q-1}$. We will find $R_q$ such that $R_1, \dots, R_q$ is good.
Let $W\in \cP_{q-1}$, let $\cS_1$ be the empty set and, if $q\geq 2$, let $W'\in \cP_{q-2}$ be such that $W\subseteq W'$ and let $\cS_q:=\{N_{R_{q-1}}(x, W): x\in W'\}$. Apply Proposition~\ref{prop:random2} (with $|W|/r$, $G_{q+1}[W]$, $\cP_q[W]$ and $\cS_q$ playing the roles of $n$, $G$, $\cP$ and $\cS$) to find $R_W\subseteq G_{q+1}[W][\cP_q[W]]=G_q[W]$ such that:%
\COMMENT{Note that $d_{G_{q+1}[W][\cP_q[W]]}(x, U_j)= d_{G_q}(x, U_j)$.}
\begin{align}
|d_{R_W}(x, U_j)-\rho d_{G_q}(x, U_j)|< \alpha |U_j|,\nonumber\\
|d_{R_W}(\{x,y\}, U_j)-\rho^2 d_{G_q}(\{x,y\}, U_j)|< \alpha |U_j|,\nonumber\\
|d_{G_{q+1}}(y, N_{R_W}(x, U_j))-\rho d_{G_{q+1}}(y, N_{G_q}(x, U_j))|< \alpha |U_j|,\nonumber\\
|d_{R_W}(y, S_j)-\rho d_{G_{q}}(y, S_j)|< \alpha |W_j|,\label{eq:sset}
\end{align}
for all $1\leq j\leq r$, all distinct $x,y\in W$, all $U\in \cP_q[W]$ and all $S\in \cS_q$. Set $R_q:=\bigcup_{W\in \cP_{q-1}}R_W$. It is clear that $R_1, \dots, R_q$ satisfy \eqref{item:a} and \eqref{item:b}. We now check that \eqref{item:c} holds when $1\leq i=q-1$. Let $1\leq j\leq r$, $W\in \cP_{q-2}$, $x,y\in W$ be distinct and $U\in \cP_{q-1}[W]$. If $y\notin U$, then $d_{R_q}(y, U_j)=0$ and so \eqref{item:c} holds. If $y\in U$, then $d_{R_q}(y,N_{R_{q-1}}(x, U))=d_{R_U}(y, N_{R_{q-1}}(x, U))$ and \eqref{item:c} follows by replacing $W$ and $S$ by $U$ and $N_{R_{q-1}}(x, U)$ in property \eqref{eq:sset}.
So $R_1, \dots, R_q$ is good.

So $G$ contains a good sequence of graphs $R_1, \dots, R_\ell$. We will now check that this sequence also satisfies \eqref{lem:randoms:iii}. If $q=\ell$, this follows immediately from \eqref{item:b}. Let $1\leq q< \ell$, $1\leq j\leq r$, $W\in \cP_{q-1}$, $x,y\in W$ be distinct and $U\in \cP_q[W]$. We have
\begin{align*}
d_{R_{q+1}}(y, N_{R_{q}}(x, U_j))&\stackrel{\mathclap{\eqref{item:c}}}{<} \rho d_{G_{q+1}}(y, N_{R_{q}}(x, U_j)) + \alpha |U_j| \\
&\stackrel{\mathclap{\eqref{item:b}}}{<} \rho^2 d_{G_{q+1}}(y, N_{G_q}(x, U_j))+ (\alpha\rho+\alpha)|U_j|\leq 2\rho^2|U_j|.
\end{align*}
Therefore,
\begin{align*}
d_{G_{q+1}'}(y, N_{R_q}(x, U_j))&=d_{G_{q+1}}(y, N_{R_q}(x, U_j))-d_{R_{q+1}}(y, N_{R_q}(x, U_j))\\
&\stackrel{\mathclap{\eqref{item:b}}}{\geq} \rho d_{G_{q+1}}(y, N_{G_q}(x, U_j))-3\rho^2|U_j|.
\end{align*}
So $R_1, \dots, R_\ell$ satisfy \eqref{lem:randoms:i}--\eqref{lem:randoms:iii}.
\end{proof}

We apply Lemma~\ref{lem:randoms} when $\cP_1, \dots, \cP_\ell$ is an $(\alpha, k, 1-1/r+\eps, m)$-partition sequence for $G$ to obtain the following result. For details of the proof, see \cite{thesis}.

\begin{cor}\label{cor:randoms}
Let $1/m\ll \alpha\ll \rho, 1/k \ll \eps, 1/r\leq 1$. Let $G$ be a $K_r$-divisible graph on $(V_1, \dots, V_r)$ with $|V_1|=\dots=|V_r|$. Suppose that $\cP_1, \dots, \cP_\ell$ is an $(\alpha, k, 1-1/r+\eps, m)$-partition sequence for $G$. Let $\cP_0:=\{V(G)\}$ and $G_q:=G[\cP_q]$ for $0\leq q\leq \ell$. There exists a sequence of graphs $R_1, \dots, R_\ell$ such that $R_q\subseteq G_q-G_{q-1}$ for each $1\leq q\leq \ell$ and the following holds. For all $1\leq q\leq \ell$, all $1\leq j, j'\leq r$, all $W\in \cP_{q-1}$, all distinct $x,y\in W$ and all $U, U'\in \cP_q[W]$:
\begin{enumerate}[\rm (i)]
\item $d_{R_q}(x, U_j)<\rho d_{G_q}(x, U_j)+\alpha |U_j|$;\label{cor:rand1}
\item $d_{R_q}(\{x,y\}, U_j)<(\rho^2+\alpha)|U_j|$;\label{cor:rand2}
\item if $x\notin U\cup U'\cup V_j \cup V_{j'}$, $|d_{R_q}(x, U_{j})-d_{R_q}(x, U'_{j'})|< 3\alpha |U_j|$;\label{cor:rand3}
\item if $x\notin U$, $y\in U$ and $x,y \notin V_j$, then
$$d_{G_{q+1}'}(y, N_{R_q}(x, U_j))\geq \rho (1-1/(r-1))d_{G_{q}}(x, U_j)+\rho^{5/4}|U_j|,$$
where $G_{q+1}':=G_{q+1}-R_{q+1}$ if $q\leq \ell-1$ and $G_{\ell+1}':=G$.\label{cor:rand4} \qed
\end{enumerate}
\end{cor}
\COMMENT{proof - see appendix}


\section{A remainder of low maximum degree}\label{sec:degred}

The aim of this section is to prove the following lemma which lets us assume that the remainder of $G$ after finding an $\eta$-approximate decomposition has small maximum degree.

\begin{lemma}\label{lem:degree}
Let $1/n \ll \alpha \ll \eta \ll \gamma\ll\eps<1/r<1$. Let $G$ be an $r$-partite graph on $(V_1, \dots, V_r)$ with $|V_1|=\dots=|V_r|=n$ and $\hat{\delta}(G)\geq (\hat{\delta}^\eta_{K_r} +\eps)n$. Suppose also that, for all $1\leq j_1, j_2 \leq r$ and every $v\notin V_{j_1}\cup V_{j_2}$, 
\begin{equation}
|d_G(v, V_{j_1})-d_G(v, V_{j_2})|< \alpha n.\label{eq:lem:degree:1}
\end{equation}
Then there exists $H\subseteq G$ such that $G-H$ has a $K_r$-decomposition and $\Delta(H)\leq \gamma n$.
\end{lemma}

Our strategy for the proof of Lemma~\ref{lem:degree} is as follows. We first remove a sparse random subgraph $H_1$ from $G$. We will then remove a further graph $H_2$ of small maximum degree from $G-H_1$ to achieve that $G-(H_1\cup H_2)$ is $K_r$-divisible. (The existence of such a graph $H_2$ is shown in Proposition~\ref{prop:fixing}.) The definition of $\delta^\eta_{K_r}$ then ensures that   $G-(H_1\cup H_2)$ has an $\eta$-approximate $K_r$-decomposition. We now consider the graph $R$ obtained from $G-H_2$ by deleting all edges in the copies of $K_r$ in this decomposition. Suppose that $v$ is a vertex whose degree in $R$ is too high. Our aim will be to find a $K_{r-1}$-matching in $H_1$ whose vertex set is the neighbourhood of $v$ in $G$. If $\rho$ denotes the edge-probability for the random subgraph $H_1$, then each vertex in $H_1$ is, on average, joined to at most $\rho d_G(v)/(r-1)\ll (1-1/(r-1)+\eps)d_G(v)/(r-1)$ vertices in each other part,%
\COMMENT{`$n$' should really be `$d_G(v, V_j)$' where $v\notin V_j$.}
so Theorem~\ref{thm:hajszem} alone is of no use. But Theorem~\ref{thm:hajszem} can be combined with the Regularity lemma in order to find the desired $K_{r-1}$-matching in $H_1$ (see Proposition~\ref{prop:randomg1}).

\subsection{Regularity}\label{subsec:reg}

In this section, we introduce a version of the Regularity lemma which we will use to prove Lemma~\ref{lem:degree}.

Let $G$ be a bipartite graph on $(A,B)$. For non-empty sets $X\subseteq A$, $Y\subseteq B$, we define the \emph{density of $G[X,Y]$} to be $d_G(X, Y):=e_G(X,Y)/|X||Y|$.  Let $\eps>0$. We say that $G$ is \emph{$\eps$-regular} if for all sets $X \subseteq A$ and $Y \subseteq B$ with $|X|\geq \eps |A|$ and $|Y| \geq \eps |B|$ we have $$|d_G(A,B) - d_G(X,Y)| < \eps.$$

The following simple result follows immediately from this definition.

\begin{prop}\label{prop:subsets}
Suppose that $0< \varepsilon \leq \alpha \leq 1/2$. Let $G$ be a bipartite graph on $(A,B)$. Suppose that $G$ is $\eps$-regular with density $d$. If $A' \subseteq A, B' \subseteq B$ with $|A'| \geq \alpha |A|$ and $|B'| \geq \alpha |B|$ then $G[A',B']$ is $\eps / \alpha$-regular and has density greater than $d-\varepsilon$.\qed
\end{prop}

Proposition~\ref{prop:subsets} shows that regularity is robust, that is, it is not destroyed by deleting even quite a large number of vertices. The next observation allows us to delete a small number of edges at each vertex and still maintain regularity. The proof again follows from the definition.

\begin{prop}\label{prop:regedges}
Let $n\in \N$ and let $0<\gamma \ll \varepsilon \leq 1$. Let $G$ be a bipartite graph on $(A,B)$ with $|A|=|B|=n$. Suppose that $G$ is $\eps$-regular with density $d$. Let $H\subseteq G$ with $\Delta(H)\leq \gamma n$ and let $G':=G-H$. Then $G'$ is $2\eps$-regular and has density greater than $d-\varepsilon/2$. \qed
\end{prop}
\COMMENT{proof - see appendix}

The following proposition takes a graph $G$ on $(V_1, \dots, V_r)$ where each pair of vertex classes induces an $\eps$-regular pair and allows us to find a $K_r$-matching covering most of the vertices in $G$. Part \eqref{prop:matching:1} follows from Proposition~\ref{prop:subsets} and the definition of regularity. For \eqref{prop:matching:2}, apply \eqref{prop:matching:1} repeatedly until only $\lceil\eps^{1/r}n\rceil$ vertices remain uncovered in each $V_j$. 

\begin{prop}\label{prop:matching}
Let $1/n \ll \eps \ll d, 1/r\leq 1$. Let $G$ be an $r$-partite graph on $(V_1, \dots, V_r)$ with $|V_1|=\dots =|V_r|=n$. Suppose that, for all $1\leq j_1<j_2\leq r$, the graph $G[V_{j_1}, V_{j_2}]$ is $\eps$-regular with density at least $d$.
\begin{enumerate}[\rm(i)]
\item For each $1\leq j\leq r$, let $W_j\subseteq V_j$ with $|W_j|= \lceil\eps^{1/r}n\rceil$. Then $G[W_1, \dots, W_r]$ contains a copy of $K_r$.\label{prop:matching:1}
\item The graph $G$ contains a $K_r$-matching which covers all but at most $2r\eps^{1/r}n$ vertices of $G$.\label{prop:matching:2}\qed
\end{enumerate}
\end{prop}
\COMMENT{proof - see appendix}

We will use a version of Szemer\'edi's Regularity lemma~\cite{szem} stated for $r$-partite graphs. It is proved in the same way as the non-partite degree version.%
\COMMENT{Proof starts with partition refining $(V_1, \dots, V_r)$ and refines this partition in each step.}

\begin{lemma}[Degree form of the $r$-partite Regularity lemma]\label{lem:degreeform}
Let $0<\eps<1$ and $k_0,r\in \N$. Then there is an $N = N(\eps,k_0,r)$ such that the following holds for every $0\leq d < 1$ and for every $r$-partite graph $G$ on $(V_1, \dots, V_r)$ with $|V_1|=\dots=|V_r|=n \geq N$. There exists a partition $\cP=\{U^0, \dots, U^k\}$ of $V(G)$, $m\in \N$ and a spanning subgraph $G'$ of $G$ satisfying the following:
	\begin{enumerate}[\rm(i)]
	\item $k_0\leq k\leq N$;\label{item:k}
	\item for each $1\leq j\leq r$, $|U^0_j|\leq \eps n$;\label{item:exceptional}
	\item for each $1\leq i\leq k$ and each $1\leq j\leq r$, $|U^i_j|=m$;\label{item:equi}
	\item for each $1\leq j\leq r$ and each $v\in V(G)$, $d_{G'}(v, V_j) > d_G(v, V_j) - (d + \eps)n$;\label{item:degree}
	\item for all but at most $\eps k^2$ pairs $U^{i_1}_{j_1}, U^{i_2}_{j_2}$ where $1\leq i_1, i_2\leq k$ and $1\leq j_1<j_2\leq r$, the graph $G'[U^{i_1}_{j_1}, U^{i_2}_{j_2}]$ is $\eps$-regular and has density either $0$ or $>d$.\label{item:regularpairs}
	\end{enumerate}
\end{lemma}

We define the \emph{reduced graph} $R$ as follows. The vertex set of $R$ is the set of clusters $\{U^i_j:1\leq i \leq k \text{ and } 1\leq j\leq r\}$. For each $U,U' \in V(R)$, $UU'$ is an edge of $R$ if the subgraph $G'[U,U']$ is $\eps$-regular and has density greater than $d$. Note that $R$ is a balanced $r$-partite graph with vertex classes $W_j:=\{U^i_j:1\leq i \leq k\}$ for $1\leq j \leq r$. The following simple proposition relates the minimum degree of $G$ and the minimum degree of $R$.

\begin{prop}\label{prop:mindeg}
Suppose that $0<2\eps \leq d \leq c/2$. Let $G$ be an $r$-partite graph on $(V_1, \dots, V_r)$ with $|V_1|=\dots =|V_r|=n$ and $\hat\delta(G) \geq cn$. Suppose that $G$ has a partition $\cP=\{U^0, \dots, U^k\}$ and a subgraph $G'\subseteq G$ as given by Lemma~\ref{lem:degreeform}. Let $R$ be the reduced graph of $G$. Then $\hat\delta(R) \geq (c-2d)k$.\qed
\end{prop}
\COMMENT{proof - see appendix}

\subsection{Degree reduction}

At the beginning of our proof of Lemma~\ref{lem:degree}, we will reserve a random subgraph $H_1$ of $G$. Proposition~\ref{prop:randomg1} below ensures that we can partition the neighbourhood of each vertex so that $H_1$ induces $\eps$-regular graphs between these parts. 
In our proof of Proposition~\ref{prop:randomg1}, we will use the following well-known result for which we omit the proof.

\begin{prop}\label{prop:slice}
Let $1/n \ll \eps\ll d, \rho \leq 1$. Let $G$ be a bipartite graph on $(A,B)$ with $|A|=|B|=n$. Suppose that $G$ is $\eps$-regular with density at least $d$. Let $H$ be a graph formed by taking each edge of $G$ independently with probability $\rho$. Then, with probability at least $1-1/n^2$, $H$ is $4\eps$-regular with density at least $\rho d/2$.\qed
\end{prop}
\COMMENT{proof - see appendix}

\begin{prop}\label{prop:randomg1}
Let $1/n\ll \alpha \ll 1/N\ll 1/k_0\leq \eps^*\ll  d\ll \rho<\eps, 1/r<1$. Let $G$ be an $r$-partite graph on $(V_1,\dots, V_r)$ with $|V_1|=\dots=|V_r|=n$ and $\hat{\delta}(G)\geq (1-1/r+\eps)n$. Suppose that for all $1\leq j_1, j_2 \leq r$ and every $v\notin V_{j_1}\cup V_{j_2}$, $|d_G(v, V_{j_1})-d_G(v, V_{j_2})|< \alpha n$. Then there exists $H\subseteq G$ satisfying the following properties:
\begin{enumerate}[\rm(i)]
\item For each $1\leq j\leq r$ and each $v\in V(G)$, $|d_H(v, V_j)-\rho d_G(v, V_j)|<\alpha n$.
In particular, for any $1\leq j_1, j_2\leq r$ such that $v\notin V_{j_1}\cup V_{j_2}$, $|d_H(v, V_{j_1})-d_H(v, V_{j_2})|<3\alpha n$.\label{prop:randomg1:1}
\item For each vertex $v\in V(G)$, there exists a partition $\cP(v)=\{U^0(v), \dots, U^{k_v}(v)\}$ of $N_G(v)$  and $m_v\in \N$ such that:\label{prop:randomg1:2}
	\begin{itemize}
	\item $k_0 \leq k_v\leq N$;
	\item for each $1\leq j\leq r$, $|U^0_j(v)|\leq \eps^* n$;
	\item for each $1\leq i \leq k_v$ and each $1\leq j\leq r$ such that $v\notin V_j$, $|U^i_j(v)|=m_v$;
	\item for each $1\leq i\leq k_v$ and all $1\leq j_1<j_2\leq r$ such that $v\notin V_{j_1}\cup V_{j_2}$, the graph $H[U^i_{j_1}(v), U^i_{j_2}(v)]$ is $\eps^*$-regular with density greater than $d$.
	\end{itemize}
\end{enumerate}
\end{prop}

Roughly speaking, \eqref{prop:randomg1:2} says that for each $v\in V(G)$ the reduced graph of $H[N_G(v)]$ has a perfect $K_{r-1}$-matching.

\begin{proof}
Let $H$ be the graph formed by taking each edge of $G$ independently with probability~$\rho$. For each $1\leq j\leq r$ and each $v\in V(G)$, Lemma~\ref{lem:chernoff} gives%
\COMMENT{$\alpha^2n^2/d_G(v, V_j)>\alpha^2n^2/n>3\log n$.}
$$\pr(|d_H(v, V_j)-\rho d_H(v, V_j)|\geq\alpha n)\leq 2e^{-2\alpha^2n}<1/rn^2.$$
So the probability that there exist $1\leq j\leq r$ and $v\in V(G)$ such that $|d_H(v, V_j)-\rho d_G(v,V_j)|\geq \alpha n$ is at most $rn/rn^2=1/n$.
Let $1\leq j_1, j_2\leq r$. Note that if $v\notin V_{j_1}\cup V_{j_2}$ and $|d_H(v, V_j)-\rho d_G(v, V_j)|<\alpha n$ for $j=j_1,j_2$, then
$$|d_H(v, V_{j_1})-d_H(v, V_{j_2})|< |\rho d_G(v, V_{j_1})-\rho d_G(v, V_{j_2})|+2\alpha n< 3\alpha n.$$
So $H$ satisfies \eqref{prop:randomg1:1} with probability at least $1-1/n$.

We will now show that $H$ satisfies \eqref{prop:randomg1:2} with probability at least $1/2$. We find partitions of the neighbourhood of each vertex $v\in V(G)$ as follows. To simplify notation, we will assume that $v\in V_1$ (the argument is identical for the other cases). For all $2\leq j_1, j_2\leq r$, we have $|d_G(v,V_{j_1})-d_G(v, V_{j_2})|<\alpha n$. So, there exists $n_v$ and, for each $2\leq j\leq r$, a subset $V_j(v)\subseteq N_G(v, V_j)$ such that $|V_j(v)|>d_G(v,V_j)-\alpha n$%
\COMMENT{need this so that not adding too many vertices to exceptional set.}
and 
$$|V_{j}(v)|=n_v\geq \hat\delta(G)\geq (1-1/r)n.$$
Let $G_v$ denote the balanced $(r-1)$-partite graph $G[V_2(v), \dots, V_r(v)]$. Note that
\begin{equation}
\hat\delta(G_v)\geq n_v-\frac{n}{r}+\eps n\geq \left(1-\frac{1}{r-1}+\eps\right)n_v.\label{eq:Gvdegree}
\end{equation}

Apply Lemma~\ref{lem:degreeform} (with $\eps^*/4$, $2d/\rho$, $k_0$ and $G_v$ playing the roles of $\eps$, $d$, $k_0$ and $G$) to find a partition $\cQ(v)=\{W^0(v), \dots, W^{k_v}(v)\}$ of $V(G_v)$ satisfying properties \eqref{item:k}--\eqref{item:regularpairs} of Lemma~\ref{lem:degreeform}. Let $m_v:=|W^1_2(v)|$.  Let $R_v$ denote the reduced graph corresponding to this partition. Proposition~\ref{prop:mindeg} together with \eqref{eq:Gvdegree} implies that%
\COMMENT{$d\leq \rho^2/8 \implies 4d/\rho \leq \eps/2$}
\begin{equation*}
\hat\delta(R_v) \geq (1-1/(r-1)+\eps/2)k_v.
\end{equation*}
So we can use Theorem~\ref{thm:hajszem} to find a perfect $K_{r-1}$-matching $M_v$ in $R_v$. Let $U^0(v):=W^0(v)\cup(N_G(v)\setminus V(G_v))$. Note that for each $2\leq j\leq r$,
$|U^0_j|< |W^0_j|+\alpha n\leq \eps^*n$. Let $\cP(v):=\{U^0(v), \dots, U^{k_v}(v)\}$ be a partition of $N_G(v)$ which is chosen such that, for each $1\leq i \leq k_v$, $\{U^i_j(v):2\leq j\leq r\}$ induces a copy of $K_{r-1}$ in $M_v$. By the definition of $R_v$, for each $1\leq i \leq k_v$ and all $2\leq j_1<j_2\leq r$, the graph $G[U^i_{j_1}(v), U^i_{j_2}(v)]$ is $\eps^*/4$-regular with density greater than $2d/\rho$. 

Fix $1\leq i\leq k_v$ and $2\leq j_1<j_2\leq r$. Proposition~\ref{prop:slice} (with $m_v$, $\eps^*/4$, $2d/\rho$, $G[U^i_{j_1}(v), U^i_{j_2}(v)]$ and $H[U^i_{j_1}(v), U^i_{j_2}(v)]$ playing the roles of $n$, $\eps$, $d$, $G$ and $H$) gives that $H[U^{i}_{j_1}(v), U^{i}_{j_2}(v)]$ is $\eps^*$-regular and has density greater than $d$ with probability at least $1-1/m_v^2$.

We require the graph $H[U^{i}_{j_1}(v), U^{i}_{j_2}(v)]$ to be $\eps^*$-regular with density greater than $d$ for every edge $U^{i}_{j_1}(v) U^{i}_{j_2}(v)\in E(M_v)$. There are $k_v$ choices for $i$ and, for each $i$, there are $\binom{r-1}{2}$ choices for $j_1$ and $j_2$. So the probability that, for fixed $v\in V(G)$, there exists an edge $U^{i}_{j_1}(v) U^{i}_{j_2}(v)\in E(M_v)$ which fails to be $\eps^*$-regular with density greater than $d$ is at most%
\COMMENT{$k_vr^2\frac{1}{m_v^2}\leq \frac{k_vr^2}{(n/2k_v)^2}=\frac{4k_v^3r^2}{n^2}<\frac{1}{2rn}$}
$$k_vr^2\frac{1}{m_v^2}<\frac{1}{2rn}.$$
We multiply this probability by $rn$ for each of the $rn$ choices of $v$ to see that $H$ satisfies property (ii) with probability at least $1-rn/2rn=1/2$. Hence, the graph $H$ satisfies both (i) and (ii) with probability at least $1/2-1/n>0$. So we can choose such a graph $H$.
\end{proof}

Recall that in order to prove Lemma~\ref{lem:degree}, we will first remove a sparse random subgraph $H_1$ from $G$. In order to find an $\eta$-approximate $K_r$-decomposition in $G':=G-H_1$, we would like to use the definition of $\hat\delta^\eta_{K_r}$ which requires $G'$ to be $K_r$-divisible. The next proposition shows that, provided that $d_{G'}(v,V_{j_1})$ is close to $d_{G'}(v, V_{j_2})$ for all $1\leq j_1, j_2\leq r$ and $v\notin V_{j_1}\cup V_{j_2}$, the graph $G'$ can be made $K_r$-divisible by removing a further subgraph $H_1$ of small maximum degree.

\begin{prop}\label{prop:fixing}
Let $1/n \ll \alpha \ll \gamma\ll 1/r<1$. Let $G$ be an $r$-partite graph on $(V_1,\dots, V_r)$ with $|V_1|=\dots=|V_r|=n$ and $\hat{\delta}(G)\geq (1/2 +2\gamma/r)n$.
Suppose that, for all $1\leq j_1, j_2 \leq r$ and every $v\in V(G)\setminus (V_{j_1}\cup V_{j_2})$, $|d_{G}(v, V_{j_1})-d_G(v, V_{j_2})|< \alpha n$.
Then there exists $H\subseteq G$ such that $G-H$ is $K_r$-divisible and $\Delta(H)\leq \gamma n$.
\end{prop}

To prove Proposition~\ref{prop:fixing}, we require the following result whose proof is based on the Max-Flow-Min-Cut theorem. 

\begin{prop}\label{prop:flow}
Suppose that $1/n \ll \alpha\ll \xi\ll 1$. Let $G$ be a bipartite graph on $(A,B)$ with $|A|=|B|=n$. Suppose that $\delta(G)\geq (1/2+4\xi)n$. For every vertex $v\in V(G)$, let $n_v\in \N$ be such that $(\xi-\alpha)n\leq n_v\leq (\xi+\alpha)n$ and such that $\sum_{a\in A}n_a=\sum_{b\in B}n_b$. Then $G$ contains a spanning graph $G'$ such that $d_{G'}(v)=n_v$ for every $v\in V(G)$.
\end{prop}

\begin{proof}
We will use the Max-Flow-Min-Cut theorem.  Orient every edge of $G$ towards $B$ and give each edge capacity one. Add a source vertex $s^*$ which is attached to every vertex $a\in A$ by an edge of capacity $n_a$. Add a sink vertex $t^*$ which is attached to every vertex in $b\in B$ by an edge of capacity $n_b$. Let $c_0:=\sum_{a\in A}n_a=\sum_{b\in B}n_b$. Note that an integer-valued $c_0$-flow corresponds to the desired spanning graph $G'$ in $G$. So, by the Max-Flow-Min-Cut theorem, it suffices to show that every cut has capacity at least $c_0$.

Consider a minimal cut $C$. Let $S\subseteq A$ be the set of all vertices $a\in A$ for which $s^*a\notin C$ and let $T\subseteq B$ be the set of all $b\in B$ for which $bt^*\notin C$. Let $S':=A\setminus S$ and $T':=B\setminus T$. Then $C$ has capacity
$$c:=\sum_{s\in S'}n_s+e_G(S, T)+\sum_{t\in T'}n_t.$$

First suppose that $|S|\geq (1/2-2\xi) n$. In this case, since $\delta(G)\geq (1/2+4\xi) n$, each vertex in $T$ receives at least $2\xi n$ edges from $S$. So
$$c\geq \sum_{t\in T'}n_t+2|T|\xi n \geq \sum_{t\in T'}n_t+|T|(\xi+\alpha) n\geq c_0.$$ 
A similar argument works if $|T|\geq (1/2-2\xi) n$.%
\COMMENT{each vertex in $S$ sends at least $2\xi n$ edges to $T$. So $$c\geq \sum_{s\in S'}n_s+2|S|\xi n \geq \sum_{s\in S'}n_s+|S|(\xi+\alpha) n\geq c_0.$$}
Suppose then that $|S|, |T|< (1/2-2\xi) n$. Then $|S'|,|T'|> (1/2+2\xi) n$ and
$$c\geq \sum_{s\in S'}n_s+\sum_{t\in T'}n_t\geq (|S'|+|T'|)(\xi-\alpha) n> (n+4\xi n)(\xi-\alpha) n \geq (\xi+\alpha) n^2\geq c_0,$$
\COMMENT{$\xi\geq 2\alpha$ and $\xi^2\geq \alpha \implies 2\xi(\xi-\alpha)\geq \alpha\iff (1+4\xi)(\xi-\alpha)=\xi-\alpha+4\xi(\xi-\alpha)\geq \xi+\alpha$}
as required.
\end{proof}

We now use Proposition~\ref{prop:flow} to prove Proposition~\ref{prop:fixing}.

\begin{proofof}{Proposition~\ref{prop:fixing}}
For each $v\in V(G)$, let $$m_v:=\min\{d_G(v, V_j): 1\leq j\leq r\text{ with } v\notin V_j\}.$$
For each $1\leq j\leq r$ and each $v\notin V_j$, let $a_{v,j}:=d_G(v, V_j)-m_v$.
Note that, 
\begin{equation}
0\leq a_{v,j}<\alpha n.
\label{eq:avj}
\end{equation}
For each $1\leq j\leq r$, let $N_j:=\sum_{v\in V_j} m_v$. We have, for any $1\leq j_1, j_2\leq r$,
\begin{align}
|N_{j_1}-N_{j_2}|&=\left|\sum_{v\in V_{j_1}}(d_G(v, V_{j_2})-a_{v,j_2})-\sum_{v\in V_{j_2}}(d_G(v, V_{j_1})-a_{v,j_1})\right|\nonumber\\
&=\left|\sum_{v\in V_{j_1}}a_{v,j_2}-\sum_{v\in V_{j_2}}a_{v,j_1}\right|\stackrel{\eqref{eq:avj}}{<}\alpha n^2.\label{eq:Nj}
\end{align}
Let $N:=\min\{N_j: 1\leq j\leq r\}$ and, for each $1\leq j\leq r$, let $M_j:=N_j-N$. Note that \eqref{eq:Nj} implies $0\leq M_j< \alpha n^2$.
For each $1\leq j\leq r$ and each $v\in V_j$, choose $p_v\in \N$ to be as equal as possible such that $\sum_{v\in V_j}p_v=M_j$. Then%
\COMMENT{$0\leq p_v\leq \lceil \alpha n\rceil$}
\begin{equation}
0\leq p_v< \alpha n+1.
\label{eq:pv}
\end{equation}

Let $\xi:=\gamma/2r$. For each $1\leq j\leq r$ and each $v\notin V_j$, let
\begin{equation*}
n_{v,j}:=\lceil\xi n\rceil+a_{v,j}+p_v.
\end{equation*}
Using \eqref{eq:avj} and \eqref{eq:pv}, we see that,
\begin{equation}
\xi n\leq n_{v,j}\leq (\xi+3\alpha)n.\label{eq:nvj2}
\end{equation}

We will consider each pair $1\leq j_1<j_2\leq r$ separately and choose a subgraph $H_{j_1,j_2}$ that will become $H[V_{j_1}, V_{j_2}]$. Fix $1\leq j_1<j_2\leq r$ and observe that,
\begin{align*}
\sum_{v\in V_{j_1}}n_{v,j_2}&=\sum_{v\in V_{j_1}}(\lceil\xi n\rceil+a_{v,j_2}+p_v)=\lceil\xi n\rceil n+\sum_{v\in V_{j_1}}a_{v,j_2}+M_{j_1}\\
&=\lceil\xi n\rceil n+M_{j_1}+\sum_{v\in V_{j_1}}(d_G(v, V_{j_2})-m_v)
=\lceil\xi n\rceil n+M_{j_1}+e_G(V_{j_1}, V_{j_2})-N_{j_1}\\
&=\lceil\xi n\rceil n-N+e_G(V_{j_1}, V_{j_2})=\sum_{v\in V_{j_2}}n_{v,j_1}.
\end{align*}
Let $G_{j_1,j_2}:=G[V_{j_1}, V_{j_2}]$ and note that $\delta(G_{j_1,j_2})\geq (1/2+4\xi) n$. 
Apply Proposition~\ref{prop:flow} (with $3\alpha$, $\xi$, $G_{j_1,j_2}$, $V_{j_1}$ and $V_{j_2}$ playing the roles of $\alpha$, $\xi$, $G$, $A$ and $B$) to find $H_{j_1, j_2}\subseteq G_{j_1, j_2}$ such that $d_{H_{j_1, j_2}}(v)=n_{v, j_2}$ for every $v\in V_{j_1}$ and $d_{H_{j_1, j_2}}(v)=n_{v, j_1}$ for every $v\in V_{j_2}$.

Let $H:=\bigcup_{1\leq j_1<j_2\leq r}H_{j_1, j_2}$. By \eqref{eq:nvj2}, we have $\Delta(H)\leq 2r\xi n= \gamma n$. For any $1\leq j\leq r$ and any $v\notin V_j$, we have%
\begin{align*}
d_{G-H}(v, V_j)&= d_G(v, V_j)-d_H(v, V_j)=d_G(v, V_j)-n_{v, j}\\
&=d_G(v, V_j)-\lceil\xi n\rceil-d_G(v, V_j)+m_v-p_v=m_v-p_v-\lceil\xi n\rceil.
\end{align*}
So $G-H$ is $K_r$-divisible.
\end{proofof}

We now have all the necessary tools to prove Lemma~\ref{lem:degree}. This lemma finds an approximate $K_r$-decomposition which covers all but at most $\gamma n$ edges at any vertex.

\begin{proofof}{Lemma~\ref{lem:degree}}
The lemma trivially holds if $r=2$, so we may assume that $r\geq 3$. In particular, by Proposition~\ref{prop:extrem}, $\hat\delta(G)\geq (1-1/(r+1)+\eps/2)n$.
Choose constants $N$, $k_0$, $\eps^*$, $d$ and $\rho$ satisfying 
$$\eta \ll 1/N\ll 1/k_0\leq \eps^*\ll d \ll \rho\ll \gamma.$$
Apply Proposition~\ref{prop:randomg1} to find a subgraph $H_1\subseteq G$ satisfying properties \eqref{prop:randomg1:1}--\eqref{prop:randomg1:2}.

Let $G_1:=G-H_1$. Using \eqref{eq:lem:degree:1} and that $H_1$ satisfies Proposition~\ref{prop:randomg1}\eqref{prop:randomg1:1}, for all $1\leq j_1, j_2\leq r$ and each $v\notin V_{j_1} \cup V_{j_2}$,
\begin{align*}
|d_{G_1}(v, V_{j_1})-d_{G_1}(v, V_{j_2})|&\leq |d_G(v, V_{j_1})-d_G(v, V_{j_2})|+ |d_{H_1}(v, V_{j_1})-d_{H_1}(v, V_{j_2})|\\
&< \alpha n + 3\alpha n= 4\alpha n.
\end{align*}
Note also that $\hat\delta(G_1) \geq 3n/4$.
So we can apply Proposition~\ref{prop:fixing} (with $G_1$, $4\alpha$ and $\gamma/2$ playing the roles of $G$, $\alpha$ and $\gamma$) to obtain $H_2\subseteq G_1$ such that $G_1-H_2$ is $K_r$-divisible and $\Delta(H_2)\leq \gamma n/2$.
Then $\hat{\delta}(G_1-H_2)\geq (\hat{\delta}^\eta_{K_r}+\eps/2)n$, so we can find an $\eta$-approximate $K_r$-decomposition $\cF$ of $G_1-H_2$.

Let $G_2:=G_1-H_2-\bigcup\cF$ be the graph consisting of all the remaining edges in $G_1-H_2$.
Let $$B:=\{v\in V(G): d_{G_2}(v)>\eta^{1/2}n\}.$$ Note that 
\begin{equation}\label{eq:d1}
|B|\leq 2e(G_2)/\eta^{1/2}n\leq 2\eta^{1/2}n.
\end{equation}
Let $\cF_1:=\{F\in \cF: F\cap B=\emptyset\}$ and let $G_3:=G-\bigcup\cF_1$. If $v\in B$, then $N_{G_3}(v)=N_{G}(v)$. Suppose that $v\notin B$. For any $u\in B$, at most one copy of $K_r$ in $\cF\setminus \cF_1$ can contain both $u$ and $v$. So there can be at most $(r-1)|B|$ edges in $\bigcup (\cF\setminus \cF_1)$ that are incident to $v$ and so
\begin{align}
d_{G_3}(v)&\leq d_{H_1}(v)+d_{H_2}(v)+d_{G_2}(v)+(r-1)|B|\nonumber\\
&\leq (r-1)(\rho+\alpha)n+\gamma n/2+\eta^{1/2}n+2(r-1)\eta^{1/2}n \leq \gamma n.\label{eq:d5}
\end{align}

Label the vertices of $B=\{v_1, v_2, \dots, v_{|B|}\}$. We will use copies of $K_r$ to cover most of the edges at each vertex $v_i$ in turn. We do this by finding a $K_{r-1}$-matching $M_i$ in $H_1[N_{G_3}(v_i)]=H_1[N_{G}(v_i)]$ in turn for each $i$. Suppose that we are currently considering $v:=v_i$ and let $\cM:=\bigcup_{1\leq j<i}M_j$. To simplify notation, we will assume that $v\in V_1$ (the proof in the other cases is identical). 

Let $\cP(v)=\{U^0(v), \dots, U^{k_v}(v)\}$ be a partition of $N_{G}(v)$ satisfying Proposition~\ref{prop:randomg1}\eqref{prop:randomg1:2}. We can choose a partition $\cQ(v)=\{W^0(v), \dots, W^{k_v}(v)\}$ of $N_{G}(v)$ and $m_v'\geq m_v-|B|$ such that, for each $1\leq i\leq k_v$:
\begin{itemize}
\item $W^i(v)\subseteq U^i(v)$;
\item $W^i(v)\cap B=\emptyset$;
\item for each $2\leq j\leq r$, $|W^i_j(v)|=m_v'$.
\end{itemize}
Note that, using \eqref{eq:d1},  $|W^0(v)|\leq |U^0(v)|+|B|k_vr\leq r(\eps^*n+2\eta^{1/2}nk_v)\leq 2\eps^*rn$.

By Proposition~\ref{prop:randomg1}\eqref{prop:randomg1:2}, for each $1\leq i\leq k_v$ and all $2\leq j_1<j_2\leq r$, the graph $H_1[U^i_{j_1}(v), U^i_{j_2}(v)]$ is $\eps^*$-regular with density greater than $d$. So Proposition~\ref{prop:subsets} implies that $H_1[W^i_{j_1}(v), W^i_{j_2}(v)]$ is $2\eps^*$-regular with density greater than $d/2$. Let $H_1':=H_1-\cM$. Using \eqref{eq:d1}, we have $\Delta(\cM[W^i_{j_1}(v), W^i_{j_2}(v)])\leq |B|\leq \eta^{1/3}m_v'$.%
\COMMENT{$|B|\leq 2\eta^{1/2}n\leq 2\eta^{1/2}(N+1)m_v\leq 4\eta^{1/2}Nm_v'\leq \eta^{1/3}m_v'$}
 So we can apply Proposition~\ref{prop:regedges} (with $m_v'$, $\eta^{1/3}$ and $2\eps^*$ playing the roles of $n$, $\gamma$ and $\eps$) to see that $H_1'[W^i_{j_1}(v), W^i_{j_2}(v)]$ is $4\eps^*$-regular with density greater than $d/3$.

We use Proposition~\ref{prop:matching} (with $m_v'$, $4\eps^*$, $d/3$ and $r-1$ playing the roles of $n$, $\eps$, $d$ and $r$) to find a $K_{r-1}$-matching covering all but at most $2(r-1)(4\eps^*)^{1/(r-1)}m_v'$ vertices in $H_1'[W^i(v)]$ for each $1\leq i \leq k_v$. Write $M_i$ for the union of these $K_{r-1}$-matchings over $1\leq i \leq k_v$. Note that $M_i$ covers all but at most
\begin{equation}\label{eq:misses}
|W^0(v)|+2(r-1)(4\eps^*)^{1/(r-1)}m_v'k_v\leq 2\eps^*rn+2(r-1)(4\eps^*)^{1/(r-1)}n\leq \gamma n
\end{equation}
vertices in $N_{G}(v)$.

Continue to find edge-disjoint $M_1, \dots, M_{|B|}$. For each $1\leq i \leq |B|$, $M_i':=\{v_i\cup K: K\in M_i\}$ is an edge-disjoint collection of copies of $K_r$ in $G_3$ covering all but at most $\gamma n$ edges at $v_i$ in $G$. Write $\cM':=\bigcup_{1\leq i\leq |B|} M_i'$ and let $H:=G_3-\bigcup \cM'=G-\bigcup (\cF_1\cup\cM')$. Then $G-H=\bigcup(\cF_1\cup\cM')$ has a $K_r$-decomposition and $\Delta(H)\leq \gamma n$, by \eqref{eq:d5} and \eqref{eq:misses}.
\end{proofof}


\section{Covering a pseudorandom remainder between vertex classes}\label{sec:pseud}
Recall from Section~\ref{sec:sketch} that in each iteration step we are given an $r$-partite graph, $G'$ say,
as well as a $k$-partition $\cP$ and our aim is to 
cover all edges of $G'[\cP]$ (which consists of those edges of $G'$ joining different partition
classes of $\cP$) with edge-disjoint $r$-cliques.
Lemma~\ref{lem:degree} allows us to assume that $G'[\cP]$ has low maximum degree. 
When carrying out the actual iteration in Section~\ref{sec:proof}, we will also add a suitable graph $R$ to $G'$ to be able to assume additionally that the remainder $G''[\cP]$ is actually quasirandom, where $G'':=R\cup G'$.
The aim of this section is to prove Corollary~\ref{cor:pseud:parts}, which allows us to cover all edges of $G''[\cP]$ while using only a small number of edges from $G''-G''[\cP]$ (the latter property is vital in order to be able to carry out the next iteration step).
We achieve this by finding, for each $x\in V(G'')$, suitable vertex-disjoint copies of $K_{r-1}$ inside $G''-G''[\cP]$ such that each copy of $K_{r-1}$ forms a copy of $K_r$ together with the edges incident to $x$ in $G''[\cP]$.

Corollary~\ref{cor:pseud:parts} will follow easily from repeated applications of 
Lemma~\ref{lem:pseud:bip}. The quasirandomness of $G[\cP]$ in Lemma~\ref{lem:pseud:bip}
is formalized by conditions (iii) and (iv) (roughly speaking, the graph $G$ in Lemma~\ref{lem:pseud:bip} plays the role of $G''$ above).
The fact that we may assume the balancedness condition (i) will follow from the arguments in Section~\ref{sec:bal}.
We can assume (ii) since this part of the graph is essentially unaffected by previous iterations.
When deriving Corollary~\ref{cor:pseud:parts}, the $W^i$ in Lemma~\ref{lem:pseud:bip} will play the role of the  neighbourhoods of the vertices $x$ appearing in Corollary~\ref{cor:pseud:parts}.
\begin{lemma}\label{lem:pseud:bip}
Let $r\geq 2$ and $1/n \ll 1/k, 1/r,\rho \leq 1$. Let $G$ be an $r$-partite graph on $(V_1, \dots, V_r)$ with $|V_1|=\dots=|V_r|=n$. Let $q\leq krn$ and let $W^1, \dots, W^q\subseteq V(G)$. Suppose that:
\begin{enumerate}[\rm(i)]
\item for each $1\leq i\leq q$, there exists $1\leq j_i\leq r$ and $n_i\in \N$ such that, for each $1\leq j\leq r$, $|W^i_{j}|=0$ if $j=j_i$ and $|W^i_j|=n_i$ otherwise;
\item for each $1\leq i\leq q$, $\hat{\delta}(G[W^i])\geq (1-1/(r-1))n_i+9kr^2\rho^{3/2}n$;\label{pseud:bip:ii}
\item for all $1\leq i_1<i_2\leq q$, $|W^{i_1}\cap W^{i_2}|\leq 2r\rho^2 n$;
\item each $v\in V(G)$ is contained in at most $2k\rho n$ of the sets $W^1, \ldots, W^q$.
\end{enumerate}
Then there exist edge-disjoint $T_1, \dots, T_q$ in $G$ such that each $T_i$ is a perfect $K_{r-1}$-matching in $G[W^i]$.
\end{lemma}

 The proof of Lemma~\ref{lem:pseud:bip} is similar to that of Lemma~10.7 in \cite{mindeg}, we include it here for completeness. The idea is to use a `random greedy' approach:
for each $s$ in turn, we find a suitable perfect $K_{r-1}$-matching $T_s$ in 
$G'_s:=G[W^s]-(T_1\cup \dots \cup T_{s-1})$. In order to ensure that $G'_s$ still has sufficiently large 
minimum degree for this to work, we choose the  $T_i$ uniformly at random from a suitable subset of 
the available candidates. To analyze this random choice, we will use the following result.

\begin{prop}[Jain, see {\cite{super-chernoff}}] \label{generalised-chernoff}
Let $X_1, \ldots, X_n$ be Bernoulli random variables such that, for any $1 \leq s \leq n$ and any $x_1, \ldots, x_{s-1}\in \{0,1\}$,
 \begin{align*}
\pr(X_s = 1 \mid X_1 = x_1, \ldots, X_{s-1} = x_{s-1}) \leq p.
\end{align*}
Let $X=\sum_{s=1}^n X_i$ and let $B \sim B(n,p)$.  Then $\pr(X \geq a) \leq \pr(B \geq a)$ for any~$a\ge 0$.
\end{prop}

\begin{proofof}{Lemma~\ref{lem:pseud:bip}}
Set $t:=\lceil 8kr\rho^{3/2}n\rceil$. Let $G_i:=G[W^i]$ for $1\leq i\leq q$. Suppose we have already found $T_1, \dots T_{s-1}$ for some $1\leq s \leq q$. We find $T_s$ as follows.

Let $H_{s-1}:=\bigcup_{i=1}^{s-1} T_i$ and $G_s':=G_s-H_{s-1}[W^s]$. If $\Delta(H_{s-1}[W^s])>(r-2)\rho^{3/2}n$, let $T'_1, \dots, T'_t$ be empty graphs on $W^s$. Otherwise, \eqref{pseud:bip:ii} implies
\begin{align*}
\hat\delta(G_s')&\geq \big(1-\frac{1}{r-1}\big)n_s+8kr^2\rho^{3/2}n
\geq \big(1-\frac{1}{r-1}+\rho^{3/2}\big)n_s+(r-2)(t-1)
\end{align*}
and we can greedily find $t$ edge-disjoint perfect $K_{r-1}$-matchings $T'_1, \dots, T'_t$ in $G_s'$ using Theorem~\ref{thm:hajszem}.%
\COMMENT{\eqref{pseud:bip:ii} implies that $n_s\geq 9kr^2\rho^{3/2}n\geq \sqrt n$ and so $1/n_s\ll \rho, 1/r$. So we can apply Theorem~\ref{thm:hajszem}.}
In either case, pick $1\leq i\leq t$ uniformly at random and set $T_s:=T'_i$. It suffices to show that, with positive probability,
\begin{equation*}
\Delta(H_{s-1}[W^s])\leq (r-2)\rho^{3/2}n \hspace{1cm}\text{ for all } 1\leq s\leq q.
\end{equation*}

Consider any $1\leq i\leq q$ and any $w\in W^i$. For $1\leq s\leq q$, let $Y_s^{i,w}$ be the indicator function of the event that $T_s$ contains an edge incident to $w$ in $G_i$. Let $X^{i,w}:=\sum_{s=1}^qY_s^{i,w}$. Note $d_{H_{q}}(w, W^i)\leq (r-2)X^{i,w}$. So it suffices to show that, with positive probability, $X^{i,w}\leq \rho^{3/2}n$ for all $1\leq i\leq q$ and all $w\in W^i$.

Fix $1\leq i\leq q$ and $w\in W^i$. Let $J^{i,w}$ be the set of indices $s\neq i$ such that $w\in W^s$; (iv) implies $|J^{i,w}|< 2k\rho n$. If $s\notin J^{i,w}\cup \{i\}$, then $w\notin W^s$ and $Y_s^{i,w}=0$. So
\begin{equation}
X^{i,w}\leq 1+\sum_{s\in J^{i,w}}Y_s^{i,w}.\label{eq:xju}
\end{equation}
Let $s_1< \dots< s_{|J^{i,w}|}$ be an enumeration of $J^{i,w}$. For any $b\leq |J^{i,w}|$, note that
$$d_{G_{s_b}}(w, W^i)\leq |W^i\cap W^{s_b}|\stackrel{\rm{(iii)}}{\leq} 2r\rho^2 n.$$
So at most $2r\rho^2n$ of the subgraphs $T'_j$ that we picked in $G_{s_b}'$ contain an edge incident to $w$ in $G_i$. Thus
$$\pr(Y_{s_b}^{i,w}=1\mid Y_{s_1}^{i,w}=y_1, \dots, Y_{s_{b-1}}^{i,w}=y_{b-1})\leq 2r\rho^2n/t\leq \rho^{1/2}/4k$$
for all $y_1, \dots, y_{b-1}\in \{0,1\}$ and $1\leq b\leq |J^{i,w}|$. Let $B\sim B(|J^{i,w}|, \rho^{1/2}/4k).$
Using Proposition~\ref{generalised-chernoff}, Lemma~\ref{lem:chernoff} and that $|J^{i,w}|\leq 2k\rho n$, we see that
\begin{align*}
\pr(X^{i,w}>\rho^{3/2}n)&\stackrel{\mathclap{\eqref{eq:xju}}}{\leq} \pr(\sum_{s\in J^{i,w}} Y_s^{i,w}> 3\rho^{3/2}n/4)\leq \pr(B>3\rho^{3/2}n/4)\\
&\leq \pr(|B-\ex(B)|>\rho^{3/2}n/4)\leq 2e^{-\rho^2n/16k}.
\end{align*}
There are at most $qrn\leq kr^2n^2$ pairs $(i,w)$, so there is a choice of $T_1, \dots, T_q$ such that $X^{i,w}\leq \rho^{3/2}n$ for all $1\leq i\leq q$ and all $w\in W^i$.
\end{proofof}

The following is an immediate consequence of Lemma~\ref{lem:pseud:bip}.

\begin{cor}\label{cor:pseud:trip}
Let $r\geq 2$ and $1/n \ll 1/k,1/r, \rho\leq 1$. Let $G$ be an $r$-partite graph on $(V_1, \dots, V_r)$ with $|V_1|=\dots=|V_r|=n$. Let $U, W\subseteq V(G)$ be disjoint with $|W_1|=\dots=|W_r|\geq \lfloor n/k \rfloor$. Suppose the following hold:
\begin{enumerate}[\rm(i)]
\item for all $1\leq j_1, j_2\leq r$ and all $x\in U\setminus(V_{j_1}\cup V_{j_2})$, $d_G(x, W_{j_1})=d_G(x, W_{j_2})$;
\item for all $1\leq j\leq r$ and all $x\in U\setminus U_j$, $\hat{\delta}(G[N_G(x, W)])\geq (1-1/(r-1))d_G(x, W_{j})+9kr\rho^{3/2}|W|$;
\item for all distinct $x, x'\in U$, $|N_G(x, W)\cap N_G(x', W)|\leq 2\rho^2 |W|$;
\item for all $y\in W$, $d_G(y,U)\leq 2k\rho |W_1|$.
\end{enumerate}
Then there exists $G_W\subseteq G[W]$ such that $G[U,W]\cup G_W$ has a $K_r$-decomposition and $\Delta(G_W)\leq 2kr\rho |W_1|$.
\end{cor}

\begin{proof}
Let $q:=|U|$ and let $u^1, \dots, u^{q}$ be an enumeration of $U$. For each $1\leq i\leq q$, let $W^{i}:=N_G(u^i, W)$. Note that $q\leq kr|W_1|$. Apply Lemma~\ref{lem:pseud:bip} (with $G[W]$ and $|W_1|$ playing the roles of $G$ and $n$) to obtain edge-disjoint perfect $K_{r-1}$-matchings $T^i$ in each $G[W^i]$. Let $G_W:=\bigcup_{i=1}^{q}T^i$. Then $G[U, W]\cup G_W$ has a $K_r$-decomposition. For each $y\in W$, we use (iv) to see that
$d_{G_W}(y)\leq (r-1)d_{G}(y, U)<2kr\rho |W_1|$.
\end{proof}

If we are given a $k$-partition $\cP$ of the $r$-partite graph $G$, we can apply Corollary~\ref{cor:pseud:trip} repeatedly with each $U\in \cP$ playing the role of $W$ to obtain the following result.

\begin{cor}\label{cor:pseud:parts}
Let $r\geq 2$ and $1/n \ll \rho \ll 1/k, 1/r\leq 1$. Let $G$ be an $r$-partite graph on $(V_1, \dots, V_r)$ with $|V_1|=\dots=|V_r|=n$. Let $\cP=\{U^1, \dots, U^k\}$ be a $k$-partition for $G$. Suppose that the following hold for all $2\leq i \leq k$:
\begin{enumerate}[\rm(i)]
\item for all $1\leq j_1, j_2\leq r$ and all $x\in U^{<i}\setminus(V_{j_1}\cup V_{j_2})$, $d_G(x, U^i_{j_1})=d_G(x, U^i_{j_2})$;\label{cor:pseud1}
\item for all $1\leq j\leq r$ and all $x\in U^{<i}\setminus V_j$, $\hat{\delta}(G[N_G(x, U^i)]\geq (1-1/(r-1))d_G(x, U^i_j)+9kr\rho^{3/2}|U^i|$;\label{cor:pseud2}
\item for all distinct $x,x'\in U^{<i}$, $|N_G(x, U^i)\cap N_G(x', U^i)|\leq 2\rho^2 |U^i|$;\label{cor:pseud3}
\item for all $y\in U^i$, $d_G(y,U^{<i})\leq 2k\rho |U^i_1|$.\label{cor:pseud4}
\end{enumerate}
Then there exists $G_0\subseteq G-G[\cP]$ such that $G[\cP]\cup G_0$ has a $K_r$-decomposition and $\Delta(G_0)\leq 3r\rho n$.
\end{cor}

\begin{proof}
For each $2\leq i\leq k$, let $G_i:=G[U^{<i}, U^i]\cup G[U^i]$. Apply Corollary~\ref{cor:pseud:trip} to each $G_i$ with $U^{<i}$, $U^i$ playing the roles of $U$, $W$ to obtain $G_i'\subseteq G[U^i]$ such that $G[U^{<i}, U^i]\cup G_i'$ has a $K_r$-decomposition and $\Delta(G_i')\leq 2kr\rho\lceil n/k\rceil\leq 3r\rho n$. Let $G_0:=\bigcup_{i=2}^k G_i'$ . Then $G[\cP]\cup G_0$ has a $K_r$-decomposition and $\Delta(G_0)\leq 3r\rho n$.
\end{proof}


\section{Balancing graph}\label{sec:bal}

In our proof we will consider a sequence of successively finer partitions $\cP_1, \dots, \cP_\ell$ in turn. When considering $\cP_i$, we will assume the leftover is a subgraph of $G-G[\cP_{i-1}]$ and aim to use Lemma~\ref{lem:degree} and then Corollary~\ref{cor:pseud:parts} to find copies of $K_r$ such that the leftover is now contained in $G-G[\cP_{i}]$ (i.e. inside the smaller partition classes). However, to apply Corollary~\ref{cor:pseud:parts} we need the leftover to be balanced with respect to the partition classes. In this section we show how this can be achieved.

Let $\cP=\{U^1, \dots, U^k\}$ be a $k$-partition of the vertex set $V=(V_1, \dots, V_r)$ with $|V_1|=\dots=|V_r|=n$. We say that a graph $H$ on $(V_1,\dots, V_r)$ is \emph{locally $\cP$-balanced} if
$$d_H(v, U^i_{j_1})=d_H(v, U^i_{j_2})$$
for all $1\leq i\leq k$, all $1\leq j_1, j_2\leq r$ and all $v\in U^i\setminus (V_{j_1}\cup V_{j_2})$.
Note that a graph which is locally $\cP$-balanced is not necessarily $K_r$-divisible but that $H[U^i]$ is $K_r$-divisible for all $1\leq i \leq k$.

Let $\gamma>0$. A \emph{$(\gamma, \cP)$-balancing graph} is a $K_r$-decomposable graph $B$ on $V$ such that the following holds.
Let $H$ be any $K_r$-divisible  graph on $V$ with:
\begin{enumerate}[(P1)]
	\item $e(H\cap B)=0$; \label{item:P1}
	\item $|d_H(v, U^i_{j_1})-d_H(v, U^i_{j_2})|<\gamma n$
for all $1\leq i \leq k$, all $1\leq j_1, j_2\leq r$ and all $v\notin V_{j_1}\cup V_{j_2}$.\label{item:P2}
\end{enumerate}
Then there exists $B'\subseteq B$ such that $B-B'$ has a $K_r$-decomposition and 
$$d_{H\cup B'}(v, U^i_{j_1})=d_{H\cup B'}(v, U^i_{j_2})$$
for all $2\leq i\leq k$, all $1\leq j_1, j_2\leq r$ and all $v\in U^{<i}\setminus (V_{j_1}\cup V_{j_2})$.

Our aim in this section will be to prove Lemma~\ref{lem:balance} which finds a $(\gamma, \cP)$-balancing graph in a suitable graph $G$.

\begin{lemma}\label{lem:balance}
Let $1/n \ll \gamma \ll \gamma'\ll 1/k\ll \eps\ll 1/r\leq 1/3$.
Let $G$ be an $r$-partite graph on $(V_1, \dots, V_r)$ with $|V_1|=\dots=|V_r|=n$. Let $\cP=\{U^1, \dots, U^k\}$ be a $k$-partition for $G$.
Suppose $d_{G}(v, U^i_j)\geq (1-1/(r+1)+\eps)|U^i_j|$ for all $1\leq i\leq k$, all $1\leq j\leq r$ and all $v\notin V_j$.
Then there exists $B\subseteq G$ which is a $(\gamma, \cP)$-balancing graph such that $B$ is locally $\cP$-balanced and $\Delta(B)<\gamma' n$.
\end{lemma}

The balancing graph $B$ will be made up of two graphs: $B_{\text{edge}}$, an edge balancing graph (which balances the total number of edges between appropriate classes), and $B_{\text{deg}}$, a degree balancing graph (which balances individual vertex degrees). These are described in Sections~\ref{sec:edge} and \ref{sec:degree} respectively.

\subsection{Edge balancing}\label{sec:edge}

Let $\cP=\{U^1, \dots, U^k\}$ be a $k$-partition of the vertex set $V=(V_1,\dots, V_r)$ with $|V_1|=\dots=|V_r|=n$. Let $\gamma>0$. A \emph{$(\gamma, \cP)$-edge balancing graph} is a $K_r$-decomposable graph $B_{\text{edge}}$ on $V$ such that the following holds.
Let $H$ be any $K_r$-divisible  graph on $V$ which is edge-disjoint from $B_{\text{edge}}$ and satisfies (P\ref{item:P2}).
Then there exists $B_{\text{edge}}'\subseteq B_{\text{edge}}$ such that $B_{\text{edge}}-B_{\text{edge}}'$ has a $K_r$-decomposition and 
$$e_{H\cup B_{\text{edge}}'}(U^{i_1}_{j_1}, U^{i_2}_{j_2})=e_{H\cup B_{\text{edge}}'}(U^{i_1}_{j_1}, U^{i_2}_{j_3})$$
for all $1\leq i_1<i_2\leq k$ and all $1\leq j_1, j_2, j_3\leq r$ with $j_1\neq j_2, j_3$.

In this section, we first construct and then find a $(\gamma, \cP)$-edge balancing graph in $G$.

For any multigraph $G$ on $W$ and any $e\in W^{(2)}$, let $m_G(e)$ be the multiplicity of the edge $e$ in $G$. We say that a $K_r$-divisible multigraph $G$ on $W=(W_1, \dots, W_r)$ is \emph{irreducible} if $G$ has no non-trivial $K_r$-divisible proper subgraphs; that is, for every $H \subsetneq G$ with $e(H)>0$, $H$ is not $K_r$-divisible. It is easy to see that there are only finitely many irreducible $K_r$-divisible multigraphs on $W$. In particular, this implies the following proposition.

\begin{prop}\label{prop:irred}
Let $r\in \N$ and let $W=(W_1, \dots, W_r)$. 
Then there exists $N=N(W)$ such that every irreducible $K_r$-divisible multigraph on $W$ has edge multiplicity at most $N$.\qed
\end{prop}

Let $\cP=\{U^1, \dots, U^k\}$ be a partition of $V=(V_1,\dots, V_r)$. Take a copy $K$ of $K_r(k)$ with vertex set $(W_1, \dots, W_r)$ where $W_j=\{w^1_j, \dots, w^k_j\}$ for each $1\leq j\leq r$. For each $1\leq i \leq k$, let $W^i:=\{w^i_j: 1\leq j\leq r\}$. Given a graph $H$ on $V$, we define an \emph{excess multigraph} $\text{EM}(H)$ on the vertex set $V(K)$ as follows. Between each pair of vertices $w^{i_1}_{j_1}$, $w^{i_2}_{j_2}$ such that $w^{i_1}_{j_1}w^{i_2}_{j_2}\in E(K)$ there are exactly
$$e_H(U^{i_1}_{j_1},U^{i_2}_{j_2})-\min\{e_H(U^{i_1}_j,U^{i_2}_{j'}): 1\leq j,j'\leq r, j\neq j'\}$$
multiedges in $\text{EM}(H)$.

\begin{prop}\label{prop:redgraph}
Let $r\in \N$ with $r\geq 3$. Let $\cP=\{U^1, \dots, U^k\}$ be a $k$-partition of the vertex set $V=(V_1, \dots, V_r)$ with $|V_1|=\dots=|V_r|=n$. Let $H$ be any $K_r$-divisible graph on $V$ satisfying {\rm(P\ref{item:P2})}. Then the excess multigraph $\textnormal{EM}(H)$ has a decomposition into at most $3\gamma k^2r^2n^2$ irreducible $K_r$-divisible multigraphs.
\end{prop}

\begin{proof}
First, note that for any $1\leq i_1, i_2\leq k$, any $1\leq j_1,j_2,j_3\leq r$ with $j_1\neq j_2, j_3$ and any $v\in U^{i_1}_{j_1}$, we have
$|d_H(v, U^{i_2}_{j_2})-d_H(v, U^{i_2}_{j_3})|<\gamma n$ by (P\ref{item:P2}). Therefore,
\begin{equation}
|e_H(U^{i_1}_{j_1}, U^{i_2}_{j_2})-e_H(U^{i_1}_{j_1}, U^{i_2}_{j_3})|<\gamma n|U^{i_1}_{j_1}|<\gamma n^2.\label{eq:gamma2}
\end{equation}

We claim that, for all $w^{i_1}_{j_1}w^{i_2}_{j_2}\in E(K)$,
\begin{equation}
m_{\text{EM}(H)}(w^{i_1}_{j_1}w^{i_2}_{j_2})< 3\gamma n^2.
\label{eq:redgraph}
\end{equation} 
Let $1\leq j_1', j_2'\leq r$ with $j_1'\neq j_2'$.
Let $1\leq j\leq r$ with $j\neq j_1, j_1'$. Then
\begin{align*}
|e_H(U^{i_1}_{j_1}, U^{i_2}_{j_2})- e_H(U^{i_1}_{j_1'}, U^{i_2}_{j_2'})|&\leq 
|e_H(U^{i_1}_{j_1}, U^{i_2}_{j_2})- e_H(U^{i_1}_{j_1}, U^{i_2}_j)|+|e_H(U^{i_1}_{j_1},U^{i_2}_j)- e_H(U^{i_1}_{j_1'},U^{i_2}_j)|
\\
&\hspace{1cm} +
|e_H(U^{i_1}_{j_1'},U^{i_2}_j)-e_H(U^{i_1}_{j_1'}, U^{i_2}_{j_2'})|\\
&\stackrel{\mathclap{\eqref{eq:gamma2}}}{<}3\gamma n^2.
\end{align*}
So \eqref{eq:redgraph} holds.

We will now show that $\text{EM}(H)$ is $K_r$-divisible. Consider any vertex $w^{i_1}_{j_1}\in V(\text{EM}(H))$ and any $1\leq j_2, j_3\leq r$ such that $j_1\neq j_2, j_3$. Note that, since $H$ is $K_r$-divisible,
\begin{align*}
d_{\text{EM}(H)}(w^{i_1}_{j_1}, W_{j_2})&=\sum_{i=1}^k m_{\text{EM}(H)}(w^{i_1}_{j_1}, w^{i}_{j_2})\\
&=e_H(U^{i_1}_{j_1},V_{j_2})-\sum_{i=1}^k \min\{e_H(U^{i_1}_{j},U^{i}_{j'}): 1\leq j,j'\leq r, j\neq j'\}\\
&=e_H(U^{i_1}_{j_1},V_{j_3})-\sum_{i=1}^k \min\{e_H(U^{i_1}_{j},U^{i}_{j'}): 1\leq j,j'\leq r, j\neq j'\}\\
&=\sum_{i=1}^k m_{\text{EM}(H)}(w^{i_1}_{j_1}, w^{i}_{j_3}) = d_{\text{EM}(H)}(w^{i_1}_{j_1}, W_{j_3}).
\end{align*}
So $\text{EM}(H)$ is $K_r$-divisible and therefore has a decomposition $\cF$ into irreducible $K_r$-divisible multigraphs. By \eqref{eq:redgraph}, there are at most $3\gamma n^2$ edges between any pair of vertices in $\text{EM}(H)$, so $|\cF|\leq (3\gamma n^2)e(K)<3\gamma k^2r^2n^2$.
\end{proof}

Recall that $K$ denotes a copy of $K_r(k)$ with vertex set $V(K)=(W_1, \dots, W_r)$ (see the paragraph after Proposition~\ref{prop:irred}).
Let $N=N(V(K))$ be the maximum multiplicity of an edge in any irreducible $K_r$-divisible multigraph on $V(K)$ ($N$ exists by Proposition~\ref{prop:irred}). Label each vertex $w^i_j$ of $K$ by $U^i_j$. Let $K(N)$ be the labelled multigraph obtained from $K$ by replacing each edge of $K$ by $N$ multiedges.

We now construct a $\cP$-labelled graph which resembles the multigraph $K(N)$ (when we compare relative differences in the numbers of edges between vertices) and has lower degeneracy. Consider any edge $e=w^{i_1}_{j_1}w^{i_2}_{j_2}\in E(K(N))$. Let $\theta(e)$ be the graph obtained by the following procedure. Take a copy $K_e$ of $K[W^{i_1}, W^{i_2}]-w^{i_1}_{j_1}w^{i_2}_{j_2}$ ($K_e$ inherits the labelling of $K[W^{i_1}, W^{i_2}]$). Note that $K[W^{i_1}, W^{i_2}]$ is a copy of $K_r$ if $i_1=i_2$ and a copy of the graph obtained from $K_{r,r}$ by deleting a perfect matching otherwise. Join $w^{i_1}_{j_1}$ to the copy of $w^{i_2}_{j_2}$ in $K_e$ and join $w^{i_2}_{j_2}$ to the copy of $w^{i_1}_{j_1}$ in $K_e$. Write $\theta(e)$ for the resulting $\cP$-labelled graph (so the vertex set of $\theta(e)$ consists of $w^{i_1}_{j_1}$, $w^{i_2}_{j_2}$ as well as all the vertices in $K_e$).
Choose the graphs $K_e$ to be vertex-disjoint for all $e\in E(K(N))$.
For any $K'\subseteq K(N)$, let $\theta(K'):=\bigcup\{\theta(e): e\in E(K')\}$.

To see that the labelling of $\theta(K(N))$ is actually a $\cP$-labelling, note that for any $U^i_j$, the set of vertices labelled $U^i_j$ forms an independent set in $\theta(K(N))$. Moreover, note that
\begin{equation}\label{eq:degK(N)}
\theta(K(N)) \text{ has degeneracy } r-1.
\end{equation}
To see this, list its vertices in the following order. First list all the original vertices of $V(K)$. These form an independent set in $\theta(K(N))$. Then list the remaining vertices of $\theta(K(N))$ in any order. Each of these vertices has degree $r-1$ in $\theta(K(N))$, so the degeneracy of $\theta(K(N))$ is $r-1$.

\begin{prop}\label{prop:edge}
Let $\cP=\{U^1, \dots, U^k\}$ be a $k$-partition of the vertex set $V=(V_1, \dots, V_r)$ with $|V_1|=\dots=|V_r|=n$. Let $J=\phi(\theta(K(N)))$ be a copy of $\theta(K(N))$ on $V$ which is compatible with its $\cP$-labelling. Then the following hold:
\begin{enumerate}[\rm(i)]
	\item $J$ is $K_r$-divisible and locally $\cP$-balanced;\label{prop:edge1}
	\item for any multigraph $H\subseteq K(N)$, any $1\leq i_1,i_2\leq k$ and any $1\leq j_1< j_2\leq r$, \label{prop:edge2}
	\begin{equation*}
	e_{\phi(\theta(H))}(U^{i_1}_{j_1}, U^{i_2}_{j_2})=
	e_H(W^{i_1}, W^{i_2})+m_H(w^{i_1}_{j_1}w^{i_2}_{j_2}).%
	\COMMENT{$m_H(w^{i_1}_{j_1}w^{i_2}_{j_2})$ is edge multiplicity of $w^{i_1}_{j_1}w^{i_2}_{j_2}$ in $H$}
	\end{equation*}
\end{enumerate}
\end{prop}

\begin{proof}
We first prove that $J$ is $K_r$-divisible. Consider any $x\in V(\theta(K(N)))$. If $x=w^i_j\in V(K)$, then $d_J(\phi(x), V_{j_1})=Nk$ for all $1\leq j_1\leq r$ with $j_1\neq j$ (since for each edge $w^i_jw^{i_1}_{j_1}\in E(K)$, $x$ has exactly $N$ neighbours labelled $U^{i_1}_{j_1}$ in $\theta(K(N))$). If $x\notin V(K)$, $x$ must appear in a copy of $K_e$ in $\theta(e)$ for some edge $e\in E(K(N))$. In this case, $d_J(\phi(x), V_j)=1$ for all $1\leq j\leq r$ such that $\phi(x)\notin V_j$. So $J$ is $K_r$-divisible.

To see that $J$ is locally $\cP$-balanced, consider any $x\in V(\theta(K(N)))$. If $x=w^i_j\in V(K)$, then $\phi(x)\in U^i_j$ and $d_J(\phi(x), U^i_{j_1})=N$ for all $1\leq j_1 \leq r$ with $j_1\neq j$. Otherwise, $x$ must appear in a copy of $K_e$ in $\theta(e)$ for some edge $e=w^{i_1}_{j_1}w^{i_2}_{j_2}\in E(K(N))$. Let $i,j$ be such that $\phi(x)\in U^{i}_{j}$ (so $i\in \{i_1, i_2\}$). If $i_1\neq i_2$, then $d_J(\phi(x), U^i_{j'})=0$ for all $1\leq j'\leq r$.  If $i_1=i_2$, then $d_J(\phi(x), U^{i}_{j'})=1$ for all $1\leq j'\leq r$ with $j'\neq j$. So $J$ is locally $\cP$-balanced. Thus \eqref{prop:edge1} holds.

We now prove \eqref{prop:edge2}. Let $1\leq i_1, i_2\leq k$ and $1\leq j_1< j_2\leq r$. Consider any edge $w^i_jw^{i'}_{j'}\in E(K(N))$. The $\cP$-labelling of $\theta(K(N))$ gives
	\begin{equation}
	e_{\phi(\theta(w^i_jw^{i'}_{j'}))}(U^{i_1}_{j_1}, U^{i_2}_{j_2})=
	\begin{cases}
	0 & \text{if}\ \{i, i'\}\neq \{i_1, i_2\},\\
	2 & \text{if}\ \{(i,j), (i',j')\}= \{(i_1, j_1), (i_2, j_2)\},\\
	1 & \text{otherwise}.
	\end{cases}
	\label{eq:edgecounter}
	\end{equation}
Let $H\subseteq K(N)$. Then \eqref{prop:edge2} follows from applying \eqref{eq:edgecounter} to each edge in $H$.
\end{proof}

The following proposition allows us to use a copy of $\theta(K(N))$ to correct imbalances in the number of edges between parts $U^{i_1}_{j_1}$ and $U^{i_2}_{j_2}$ when $\text{EM}(H)$ is an irreducible $K_r$-divisible multigraph.

\begin{prop}\label{prop:edgebalancer1}
Let $\cP=\{U^1, \dots, U^k\}$ be a $k$-partition of the vertex set $V=(V_1, \dots, V_r)$ with $|V_1|=\dots=|V_r|=n$. Let $H$ be a graph on $V$ such that $\textnormal{EM}(H)=I$ is an irreducible $K_r$-divisible multigraph. Let $J=\phi(\theta(K(N)))$ be a copy of $\theta(K(N))$ on $V$  which is compatible with its $\cP$-labelling and edge-disjoint from $H$. Then there exists $J'\subseteq J$ such that $J-J'$ is $K_r$-divisible and $H':=H\cup J'$ satisfies
\begin{equation}\label{prop:edgebalancer1:1}
e_{H'}(U^{i_1}_{j_1}, U^{i_2}_{j_2})=e_{H'}(U^{i_1}_{j_1}, U^{i_2}_{j_3})
\end{equation}
for all $1\leq i_1<i_2\leq k$ and all $1\leq j_1, j_2, j_3\leq r$ with $j_1\neq j_2, j_3$.
\end{prop}

\begin{proof}
Recall that $N$ denotes the maximum multiplicity of an edge in an irreducible $K_r$-divisible multigraph on $V(K)$. So we may view $I$ as a subgraph of $K(N)$. Let $J':=J-\phi(\theta(I))$. For all $1\leq i_1<i_2\leq k$, let 
$$p_{i_1, i_2}:= \min\{ e_H(U^{i_1}_{j_1},U^{i_2}_{j_2}): 1\leq j_1, j_2, \leq r, j_1\neq j_2\}.$$

Proposition~\ref{prop:edge} gives, for all $1\leq i_1<i_2\leq k$ and all $1\leq j_1, j_2\leq r$ with $j_1\neq j_2$,
	\begin{align*}
	e_{J'}(U^{i_1}_{j_1}, U^{i_2}_{j_2})&=e_{\phi(\theta(K(N)))}(U^{i_1}_{j_1}, U^{i_2}_{j_2})-e_{\phi(\theta(I))}(U^{i_1}_{j_1}, U^{i_2}_{j_2})\\
	&=e_{K(N)}(W^{i_1}, W^{i_2})+N-(e_{I}(W^{i_1}, W^{i_2})+m_I(w^{i_1}_{j_1}w^{i_2}_{j_2}))\\
	&=e_{K(N)-I}(W^{i_1}, W^{i_2})+N-m_I(w^{i_1}_{j_1}w^{i_2}_{j_2}).
	\end{align*}
Recall that $I=\text{EM}(H)$, so $e_{H}(U^{i_1}_{j_1}, U^{i_2}_{j_2})=m_I(w^{i_1}_{j_1}w^{i_2}_{j_2})+p_{i_1, i_2}$
and
\begin{equation*}
	e_{H'}(U^{i_1}_{j_1}, U^{i_2}_{j_2})=e_H(U^{i_1}_{j_1}, U^{i_2}_{j_2})+e_{J'}(U^{i_1}_{j_1}, U^{i_2}_{j_2})=e_{K(N)-I}(W^{i_1}, W^{i_2})+N+p_{i_1, i_2}.
\end{equation*}
Note that the right hand side is independent of $j_1, j_2$. 
Thus \eqref{prop:edgebalancer1:1} holds.
\end{proof}

The following proposition describes a $(\gamma,\cP)$-edge balancing graph based on the construction in Propositions~\ref{prop:edge}~and~\ref{prop:edgebalancer1}

\begin{prop}\label{prop:edgebalancer}
Let $k,r\in \N$ with $r\geq 3$. Let $\cP=\{U^1, \dots, U^k\}$ be a $k$-partition of the vertex set $V=(V_1, \dots, V_r)$ with $|V_1|=\dots=|V_r|=n$. Let $J_1, \dots, J_\ell$ be a collection of $\ell\geq 3\gamma k^2r^2n^2$ copies of $\theta(K(N))$ on $V$ which are compatible with their labellings. Let $\{A_1, \dots, A_m\}$ be an absorbing set for $J_1, \dots, J_\ell$ on $V$. Suppose that $J_1, \dots, J_\ell, A_1, \dots, A_m$ are edge-disjoint. Then $B_{\textnormal{edge}}:=\bigcup_{i=1}^\ell J_i\cup \bigcup_{i=1}^m A_i$ is a $(\gamma, \cP)$-edge balancing graph.
\end{prop}

\begin{proof}
Let $H$ be any $K_r$-divisible graph on $V$ which is edge-disjoint from $B_{\text{edge}}$ and satisfies (P\ref{item:P2}). Apply Proposition~\ref{prop:redgraph} to find a decomposition of $\text{EM}(H)$ into a collection $\cI=\{I_1, \dots, I_{\ell'}\}$ of irreducible $K_r$-divisible multigraphs, where $\ell'\leq 3\gamma k^2r^2n^2\leq \ell$. If $\ell'=0$, let $B_{\text{edge}}'\subseteq B_{\text{edge}}$ be the empty graph. If $\ell'>0$, we proceed as follows to find $B_{\text{edge}}'$. 
Let $H_1, \dots, H_{\ell'}$ be graphs on $V$ which partition the edge set of $H$ and satisfy $\text{EM}(H_s)=I_s$ for each $1\leq s\leq \ell'$. (To find such a partition, for each $1\leq s<\ell'$ form $H_s$ by taking one $U^{i_1}_{j_1}U^{i_2}_{j_2}$-edge from $H$ for each edge $w^{i_1}_{j_1}w^{i_2}_{j_2}$ in $I_s$. Let $H_{\ell'}$ consist of all the remaining edges.)

Apply Proposition~\ref{prop:edgebalancer1} for each $1\leq s\leq \ell'$ with $H_s$ and $J_s$ playing the roles of $H$ and $J$ to find $J_s'\subseteq J_s$ such that
$J_s-J_s'$ is $K_r$-divisible and $H_s':=H_s\cup J_s'$ satisfies
\begin{equation}
e_{H_s'}(U^{i_1}_{j_1}, U^{i_2}_{j_2})=e_{H_s'}(U^{i_1}_{j_1}, U^{i_2}_{j_3})\label{eq:Hi_bal}
\end{equation}
for all $1\leq i_1< i_2\leq k$ and all $1\leq j_1, j_2, j_3\leq r$ with $j_1\neq j_2, j_3$.
Let $B_{\text{edge}}':=\bigcup_{s=1}^{\ell'}J_s'$. Then \eqref{eq:Hi_bal} implies that the graph $H':=H\cup B_{\text{edge}}'=\bigcup_{s=1}^{\ell'} H_s'$ satisfies
$$e_{H'}(U^{i_1}_{j_1}, U^{i_2}_{j_2})=e_{H'}(U^{i_1}_{j_1}, U^{i_2}_{j_3})$$
for all $1\leq i_1< i_2\leq k$ and all $1\leq j_1, j_2, j_3\leq r$ with $j_1\neq j_2, j_3$.

We now check that $B_{\text{edge}}$ and $B_{\text{edge}}-B_{\text{edge}}'$ are $K_r$-decomposable. Recall that every absorber $A_i$ is $K_r$-decomposable. Also recall that, for every $1\leq s \leq \ell$, $J_s$ is $K_r$-divisible, by Proposition~\ref{prop:edge}\eqref{prop:edge1}. Since $\{A_1, \dots, A_m\}$ is an absorbing set, it contains a distinct absorber for each $J_s$. So for each $1\leq s \leq \ell$, there exists a distinct $1\leq i_s\leq m$ such that $A_{i_s}\cup J_s$ has a $K_r$-decomposition. Therefore $B_{\text{edge}}$ is $K_r$-decomposable. To see that $B_{\text{edge}}-B_{\text{edge}}'$ is $K_r$-decomposable, recall that for each $1\leq s \leq \ell'$, $J_s-J_s'$ is a $K_r$-divisible subgraph of $J_s$. So for each $1\leq s \leq \ell$, there exists a distinct $1\leq j_s\leq m$ such that, if $s\leq \ell'$, $A_{j_s}\cup (J_s-J_s')$ has a $K_r$-decomposition and, if $s>\ell'$, $A_{j_s}\cup J_s$ has a $K_r$-decomposition. So we can find a $K_r$-decomposition of
$$B_{\text{edge}}-B_{\text{edge}}'=\bigcup_{s=1}^{\ell'}(J_s-J_s')\cup \bigcup_{s=\ell'+1}^\ell J_s\cup \bigcup_{s=1}^m A_m.$$
Therefore, $B_{\text{edge}}$ is a $(\gamma, \cP)$-edge balancing graph.
\end{proof}

The next proposition finds a copy of this $(\gamma, \cP)$-edge balancing graph in $G$.

\begin{prop}\label{prop:PE}
Let $1/n \ll \gamma \ll \gamma'\ll 1/k\ll \eps\ll 1/r \leq 1/3$.
Let $G$ be an $r$-partite graph on $(V_1, \dots, V_r)$ with $|V_1|=\dots=|V_r|=n$. Let $\cP=\{U^1, \dots, U^k\}$ be a $k$-partition for $G$. 
Suppose that $d_G(v, U^i_j)\geq (1-1/(r+1)+\eps)|U^i_j|$ for all $1\leq i \leq k$, all $1\leq j\leq r$ and all $v\notin V_j$.
Then there exists a $(\gamma, \cP)$-edge balancing graph $B_{\textnormal{edge}}\subseteq G$ such that $B_{\textnormal{edge}}$ is locally $\cP$-balanced and $\Delta(B_{\textnormal{edge}})<\gamma' n$.
\end{prop}

\begin{proof} Let $\gamma_1$ be such that $\gamma\ll \gamma_1 \ll \gamma'$. Recall from (\ref{eq:degK(N)}) that $\theta(K(N))$ is a $\cP$-labelled graph with degeneracy $r-1$ and all vertices of $\theta(K(N))$ are free vertices. Also,%
\COMMENT{$|\theta(K(N))|\leq |K|+2re(K)N$. First term accounts for the $|K|$ original vertices of $K(N)$. Each edge of $K(N)$ contributes a copy of $K_r$ or $(K_{r,r}-\text{perfect matching})$ to $\theta(K(N))$, so at most $2r$ new vertices. Since $e(K(N))=e(K)N$, we have added at most $2re(K)N$ new vertices.}
$$|\theta(K(N))|\leq |K|+2re(K)N= kr+2rk^2\binom{r}{2}N\leq k^2r^3N.$$
Let $\ell:= \lceil 3\gamma k^2r^2n^2\rceil\leq \gamma^{1/2}n^2$. We can apply Lemma~\ref{lem:emb} (with $\gamma^{1/2}$, $\gamma_1$, $r-1$, $k^2r^3N$ playing the roles of $\eta$, $\eps$, $d$, $b$ and with each $H_i$ being a copy of $\theta(K(N))$) to find edge-disjoint copies $J_1, \dots, J_\ell$ of $\theta(K(N))$ in $G$ which are compatible with their labellings and satisfy $\Delta(\bigcup_{i=1}^\ell J_i)\leq \gamma_1n$.

Let $G':=G[\cP]-\bigcup_{i=1}^\ell J_i$ and note that
$$\hat\delta(G')\geq (1-1/(r+1)+\eps)n-\lceil n/k \rceil-\gamma_1 n\geq (1-1/(r+1)+\gamma')n.$$
Apply Lemma~\ref{lem:absorbset} (with $\gamma_1$, $\gamma'/2$, $k^2r^3N$ and $G'$ playing the roles of $\eta$, $\eps$, $b$ and $G$) to find an absorbing set $\cA$ for $J_1, \dots, J_{\ell}$ in $G'$ such that $\Delta(\bigcup \cA)\leq \gamma'n/2$.

Let $B_{\text{edge}}:=\bigcup_{i=1}^\ell J_i\cup \bigcup \cA$. Then $B_{\text{edge}}$ is a $(\gamma, \cP)$-edge balancing graph by Proposition~\ref{prop:edgebalancer}. Also, $\Delta(B_{\text{edge}}) < \gamma'n$. Note that for each $1\leq i \leq k$, $B_{\text{edge}}[U^i]=\bigcup_{s=1}^\ell J_s[U^i]$ (this is the reason for finding $\cA$ in $G[\cP]$). Moreover, each $J_s$ is locally $\cP$-balanced by Proposition~\ref{prop:edge}\eqref{prop:edge1}. Therefore $B_{\text{edge}}$ is also locally $\cP$-balanced.
\end{proof}

\subsection{Degree balancing}\label{sec:degree}

Let $\cP=\{U^1, \dots, U^k\}$ be a $k$-partition of the vertex set $V=(V_1, \dots, V_r)$ with $|V_1|=\dots=|V_r|=n$. Let $\gamma>0$. A \emph{$(\gamma, \cP)$-degree balancing graph} is a $K_r$-decomposable graph $B_{\text{deg}}$ on $V$ such that the following holds.
Let $H$ be any $K_r$-divisible graph on $V$ satisfying:
\begin{enumerate}[(Q1)]
\item $e(H\cap B_{\text{deg}})=0$;\label{item:Q0}
\item $e_{H}(U^{i_1}_{j_1}, U^{i_2}_{j_2})=e_{H}(U^{i_1}_{j_1}, U^{i_2}_{j_3})$
for all $1\leq i_1<i_2\leq k$ and all $1\leq j_1, j_2, j_3\leq r$ with $j_1\neq j_2, j_3$;\label{item:Q1}
\item $|d_H(v, U^i_{j_2})-d_H(v, U^i_{j_3})|<\gamma |U^i_{j_1}|$ for all $2\leq i\leq k$, all $1\leq j_1, j_2, j_3\leq r$ with $j_1\neq j_2, j_3$ and all $v\in U^{<i}_{j_1}$.\label{item:Q2}
\end{enumerate}
Then there exists $B_{\text{deg}}'\subseteq B_{\text{deg}}$ such that $B_{\text{deg}}-B_{\text{deg}}'$ has a $K_r$-decomposition and $$d_{H\cup B_{\text{deg}}'}(v, U^i_{j_1})=d_{H\cup B_{\text{deg}}'}(v, U^i_{j_2})$$
for all $2\leq i\leq k$, all $1\leq j_1, j_2\leq r$ and all $v\in U^{<i}\setminus(V_{j_1}\cup V_{j_2})$.

We will build a degree balancing graph by combining smaller graphs which correct the degrees between two parts of the partition at a time. So, let us assume that the partition has only two parts, i.e., let $\cP=\{U^1, U^2\}$ partition the vertex set $V=(V_1, \dots, V_r)$. We begin by defining those graphs which will form the basic gadgets of the degree balancing graph.  Let $D_0$ be a copy of $K_r(3)$ with vertex classes $\{w^i_j:1\leq i \leq 3\}$ for $1\leq j\leq r$. For each $1\leq i \leq 3$, let $W^i:=\{w^i_j: 1\leq j\leq r\}$. We define a labelling $L:V(D_0)\rightarrow \{U^{1}_j, U^{2}_j:1\leq j\leq r\}$ as follows:
\begin{equation*}
L(w^{i}_{j})=
	\begin{cases}
	U^{1}_{j} &\text{if}\ i=1,2,\\
	U^{2}_{j} &\text{if}\ i=3.
	\end{cases}
\end{equation*}
Suppose that $x,y$ are distinct vertices in $U^{1}_{j_1}$ where $1\leq j_1\leq r$.
Obtain the $\cP$-labelled graph $D_{x,y}$%
\COMMENT{$D_{x,y}$ provides a locally $\cP$-balanced graph to contain each $D^{j_2}_{x\rightarrow y}$.}
by taking the labelled copy of $D_0$ and changing the label of $w^1_{j_1}$ to $\{x\}$ and $w^2_{j_1}$ to $\{y\}$. 
Let $1\leq j_2\leq r$ be such that $j_2\neq j_1$. Let $D^{j_2}_{x\rightarrow y}$
be the $\cP$-labelled subgraph of $D_{x,y}$ which has as its vertex set
$$W^1 \cup \{w^2_{j_1}\}\cup (W^3\setminus\{w^3_{j_1}\}),$$
contains all possible edges in $W^1\setminus\{w^1_{j_1}\}$, all possible edges in $W^3\setminus\{w^3_{j_1}\}$, all edges of the form $w^1_{j_1}w^3_j$ and $w^1_jw^2_{j_1}$ where $1\leq j\leq r$ and $j\neq j_1, j_2$, as well as the edges $w^1_{j_1}w^1_{j_2}$ and $w^2_{j_1}w^3_{j_2}$.  (Note that if we were to identify the vertices $w^1_{j_1}$ and $w^2_{j_1}$ we would obtain two copies of $K_r$ which have only one vertex in common.)

\begin{figure}[h]
\includegraphics[scale=0.5]{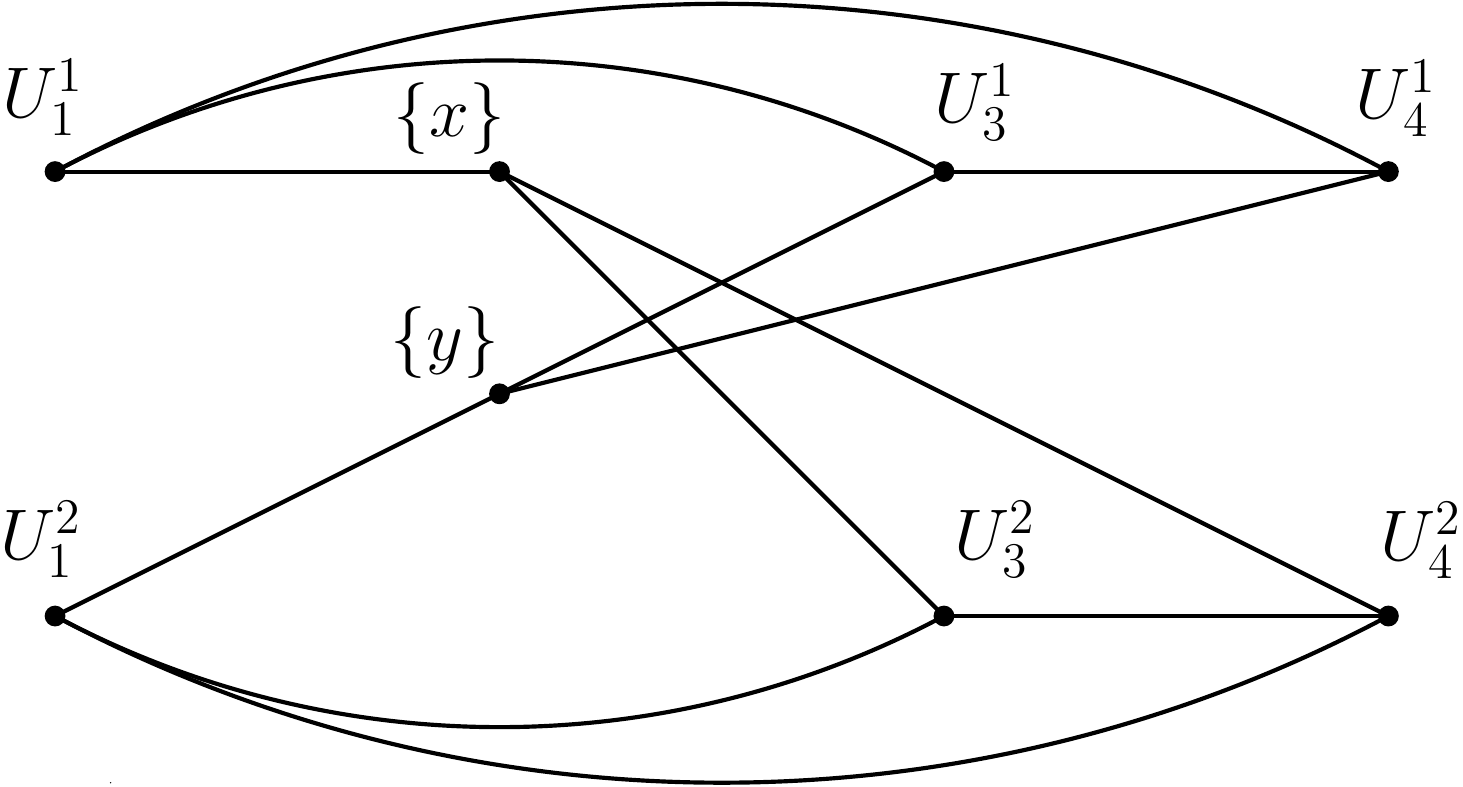}
\caption{A copy of $D^1_{x\rightarrow y}$ when $r=4$ and $x,y\in U^1_2$.}
\end{figure}

As in Section~\ref{sec:edge}, we would like to reduce the degeneracy of $D_{x,y}$. The operation $\theta$ (which will be familiar from Section~\ref{sec:edge}) replaces each edge of $D_{x,y}$ by a $\cP$-labelled graph as follows. Consider any edge $e=w^{i_1}_{j_3}w^{i_2}_{j_4}\in E(D_{x,y})$.
Take a labelled copy $D_e$ of $D_0[W^{i_1}, W^{i_2}]-w^{i_1}_{j_3}w^{i_2}_{j_4}$ ($D_e$ inherits the labelling of $D_0[W^{i_1}, W^{i_2}]$). Note that $D_0[W^{i_1}, W^{i_2}]$ is a copy of $K_r$ if $i_1=i_2$ and a copy of the graph obtained from $K_{r,r}$ by deleting a perfect matching otherwise. Join $w^{i_1}_{j_3}$ to the copy of $w^{i_2}_{j_4}$ in $D_e$ and join $w^{i_2}_{j_4}$ to the copy of $w^{i_1}_{j_3}$ in $D_e$ (so the vertex set of $\theta(e)$ consists of $w^{i_1}_{j_3}$, $w^{i_2}_{j_4}$ as well as all the vertices in $D_e$). Write $\theta(e)$ for the resulting $\cP$-labelled graph.
Choose the graphs $D_e$ to be vertex-disjoint for all $e\in E(D_{x,y})$. 
For any $D'\subseteq D_{x,y}$, let $\theta(D'):=\bigcup\{\theta(e): e\in E(D')\}$. The graph $\theta(D_{x,y})$ has the following properties:
\begin{enumerate}[($\theta$1)]
	\item $|\theta(D_{x,y})|\leq 3r+2r3^2\binom{r}{2}\leq 10r^3$ (since we add at most $2re(K_r(3))$ new vertices to obtain $\theta(D_{x,y})$ from $D_{x,y}$); \label{theta1}
	\item $\theta(D_{x,y})$ has degeneracy $r-1$ (to see this, take the original vertices of $D_{x,y}$ first, followed by the remaining vertices in any order). \label{theta2}
\end{enumerate}

Suppose that $H$ is a graph on $V$ and  $x,y\in U^1_{j_1}$. Suppose that $d_H(x, U^2_{j_2})$ is currently too large and $d_H(y, U^2_{j_2})$ is too small. The next proposition allows us to use copies of $\theta(D_{x\rightarrow y}^{j_2})$ to `transfer' some of this surplus from $x$ to $y$.

\begin{prop}\label{prop:mover}
Let $\cP=\{U^1, U^2\}$ be a partition of the vertex set $V=(V_1, \dots, V_r)$. Let $1\leq j_1, j_2\leq r$ with $j_1\neq j_2$ and suppose $x,y\in U^1_{j_1}$. Suppose that $D_1=\phi(\theta(D_{x,y}))$ is a copy of $\theta(D_{x,y})$ on $V$ which is compatible with its labelling. Let $D_2:=\phi(\theta(D_{x\rightarrow y}^{j_2})) \subseteq D_1$. Then the following hold:
\begin{enumerate}[\rm(i)]
		\item both $D_1$ and $D_2$ are $K_r$-divisible;\label{item:mover1}
		\item $D_1$ is locally  $\cP$-balanced;\label{item:mover3}
		\item for any $1\leq j_3, j_4\leq r$ with $j_4\neq j_2$ and any $v\in U^1\setminus (V_{j_3}\cup V_{j_4})$,\label{item:mover2}
		\begin{equation*}
			d_{D_2}(v, U^2_{j_3})-d_{D_2}(v, U^2_{j_4})=
			\begin{cases}
			-1& \text{if}\ v=x \text{ and } j_3=j_2,\\
			1&\text{if}\ v=y\text{ and } j_3=j_2,\\
			0& \text{otherwise.}
			\end{cases}
		\end{equation*}
\end{enumerate}
\end{prop}

\begin{proof}
First we show that \eqref{item:mover1} holds. Consider any $v\in V(\theta(D_{x,y}))$. If $v\in V(D_{x,y})$, then $d_{D_1}(\phi(v),V_{j})=3$ for all $1\leq j\leq r$ such that $\phi(v) \notin V_j$. Otherwise, $v$ appears in a copy of $D_e$ for some edge $e\in E(D_{x,y})$ and $d_{D_1}(\phi(v),V_{j})=1$ for all $1\leq j\leq r$ such that  $\phi(v)\notin V_{j}$. So $D_1$ is $K_r$-divisible. For $D_2$, consider any $v\in V(\theta(D_{x\rightarrow y}^{j_2}))$. If $v\in V(D_{x\rightarrow y}^{j_2})$, then $d_{D_2}(\phi(v),V_{j})=1$ for all $1\leq j\leq r$ with $\phi(v) \notin V_j$. Otherwise, $v$ appears in a copy of $D_e$ for some edge $e\in E(D_{x\rightarrow y}^{j_2})$ and $d_{D_2}(\phi(v),V_{j})=1$ for all $1\leq j\leq r$ such that $\phi(v)\notin V_{j}$. So $D_2$ is $K_r$-divisible.

For \eqref{item:mover3}, consider any $v\in V(\theta(D_{x,y}))$. First suppose $v=w^i_j\in V(D_{x,y})$. If $i=1,2$, then $\phi(v)\in U^1_j$ and $d_{D_1}(\phi(v),U^1_{j'})=2$ for all $1\leq j'\leq r$ with $j'\neq j$.  If $i=3$, then $\phi(v)\in U^2_j$ and $d_{D_1}(\phi(v),U^2_{j'})=1$ for all $1\leq j'\leq r$ with $j'\neq j$.   Otherwise, $v$ must appear in a copy of $D_e$ in $\theta(e)$ for some edge $e=w^{i_1}_{j_1}w^{i_2}_{j_2}\in E(D_{x,y})$. Let $i,j$ be such that $\phi(v)\in U^{i}_{j}$. If $i_1, i_2\in \{1,2\}$ or if $i_1=i_2=3$, then $d_{D_1}(\phi(v), U^{i}_{j'})=1$ for all $1\leq j'\leq r$ with $j'\neq j$. Otherwise, $d_{D_1}(\phi(v), U^i_{j'})=0$ for all $1\leq j'\leq r$. So $D_1$ is locally $\cP$-balanced.

Property \eqref{item:mover2} will follow from the $\cP$-labelling of $\theta(D_{x\rightarrow y}^{j_2})$.
Note that
\begin{align*}
d_{D_2}(x, U^2_{j'})=\begin{cases}
0&\text{if}\ j'\in \{j_1, j_2\}, \\
1&\text{otherwise}
\end{cases}
\hspace{1cm}\text{and}\hspace{1cm}
d_{D_2}(y, U^2_{j'})=\begin{cases}
1&\text{if}\ j'=j_2, \\
0&\text{otherwise.} 
\end{cases}
\end{align*}
The only other edges $ab$ in $D_2$ of the form $U^1U^2$ are those which appear in the image of $D_e$ for some $e=w^i_{j}w^3_{j'}\in E(D_{x\rightarrow y}^{j_2})$ with $i=1,2$. Note that such $e$ must be incident to $x$ or $y$ and that $a$ and $b$ are new vertices, i.e., $a,b\notin V(D_{x\rightarrow y}^{j_2})$. But for any $v\in \phi(D_e)\cap U^1$, we have $d_{D_2}(v, U^2_{j'})=1$ for every $1\leq j'\leq r$ such that $\phi(v)\notin V_{j'}$.
It follows that \eqref{item:mover2} holds.
\end{proof}

In what follows, given a collection $\cD$ of graphs and an embedding $\phi(D)$ for each $D\in \cD$, we write $\phi(\cD):=\{\phi(D): D\in \cD\}$.

\begin{lemma}\label{lem:mover2}
Let $1/n\ll \gamma\ll \gamma'\leq 1/r\leq 1/3$. Let $V=(V_1, \dots, V_r)$ with $|V_1|=\dots =|V_r|=n$. Let $\cP=\{U^1, U^2\}$ be a $2$-partition of $V$. Let $1\leq j_1\leq r$. Then there exists $\cD\subseteq \{\theta(D_{x\rightarrow y}^j): x,y \in U^1_{j_1}, x\neq y, 1\leq j\leq r, j\neq j_1\}$ such that the following hold.
\begin{enumerate}[\rm(i)]
\item $|\cD|\leq \gamma'n^2$.\label{mover2:1}
\item Each vertex $v\in V$ is a root vertex in at most $\gamma' n$ elements of $\cD$.\label{mover2:2}
\item Suppose that, for each $D\in \cD$, $\phi(D)$ is a copy of $D$ on $V$ which is compatible with its labelling. Suppose further that $\phi(D)$ and $\phi(D')$ are edge-disjoint for all distinct $D, D'\in \cD$. Let $H$ be any $r$-partite graph on $V$ which is edge-disjoint from $\bigcup\phi(\cD)$ and satisfies {\rm(Q\ref{item:Q1})} and {\rm(Q\ref{item:Q2})}. Then there exists $\cD'\subseteq \cD$ such that $H':=H\cup \bigcup\phi(\cD')$
satisfies the following. For all $v\in U^{1}_{j_1}$, and all $1\leq j_2, j_3\leq r$ such that $j_1\neq j_2, j_3$,
\begin{equation*}
d_{H'}(v, U^{2}_{j_2})=d_{H'}(v, U^{2}_{j_3})
\end{equation*}
and for all $1\leq j_2, j_3\leq r$ and all $v\in U^1\setminus (V_{j_1}\cup V_{j_2}\cup V_{j_3})$,
\begin{equation*}
d_{H'}(v, U^{2}_{j_2})-d_{H'}(v, U^{2}_{j_3})=d_{H}(v, U^{2}_{j_2})-d_{H}(v, U^{2}_{j_3}).
\end{equation*}
In particular, $H'$ satisfies {\rm(Q\ref{item:Q1})} and {\rm(Q\ref{item:Q2})}.\label{mover2:3}
\end{enumerate}
\end{lemma}

\begin{proof} Let $p:=\gamma'/4(r-1)$ and $m:=|U^1_{j_1}|$. Define an auxiliary graph $R$ on $U^{1}_{j_1}$ such that $\Delta(R)< 2pm$ and 
\begin{equation}
|N_{R}(S)|\geq p^2m/2\label{eq:NRS}
\end{equation}
for all $S\subseteq U^1_{j_1}$ with $|S|\leq 2$. It is easy to find such a graph $R$; indeed, a random graph with edge probability $p$ has these properties with high probability.

Let $$\cD:=\{\theta(D_{x\rightarrow y}^j), \theta(D_{y\rightarrow x}^j): xy\in E(R), 1\leq j\leq r, j\neq j_1\}.$$
Each vertex of $V$ appears as $x$ or $y$ in some $\theta(D_{x\rightarrow y}^j)$ in $\cD$ at most $2(r-1)\Delta(R)< 4(r-1)pm= \gamma' m$ times. In particular, this implies $|\cD|\leq\gamma' m^2$.
So $\cD$ satisfies \eqref{mover2:1} and \eqref{mover2:2}.

We now show that $\cD$ satisfies \eqref{mover2:3}. Suppose that, for each $D\in \cD$, $\phi(D)$ is a copy of $D$ on $V$ which is compatible with its labelling. Suppose further that $\phi(D)$ and $\phi(D')$ are edge-disjoint for all distinct $D, D'\in \cD$.
Let $H$ be any $r$-partite graph on $V$ which is edge-disjoint from $\bigcup\phi(\cD)$ and satisfies {\rm(Q\ref{item:Q1})} and {\rm(Q\ref{item:Q2})}. 

Let $j_{\text{min}}:=\min\{j: 1\leq j\leq r, j\neq j_1\}$. For each $v\in U^1_{j_1}$ and each $j_\text{min}<j\leq r$ such that $j\neq j_1$, let 
\begin{equation}
f(v,j):=d_{H}(v,U^2_{j})-d_{H}(v, U^2_{j_\text{min}}).\label{eq:fdef}
\end{equation}
By (Q\ref{item:Q2}) and the fact that $\mathcal{P}=\{U_1,U_2\}$ is a $2$-partition, we have
\begin{equation}
|f(v,j)| <\gamma (m+1)<2\gamma m.\label{eq:fbound}
\end{equation}
Let $U^+(j)$ be a multiset such that each $v\in U^{1}_{j_1}$ appears precisely $\max \{f(v,j),0\}$ times. Let $U^-(j)$ be a multiset such that each $v\in U^{1}_{j_1}$ appears precisely $\max \{-f(v,j),0\}$ times. Property
(Q\ref{item:Q1}) implies that $|U^+(j)|=|U^-(j)|$, so there is a bijection $g_j: U^+(j)\rightarrow U^-(j)$.

For each copy $u'$ of $u$  in $U^+(j)$, let $P_{u'}$ be a path of length two whose vertices are labelled, in order, $$\{u\}, U^{1}_{j_1}, \{g_j(u')\}.$$ So $P_{u'}$ has degeneracy two. Let $\cS_j:=\{P_{u'}: u'\in U^+(j)\}$. It follows from \eqref{eq:fbound} that each vertex is used as a root vertex at most $2\gamma m$ times in $\cS_j$ and $|\cS_j|\leq 2\gamma m^2$.
Using \eqref{eq:NRS}, we can apply Lemma~\ref{lma:finding} 
(with $m$, $2$, $3$, $2\gamma$, $p^2/2$ and $R$ playing the roles of $n$, $d$, $b$, $\eta$, $\eps$ and $G$) to find a set of edge-disjoint copies $\cT_j$ of the paths in $\cS_j$ in $R$ which are compatible with their labellings. (Note that we do not require the paths in $\cT_j$ to be edge-disjoint from the paths in $\cT_{j'}$ for $j\neq j'$.) We will view the paths in $\cT_j$ as directed paths whose initial vertex lies in $U^+(j)$ and whose final vertex lies in $U^-(j)$.

For each $j_{\text{min}}<j\leq r$ such that $j\neq j_1$, let $\cD_{j}:=\{\theta(D_{x\rightarrow y}^j): \overrightarrow{xy}\in E(\bigcup \cT_j)\}$. Let
$$\cD':=\bigcup_{\substack{j_{\text{min}}<j\leq r\\j\neq j_1}}\cD_j\subseteq \cD.$$

It remains to show that $H':=H\cup \bigcup\phi(\cD')$ satisfies \eqref{mover2:3}. For each $j_\text{min}< j\leq r$ such that $j\neq j_1$, let $H_j:=\bigcup\phi(\cD_j)$. Consider any vertex $v\in U^1_{j_1}$ and let $j_\text{min}< j_2\leq r$ be such that $j_2\neq j_1$.
Now $v$ will be the initial vertex in exactly $a:=\max\{f(v,j_2), 0\}$ paths and the final vertex in exactly $b:=\max\{-f(v, j_2), 0\}=a-f(v, j_2)$ paths in $\cT_{j_2}$. Let $c$ be the number of paths in $\cT_{j_2}$ for which $v$ is an internal vertex. By definition, $H_{j_2}$ contains $a+c$ graphs $\phi(D)$ where $D$ is of the form $\theta(D_{v\rightarrow y}^{j_2})$ for some $y\in U^1_{j_1}$. Also, $H_{j_2}$ contains $b+c$ graphs $\phi(D)$ where $D$ of the form $\theta(D_{x\rightarrow v}^{j_2})$ for some $x\in U^1_{j_1}$. 
Proposition~\ref{prop:mover}\eqref{item:mover2} then implies that
\begin{align}
d_{H_{j_2}}(v, U^2_{j_2})-d_{H_{j_2}}(v, U^2_{j_\text{min}})&=(b+c)-(a+c)=-f(v, j_2).\label{eq:Hj2}
\end{align}
For any $j_\text{min}< j_3\leq r$ such that $j_3\neq j_1, j_2$, Proposition~\ref{prop:mover}\eqref{item:mover2} implies that
\begin{align}
d_{H_{j_3}} (v, U^2_{j_2})-d_{H_{j_3}}(v, U^2_{j_\text{min}})=0.\label{eq:Hj3}
\end{align}
Equations \eqref{eq:Hj2} and \eqref{eq:Hj3} imply that
$$d_{\bigcup\phi(\cD')}(v, U^2_{j_2})-d_{\bigcup\phi(\cD')}(v, U^2_{j_\text{min}})=d_{H_{j_2}}(v, U^2_{j_2})-d_{H_{j_2}}(v, U^2_{j_\text{min}})=-f(v, j_2),$$
which together with \eqref{eq:fdef} gives
\begin{align}
d_{H'}(v, U^2_{j_2})-d_{H'}(v, U^2_{j_\text{min}})&=d_{H}(v, U^2_{j_2})-d_H(v, U^2_{j_\text{min}})-f(v, j_2)=0.\label{eq:U1jeq}
\end{align}
Thus, for all $v\in U^{1}_{j_1}$ and all $1\leq j_2, j_3\leq r$ such that $j_1\neq j_2, j_3$,
\begin{equation*}
d_{H'}(v, U^{2}_{j_2})=d_{H'}(v, U^{2}_{j_\text{min}})=d_{H'}(v, U^{2}_{j_3}).
\end{equation*}

Finally, consider any $1\leq j_2, j_3\leq r$ and any $v\in U^1\setminus (V_{j_1}\cup V_{j_2}\cup V_{j_3})$. 
Proposition~\ref{prop:mover}\eqref{item:mover2} implies that
\begin{align*}
d_{\bigcup\phi(\cD')}(v, U^2_{j_2})-d_{\bigcup\phi(\cD')}(v, U^2_{j_3})=0,
\end{align*}
so
\begin{align}
d_{H'}(v, U^2_{j_2})-d_{H'}(v, U^2_{j_3})=d_{H}(v, U^2_{j_2})-d_{H}(v, U^2_{j_3}).\label{eq:stillbal}
\end{align}
That $H'$ satisfies (Q\ref{item:Q1}) and (Q\ref{item:Q2}) follows immediately from \eqref{eq:U1jeq} and \eqref{eq:stillbal}.
\end{proof}

Let $\cP=\{U^1,U^2\}$ partition the vertex set $V=(V_1, \dots, V_r)$ with $|V_1|=\dots=|V_r|=n$. We say that a collection $\cD$ of $\cP$-labelled graphs is a \emph{$(\gamma, \gamma')$-degree balancing set}%
\COMMENT{NOTE: degree balancing set is a set of LABELLED graphs}
for the pair $(U^1, U^2)$ if the following properties hold. Suppose that, for each $D\in \cD$, $\phi(D)$ is a copy of $D$ on $V$ which is compatible with its labelling. Suppose further that $\phi(D)$ and $\phi(D')$ are edge-disjoint for all distinct $D, D'\in \cD$.
\begin{enumerate}[\rm(a)]
\item Each $D\in \cD$ has degeneracy at most $r-1$ and $|D|\leq 10r^3$.\label{def:deg:a}
\item $|\cD|\leq \gamma'n^2$.\label{def:deg:b}
\item Each vertex $v\in V$ is a root vertex in at most $\gamma' n$ elements of $\cD$.\label{def:deg:c}
\item For each $D\in \cD$, $\phi(D)$ is $K_r$-divisible and locally $\cP$-balanced.\label{def:deg:d}
\item Let $H$ be any $r$-partite graph on $V$ which is edge-disjoint from $\bigcup\phi(\cD)$ and satisfies {\rm(Q\ref{item:Q1})} and {\rm(Q\ref{item:Q2})}. Then, for each $D\in \cD$, there exists $D'\subseteq D$ such that $\phi(D')$ is $K_r$-divisible and, if $\cD':=\{D': D\in \cD\}$ and  $H':=H\cup \bigcup\phi(\cD')$, then 
\begin{equation*}
d_{H'}(v, U^{2}_{j_1})=d_{H'}(v, U^{2}_{j_2})
\end{equation*}
for all $1\leq j_1,j_2\leq r$ and all $v\in U^{1}\setminus (V_{j_1}\cup V_{j_2})$.\label{def:deg:f}
\end{enumerate}

The following result describes a $(\gamma, \gamma')$-degree balancing set based on the gadgets constructed so far.

\begin{prop}\label{prop:degbalset}
Let $1/n\ll \gamma\ll \gamma'\leq 1/r\leq 1/3$. Let $V=(V_1,\dots, V_r)$ with $|V_1|=\dots=|V_r|=n$. Let $\cP=\{U^1, U^2\}$ be a $2$-partition for $V$. Then $(U^1, U^2)$ has a $(\gamma, \gamma')$-degree balancing set.
\end{prop}

\begin{proof}
Apply Lemma~\ref{lem:mover2} for each $1\leq j_1\leq r$ with $\gamma'/r$ playing the role of $\gamma'$ to find sets $\cD_{j_1}\subseteq \{\theta(D_{x\rightarrow y}^j): x,y\in U^1_{j_1}, x\neq y, 1\leq j\leq r, j\neq j_1\}$ satisfying the properties \eqref{mover2:1}--\eqref{mover2:3}. Let $\cD$ consist of one copy of $\theta(D_{x,y})$ for each $\theta(D_{x\rightarrow y}^j)$ in $\bigcup_{j=1}^r \cD_j$.%
\COMMENT{each copy of $\theta(D_{x,y})$ is a LABELLED graph - we don't actually need to find a copy on $V$ at this point}
We claim that $\cD$ is a $(\gamma, \gamma')$-degree balancing set. Note that each $\theta(D_{x,y})$ satisfies $|\theta(D_{x,y})|\leq 10r^3$ and has degeneracy at most $r-1$ by ($\theta$\ref{theta1}) and ($\theta$\ref{theta2}), so \eqref{def:deg:a} holds. For each $1\leq j\leq r$, $|\cD_j|\leq \gamma'n^2/r$, so \eqref{def:deg:b} holds. Also, each vertex $v\in V$ is used as a root vertex in at most $\gamma' n/r$ elements of each $\cD_j$. Since $\theta(D_{x,y})$ and $\theta(D_{x\rightarrow y}^j)$ have the same set of root vertices, \eqref{def:deg:c} holds. Property \eqref{def:deg:d} follows from Proposition~\ref{prop:mover}\eqref{item:mover1} and \eqref{item:mover3}. 

It remains to show that \eqref{def:deg:f} is satisfied.  Suppose that, for each $D\in \cD$, $\phi(D)$ is a copy of $D$ on $V$ which is compatible with its labelling. Suppose further that $\phi(D)$ and $\phi(D')$ are edge-disjoint for all distinct $D, D'\in \cD$. Let $H$ be any $r$-partite graph on $V$ which is edge-disjoint from $\bigcup\phi(\cD)$ and satisfies {\rm(Q\ref{item:Q1})} and {\rm(Q\ref{item:Q2})}. Using property~\eqref{mover2:3} of $\cD_1$ in Lemma~\ref{lem:mover2}, we can find $\cD_1'\subseteq \cD_1$ such that $H_1:=H\cup \bigcup\phi(\cD_1')$ satisfies {\rm(Q\ref{item:Q1})}, {\rm(Q\ref{item:Q2})} and 
\begin{equation*}
d_{H_1}(v, U^{2}_{j_1})=d_{H_1}(v, U^{2}_{j_2})
\end{equation*}
for all $v\in U^{1}_1$ and all $2\leq j_1, j_2\leq r$.
We can then find $\cD_2'\subseteq \cD_2$ such that $H_2:=H_1\cup \bigcup\phi(\cD_2')$ satisfies
{\rm(Q\ref{item:Q1})}, {\rm(Q\ref{item:Q2})} and 
\begin{equation*}
d_{H_2}(v, U^{2}_{j_1})=d_{H_2}(v, U^{2}_{j_2})
\end{equation*}
for all $v\in U^1_{j}$ where $j=1,2$ and all $1\leq j_1, j_2\leq r$ with $j\neq j_1, j_2$.
Continuing in this way, we eventually find $\cD_r'\subseteq \cD_r$ such that $H_r:=H_{r-1}\cup \bigcup\phi(\cD_{r-1}')$ satisfies
\begin{equation}
d_{H_r}(v, U^{2}_{j_1})=d_{H_r}(v, U^{2}_{j_2})\label{eq:leftright}
\end{equation}
for all $1\leq j_1, j_2\leq r$ and all $v\in U^1\setminus (V_{j_1}\cup V_{j_2})$.

For each $D\in \cD_j$, if $D\in \cD_j'$, then let $D':=D$; otherwise let $D'$ be the empty graph. Let $\cD':=\{D': D\in \bigcup_{j=1}^r\cD_j\}$. For each $D'\in \cD'$, $D'$ is either empty or of the form $\theta(D_{x\rightarrow y}^j)$, so $\phi(D')$ is $K_r$-divisible by Proposition~\ref{prop:mover}\eqref{item:mover1}. By \eqref{eq:leftright}, $\cD'$ satisfies
\eqref{def:deg:f}. So $\cD$ satisfies \eqref{def:deg:a}--\eqref{def:deg:f} and is a $(\gamma, \gamma')$-degree balancing set for $(U^1, U^2)$.
\end{proof}

The following result finds copies of the degree balancing sets described in the previous proposition.

\begin{prop}\label{prop:PD}
Let $1/n \ll \gamma \ll \gamma'\ll 1/k\ll \eps\ll 1/r\leq 1/3$.
Let $G$ be an $r$-partite graph on $(V_1, \dots, V_r)$ with $|V_1|=\dots=|V_r|=n$. Let $\cP=\{U^1, \dots, U^k\}$ be a $k$-partition for $G$. 
Suppose that $d_G(v, U^i_j)\geq (1-1/(r+1)+\eps)|U^i_j|$ for all $1\leq i\leq k$, all $1\leq j\leq r$ and all $v\notin V_j$.
Then there exists a $(\gamma, \cP)$-degree balancing graph $B_{\textnormal{deg}}\subseteq G$ such that $B_{\textnormal{deg}}$ is locally $\cP$-balanced and $\Delta(B_{\textnormal{deg}})<\gamma' n$.
\end{prop}

\begin{proof}
Choose $\gamma_1$, $\gamma_2$ such that $\gamma \ll \gamma_1\ll \gamma_2 \ll \gamma'$. Proposition~\ref{prop:degbalset} describes a $(\gamma, \gamma_1^2)$-degree balancing set $\cD_{i_1, i_2}$ for each pair $(U^{i_1}, U^{i_2})$ with $1\leq i_1<i_2\leq k$.%
\COMMENT{Apply proposition to the vertex set $((U^{i_1}\cup U^{i_2})_1, \dots, (U^{i_1}\cup U^{i_2})_r)$ to find a $(\gamma, \gamma_1^2)$-degree balancing set for $(U^{i_1}, U^{i_2})$.}
Let $\cD:=\bigcup_{1\leq i_1<i_2\leq k}\cD_{i_1, i_2}$. We have $|\cD|\leq k^2\gamma_1^2n^2\leq \gamma_1n^2$ and each vertex is used as a root vertex in at most $k^2\gamma_1^2n\leq \gamma_1n$ elements of $\cD$. By \eqref{def:deg:a}, we can apply Lemma~\ref{lem:emb} (with $\gamma_1$, $\gamma_2$, $r-1$ and $10r^3$ playing the roles of $\eta$, $\eps$, $d$ and $b$) to find edge-disjoint copies $\phi(D)$ of each $D\in \cD$ in $G$ which are compatible with their labellings and satisfy $\Delta(\bigcup\phi(\cD))\leq \gamma_2n$.

Let $G':=G[\cP]-\bigcup\phi(\cD)$ and note that
$$\hat\delta(G')\geq (1-1/(r+1)+\eps)n-\lceil n/k \rceil-\gamma_2 n\geq (1-1/(r+1)+\gamma')n.$$
Apply Lemma~\ref{lem:absorbset} (with $\gamma_2$,%
\COMMENT{$\Delta(\bigcup\phi(\cD))\leq \gamma_2n$, so each vtx appears in at most $\gamma_2n$ of the $\phi(\cD)$}
$\gamma'/2$, $10r^3$ and $G'$ playing the roles of $\eta$, $\eps$, $b$ and $G$) to find an absorbing set $\cA$ for $\phi(\cD)$ in $G'$ such that $\Delta(\bigcup\cA)\leq \gamma'n/2$.

Let $B_{\text{deg}}:=\bigcup\phi(\cD)\cup \bigcup\cA$. Then, $\Delta(B_{\text{deg}}) < \gamma'n$. For all $1\leq i_1<i_2\leq k$, $\cD_{i_1, i_2}$ is a degree balancing set so $\bigcup\phi(\cD_{i_1, i_2})$ is locally $\cP$-balanced by \eqref{def:deg:d}. Since $B_{\text{deg}}[U^i]=\bigcup\phi(\cD)[U^i]$ for each $1\leq i \leq k$, the graph $B_{\text{deg}}$ must also be locally $\cP$-balanced.

We now check that $B_{\text{deg}}$ is a $(\gamma, \cP)$-degree balancing graph. Let $H$ be any $K_r$-divisible graph on $V$ satisfying (Q\ref{item:Q0})--(Q\ref{item:Q2}). Consider any $1\leq i_1<i_2\leq k$. Note that $H[U^{i_1}\cup U^{i_2}]$ satisfies (Q\ref{item:Q0})--(Q\ref{item:Q2}). Since $\cD_{i_1, i_2}$ is a $(\gamma, \gamma')$-degree balancing set for $(U^{i_1}, U^{i_2})$, 
there exist $D'\subseteq D$ for each $D\in \cD_{i_1, i_2}$ such that $\phi(D')$ is $K_r$-divisible and, if $\cD_{i_1, i_2}':=\{D': D\in \cD_{i_1, i_2}\}$ and  $H_{i_1, i_2}':=H\cup \bigcup\phi(\cD_{i_1, i_2}')$,  then
\begin{equation*}
d_{H_{i_1, i_2}'}(v, U^{i_2}_{j_1})=d_{H_{i_1, i_2}'}(v, U^{i_2}_{j_2})
\end{equation*}
for all $1\leq j_1,j_2\leq r$ and all $v\in U^{i_1}\setminus (V_{j_1}\cup V_{j_2})$.
Let $B_{\text{deg}}':=\bigcup_{1\leq i_1<i_2\leq k}\phi(\cD_{i_1, i_2}')$ and let $H':=H\cup B_{\text{deg}}'$. Note that $V(\bigcup\phi(\cD_{i_1, i_2}'))\subseteq U^{i_1}\cup U^{i_2}$  for all $1\leq i_1<i_2\leq k$. So we have
$d_{H'}(v, U^{i}_{j_1})=d_{H'}(v, U^{i}_{j_2})$ for all $2\leq i \leq k$, all $1\leq j_1, j_2\leq r$ and all $v\in U^{<i}\setminus (V_{j_1}\cup V_{j_2})$.

It remains to show that $B_{\text{deg}}$ and $B_{\text{deg}}-B_{\text{deg}}'$ both have $K_r$-decompositions. 
Recall that $\cA$ is an absorbing set for $\phi(\cD)$. So, for any $K_r$-divisible subgraph $D^*$ of any graph in $\phi(\cD)$, $\cA$ contains an absorber for $D^*$. Also, $A$ is $K_r$-decomposable for each $A\in \cA$. Since $\phi(D)$ is $K_r$-divisible for each $D\in \cD$ by \eqref{def:deg:d}, we see that $B_{\text{deg}}$ has a $K_r$-decomposition. Note that, for each $D\in \cD_{i_1, i_2}$, $\phi(D')$ is $K_r$-divisible by \eqref{def:deg:f} and hence $\phi(D)-\phi(D')$ is also $K_r$-divisible. So
$$B_{\text{deg}}-B_{\text{deg}}'= \bigcup\cA\cup \bigcup_{D\in \cD}(\phi(D)-\phi(D'))$$
has a $K_r$-decomposition.
Therefore, $B_{\text{deg}}$ is a $(\gamma, \cP)$-degree balancing graph.
\end{proof}

\subsection{Finding the balancing graph}
Finally, we combine the edge balancing graph and degree balancing graph from Propositions~\ref{prop:PE}~and~\ref{prop:PD} respectively to find a $(\gamma, \cP)$-balancing graph in $G$.

\begin{proofof}{Lemma~\ref{lem:balance}}
Choose constants $\gamma_1$ and $\gamma_2$ such that $\gamma\ll \gamma_1 \ll \gamma_2 \ll \gamma'$. First apply Proposition~\ref{prop:PE} to find a $(\gamma, \cP)$-edge balancing graph $B_{\text{edge}}\subseteq G$ such that $B_{\text{edge}}$ is locally $\cP$-balanced and $\Delta(B_{\text{edge}})<\gamma_1 n$. Now $G':=G-B_{\text{edge}}$ satisfies $d_{G'}(v, U^i_j)\geq (1-1/(r+1)+\eps/2)|U^i_j|$ for all $v\notin V_j$, so we can apply Proposition~\ref{prop:PD} to find a $(\gamma_2, \cP)$-degree balancing graph $B_{\text{deg}}\subseteq G'$ such that $B_{\text{deg}}$ is locally $\cP$-balanced and $\Delta(B_{\text{deg}})<\gamma' n/2$. Let $B:=B_{\text{edge}}\cup B_{\text{deg}}$. Then $\Delta(B)<\gamma'n$ and $B$ is locally $\cP$-balanced. Also, since both $B_{\text{edge}}$ and $B_{\text{deg}}$ are $K_r$-decomposable, $B$ is $K_r$-decomposable.

We now show that $B$ is a $(\gamma, \cP)$-balancing graph. Let $H$ be any $K_r$-divisible graph on $V$ satisfying (P\ref{item:P1}) and (P\ref{item:P2}). Since $B_{\text{edge}}$ is a $(\gamma, \cP)$-edge balancing graph, there exists $B_{\text{edge}}'\subseteq B_{\text{edge}}$ such that $B_{\text{edge}}-B_{\text{edge}}'$ has a $K_r$-decomposition and $H_1:=H\cup B_{\text{edge}}'$ satisfies
$$e_{H_1}(U^{i_1}_{j_1}, U^{i_2}_{j_2})=e_{H_1}(U^{i_1}_{j_1}, U^{i_2}_{j_3})$$
for all $1\leq i_1<i_2\leq k$ and all $1\leq j_1, j_2, j_3\leq r$ with $j_1\neq j_2, j_3$.

Note that $H_1$ is $K_r$-divisible. Also 
$$|d_{H_1}(v, U^i_{j_2})-d_{H_1}(v, U^i_{j_3})|\leq |d_{H}(v, U^i_{j_2})-d_{H}(v, U^i_{j_3})|+\Delta(B_{\text{edge}})<\gamma n+\gamma_1 n\leq \gamma_2 |U^i_{j_1}|$$ for all $2\leq i\leq k$, all $1\leq j_1, j_2, j_3\leq r$ with $j_1\neq j_2, j_3$ and all $v\in U^{<i}_{j_1}$.
So $H_1$ satisfies (Q\ref{item:Q0})--(Q\ref{item:Q2}) with $H_1$ and $\gamma_2$ replacing $H$ and $\gamma$. Now, $B_{\text{deg}}$ is a $(\gamma_2, \cP)$-degree balancing graph so there exists $B_{\text{deg}}'\subseteq B_{\text{deg}}$ such that $B_{\text{deg}}-B_{\text{deg}}'$ has a $K_r$-decomposition and $H_2:=H_1\cup B_{\text{deg}}'$ satisfies
$$d_{H_2}(v, U^i_{j_1})=d_{H_2}(v, U^i_{j_2})$$
for all $2\leq i\leq k$, all $1\leq j_1, j_2\leq r$ and all $v\in U^{<i}\setminus(V_{j_1}\cup V_{j_2})$.

Let $B':=B_{\text{edge}}'\cup B_{\text{deg}}'$. Then $B-B'=(B_{\text{edge}}-B_{\text{edge}}')\cup (B_{\text{deg}}-B_{\text{deg}}')$ has a $K_r$-decomposition. Note that $H\cup B'=H_2$. So $B$ is a $(\gamma, \cP)$-balancing graph.
\end{proofof}


\section{Proof of Theorem~\ref{thm:main}}\label{sec:proof}

In this section, we prove our main result, Theorem~\ref{thm:main}. The idea is to take a suitable partition $\cP$ of $V(G)$, cover all edges in $G[\cP]$ by edge-disjoint copies of $K_r$ and then absorb all remaining edges using an absorber which we set aside at the start of the process. However, for the final step to work, we need that the classes of $\cP$ have bounded size. A key step towards this is the following lemma which, for a partition $\cP$ into a bounded number of parts, finds an approximate $K_r$-decomposition which covers all edges of $G[\cP]$. We then iterate this lemma inductively to get a similar lemma where the parts have bounded size (see Lemma~\ref{lem:iteration}).

\begin{lemma}\label{lem:big}
Let $1/n\ll \alpha \ll \eta \ll \rho\ll 1/k \ll \eps \ll 1/r\leq 1/3$. Let $G$ be a $K_r$-divisible graph on $(V_1, \dots, V_r)$ with $|V_1|=\dots=|V_r|=n$. Let $\cP$ be a $k$-partition for $G$.
For each $x\in V(G)$, each $U\in \cP$ and each $1\leq j\leq r$, let $0\leq d_{x,U_j}\leq |U_j|$.
Let $G_0\subseteq G-G[\cP]$, $G_1:=G-G_0$ and $R\subseteq G[\cP]$.
Suppose the following hold for all $U, U'\in \cP$ and all $1\leq j,j_1, j_2\leq r$ such that $j\neq j_1, j_2$:
\begin{enumerate}[\rm(a)]
	\item for all $x\in U_j$,
$|d_G(x, U_{j_1})-d_G(x, U_{j_2})|< \alpha |U_j|$;\label{big:item:1}
	\item for all $x\notin V_j$,  $d_{G_1}(x, U_{j})\geq (\hat\delta^\eta_{K_r}+\eps)|U_{j}|$;\label{big:item:2}
	\item for all $x\in V(G)$, $d_{R}(x, U_{j})<\rho d_{x,U_j}+\alpha |U_{j}|$;\label{big:item:3}
	\item for all distinct $x,y\in V(G)$, $d_{R}(\{x,y\}, U_{j})<(\rho^2+\alpha)|U_{j}|$;\label{big:item:6}
	\item for all $x\notin U\cup U'\cup V_{j_1}\cup V_{j_2}$, $|d_{R}(x, U_{j_1})-d_{R}(x, U'_{j_2})|< 3\alpha |U_{j_1}|$;\label{big:item:4}
	\item for all $x\notin U$ and all $y\in U$ such that $x,y\notin V_j$,\label{big:item:5}
	$$d_{G_1}(y, N_{R}(x, U_{j}))\geq \rho(1-1/(r-1)) d_{x,U_j}+\rho^{5/4}|U_{j}|.$$
\end{enumerate}
Then there is a subgraph $H\subseteq G_1-G[\cP]$ such that $G[\cP]\cup H$ has a $K_r$-decomposition and $\Delta(H)\leq 4r\rho n$.
\end{lemma}

To prove Lemma~\ref{lem:big}, we apply Lemma~\ref{lem:degree} to cover almost all the edges of $G[\cP]$. We then balance the leftover using Lemma~\ref{lem:balance}. The remaining edges in $G[\cP]$ can then be covered using Corollary~\ref{cor:pseud:parts}. The graph $R$ in Lemma~\ref{lem:big} forms the main part of the graph $G$ in Corollary~\ref{cor:pseud:parts}. Conditions \eqref{big:item:3}--\eqref{big:item:5} ensure that $R$ is `quasirandom'.

\begin{proof}
Write $\cP=\{U^1, \dots, U^k\}$. Let $G_2:=G_1-R=G-G_0-R$. Note that Proposition~\ref{prop:extrem} together with \eqref{big:item:2} and \eqref{big:item:3} implies that for any $1\leq i \leq k$, any $1\leq j\leq r$ and any $x\notin V_j$,
$$d_{G_2}(x, U^i_j)\geq (\hat\delta^\eta_{K_r}+\eps-2\rho)|U^i_j|\geq (1-1/(r+1)+\eps/2)|U^i_j|.$$
Choose constants $\gamma_1, \gamma_2$ such that
$\eta \ll \gamma_1\ll \gamma_2 \ll\rho$. Apply Lemma~\ref{lem:balance} (with
$\gamma_1$, $\gamma_2$, $\eps/2$, $k$, $G_2$, $\cP$
playing the roles of
$\gamma$, $\gamma'$, $\eps$, $k$, $G$, $\cP$)
to find a $(\gamma_1, \cP)$-balancing graph $B\subseteq G_2$ such that 
\begin{equation}
\Delta(B)< \gamma_2n\label{eq:Pdeg}
\end{equation}
and $B$ is locally $\cP$-balanced. As $B$ is also $K_r$-decomposable, for all $1\leq j_1, j_2\leq r$ and all $x\notin V_{j_1}\cup V_{j_2}$,
\begin{equation}
d_{B[\cP]}(x, V_{j_1})=d_{B[\cP]}(x, V_{j_2}).\label{eq:balanced}
\end{equation}

\begin{figure}[t]
\begin{tikzpicture}[node distance=1.6cm]
                    
  \tikzstyle{block} = [rectangle, draw, text centered, rounded corners]
	\tikzstyle{line} = [->,draw]

  \node[block] (1) {$G$};
  \node (2) [above right of=1] {$G_1$};
  \node (3) [right of=1] {$G_0$};
  \node (4) [above right of=2] {$G_2$};
  \node (5) [right of=2] {$R$};
  \node (7) [right of=4] {$G_3$};
  \node[block] (8) [above right of=7] {$\mathcal{F}_1$};
  \node (9) [right of=7] {$G_4$};
  \node (10) [above right of=4] {$B$};
  \node (11) [below right of=9] {$G_5$};
  \node[block] (12) [above right of=10] {$\mathcal{F}_2$};
	\node (space) [right of=12] {};
  \node (13) [right of=space] {$G_6$};
  \node[block] (14) [right of=13] {$\mathcal{F}_3$};
  \node[block] (15) [below of=14, draw, align=left] {$G_6-(G_6[\cP]\cup H_2)$\\ $=G_1-(G[\mathcal{P}]\cup H)$};
  \node[block] (16) [below right of=11] {$G_0$};	
  
  \path[line]  (1) -- (2);
  \path[line]  (1) -- (3);
  \path[line]  (2) -- (4);
  \path[line]  (2) -- (5);
	\path[line]  (3) -- (11);
  \path[line]  (4) -- (7);
  \path[line]  (4) -- (10);	
  \path[line]  (7) -- (8);
  \path[line]  (7) -- (9);
  \path[line]  (5) -- (11);
  \path[line]  (4) -- (11);
  \path[line]  (9) -- (11);
  \path[line]  (10) -- (12);
  \path[line]  (10) -- (13);
  \path[line]  (11) -- (13);
  \path[line]  (11) -- (16);
	\path[line]  (13) -- (14);
  \path[line]  (13) -- (15);
  
\end{tikzpicture}
\caption{Outline for Proof of Lemma~\ref{lem:big}. 
}
\end{figure}
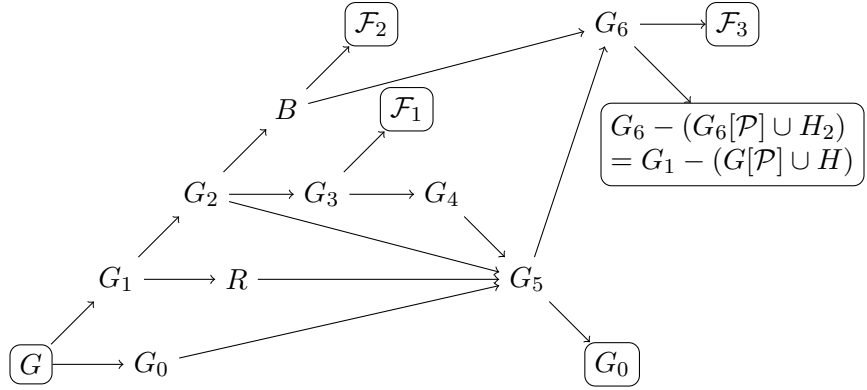

Let $G_3:=G_2[\cP]-B= G[\cP]-R-B$. Then \eqref{big:item:2}, \eqref{big:item:3} and \eqref{eq:Pdeg} give
$$\hat{\delta}(G_3)\geq (\hat\delta^\eta_{K_r}+\eps)n-\lceil n/k\rceil-2\rho n-\gamma_2 n\geq (\hat\delta^\eta_{K_r}+\eps/2)n.$$
Consider any $1\leq j_1, j_2\leq r$ and any $x\notin V_{j_1}\cup V_{j_2}$. Using \eqref{big:item:1}, \eqref{big:item:4} and \eqref{eq:balanced}, we have
\begin{align*}
|d_{G_3}(x, V_{j_1})-d_{G_3}(x, V_{j_2})|&\leq |d_{G[\cP]}(x, V_{j_1})-d_{G[\cP]}(x, V_{j_2})|+|d_{R}(x, V_{j_1})-d_{R}(x, V_{j_2})|\\
&<\alpha n+ 3\alpha n=4\alpha n.
\end{align*}
So we can apply Lemma~\ref{lem:degree} (with $4\alpha$, $\eta$, $\gamma_1/2$, $\eps/2$, $G_3$ playing the roles of $\alpha$, $\eta$, $\gamma$, $\eps$, $G$) to find $G_4\subseteq G_3$ such that $G_3-G_4$ has a $K_r$-decomposition $\cF_1$ and 
\begin{equation}
\Delta(G_4)\leq \gamma_1n/2.\label{eq:G4}
\end{equation}
The graphs $G$, $G_3-G_4$ and $B$ are all $K_r$-divisible (and $G_3-G_4$ and $B$ are edge-disjoint), so
$$G_5:=G-(G_3-G_4)-B=(G-G[\cP]-B)\cup G_4\cup R$$
must also be $K_r$-divisible. Note that $e(G_5\cap B)=0$ and $G_5[\cP]=G_4\cup R$. Consider any $1\leq i \leq k$, any $1\leq j_1, j_2\leq r$ and any $x\notin V_{j_1}\cup V_{j_2}$. If $x\notin U^i$, \eqref{eq:G4} and \eqref{big:item:4} give
\begin{align*}
|d_{G_5}(x, U^i_{j_1})-d_{G_5}(x, U^{i}_{j_2})|&=|d_{G_4\cup R}(x, U^i_{j_1})-d_{G_4\cup R}(x, U^{i}_{j_2})|\\
&\leq \Delta(G_4)+|d_{R}(x, U^i_{j_1})-d_{R}(x, U^{i}_{j_2})|< (\gamma_1/2+3\alpha) n<\gamma_1 n.
\end{align*}
If $x\in U^i$, then we use \eqref{big:item:1}, that $B$ is locally $\cP$-balanced and that $G_4, R\subseteq G[\cP]$ to see that 
\begin{align*}
|d_{G_5}(x, U^i_{j_1})-d_{G_5}(x, U^i_{j_2})|&\leq |d_{G}(x, U^i_{j_1})-d_{G}(x, U^i_{j_2})|+|d_{B}(x, U^i_{j_1})-d_{B}(x, U^i_{j_2})|\\
&<\alpha n\leq \gamma_1n.
\end{align*}
So (P\ref{item:P1}) and (P\ref{item:P2}) in Section~\ref{sec:bal} hold with $G_5$ and $\gamma_1$ replacing $H$ and $\gamma$. Since $B$ is a $(\gamma_1, \cP)$-balancing graph, there exists $B'\subseteq B$ such that $B-B'$ has a $K_r$-decomposition $\cF_2$ and, for all $2\leq i \leq k$, all $1\leq j_1, j_2\leq r$ and all $x\in U^{<i}\setminus (V_{j_1}\cup V_{j_2})$,
\begin{align}
d_{G_5\cup B'}(x, U^i_{j_1})=d_{G_5\cup B'}(x, U^i_{j_2}).\label{eq:degbal}
\end{align}
Write $H_1:=\bigcup_{i=1}^k(B-B')[U^i]$ and let $$G_6:= G_5\cup B'-G_0=(G-G[\cP]-G_0-B)\cup R\cup G_4 \cup B'.$$
Note that 
\begin{equation}\label{eq:G6'}
G_6[\cP]=R\cup G_4 \cup B'[\cP]=G_5[\cP]\cup B'[\cP].
\end{equation}
We now check conditions \eqref{cor:pseud1}--\eqref{cor:pseud4} of Corollary~\ref{cor:pseud:parts} (with $G_6$ playing the role of $G$). Since $G_0\subseteq G-G[\cP]$, \eqref{cor:pseud1} follows immediately from \eqref{eq:degbal}. For \eqref{cor:pseud2}, suppose that $2\leq i \leq k$ and $x\in U^{<i}$. For any $1\leq j\leq r$, using \eqref{big:item:3}, \eqref{eq:G4} and \eqref{eq:Pdeg}, we have
\begin{align}
d_{G_6}(x, U^i_j)&\stackrel{\mathclap{\eqref{eq:G6'}}}{\leq} d_R(x,U^i_j) +\Delta(G_4)+\Delta(B)<\rho d_{x,U^i_j}+ \alpha |U^i_j|+\gamma_1n/2+\gamma_2n\nonumber\\
&\leq \rho d_{x,U^i_j}+2\gamma_2n.\label{eq:G6}
\end{align}
Consider any $y\in N_{G_6}(x, U^i)$.
Note that $G_6[U^i]=G_1[U^i]-(B-B')[U^i]$. So, for any $1\leq j\leq r$ such that $x, y\notin V_{j}$, we have%
\COMMENT{final inequality: $\rho^{1/4}\leq 1/18kr^2\implies 9kr^2\rho^{3/2}\leq \rho^{5/4}/2$.}
\begin{align*}
d_{G_6}(y, N_{G_6}(x, U^i_{j}))&\geq d_{G_6}(y, N_{R}(x, U^i_{ j}))
\geq d_{G_1}(y, N_{R}(x, U^i_{ j}))-\Delta(B)\\
&\stackrel{\mathclap{\eqref{big:item:5}, \eqref{eq:Pdeg}}}{\geq}(1-1/(r-1))\rho d_{x,U^i_j}+\rho^{5/4}|U^i_j|-\gamma_2n\\
&\stackrel{\mathclap{\eqref{eq:G6}}}{\geq} (1-1/(r-1))d_{G_6}(x, U^i_j)+\rho^{5/4}|U^i_j| -3\gamma_2n\\
&> (1-1/(r-1))d_{G_6}(x, U^{i}_j)+9kr\rho^{3/2}|U^i|.
\end{align*}
So \eqref{cor:pseud2} holds.

To see that $G_6$ satisfies property \eqref{cor:pseud3} of Corollary~\ref{cor:pseud:parts}, note that for all $2\leq i\leq k$ and all distinct $x,x'\in U^{<i}$, \eqref{big:item:6}, \eqref{eq:Pdeg}, \eqref{eq:G4} and \eqref{eq:G6'} imply that
\begin{align*}
|N_{G_6}(x,U^i)\cap N_{G_6}(x',U^i)|&\leq d_R(\{x,x'\}, U^i)+\Delta(G_4)+\Delta(B) \\
&< (\rho^2+\alpha)|U^i|+\gamma_1n/2+\gamma_2n\leq 2\rho^2|U^i|.
\end{align*}
Finally, by \eqref{big:item:3}, \eqref{eq:Pdeg}, \eqref{eq:G4} and \eqref{eq:G6'}, for any $y\in U^i$, we have that
\begin{align*}
d_{G_6}(y, U^{<i})\leq \Delta(R)+\Delta(G_4)+\Delta(B) \leq 3\rho n/2\leq  2k\rho |U^i_1|,
\end{align*}
and \eqref{cor:pseud4} holds.
Hence we can apply Corollary~\ref{cor:pseud:parts} to $G_6$ to find a subgraph $H_2\subseteq G_6-G_6[\cP]$ such that $G_6[\cP]\cup H_2$ has a $K_r$-decomposition $\cF_3$ and $\Delta(H_2)\leq 3r\rho n$.
Set $H:=H_1\cup H_2\subseteq G_1-G[\cP]$. We have $\Delta(H)\leq \Delta(H_1)+\Delta(H_2)\leq \Delta(B)+\Delta(H_2)\leq 4r\rho n$. Now,
\begin{align*}
G[\cP]\cup H&=G_2[\cP]\cup R\cup H=G_3\cup R\cup H\cup B[\cP]\\
&=\bigcup\cF_1\cup G_4\cup R\cup H \cup B[\cP]=\bigcup\cF_1\cup G_5[\cP]\cup H_1\cup H_2\cup B[\cP]\\
&=\bigcup(\cF_1\cup \cF_2) \cup G_5[\cP]\cup H_2\cup B'[\cP]\stackrel{\eqref{eq:G6'}}{=}\bigcup(\cF_1\cup \cF_2)\cup G_6[\cP]\cup H_2\\
&=\bigcup(\cF_1\cup \cF_2\cup \cF_3).
\end{align*}
So $G[\cP]\cup H$ has a $K_r$-decomposition $\cF_1\cup \cF_2\cup \cF_3$.
\end{proof}

We now iterate Lemma~\ref{lem:big}, applying it to each partition $\cP_i$ in a partition sequence $\cP_1, \dots, \cP_\ell$ for $G$. This allows us to cover all of the edges in $G[\cP_\ell]$ by edge-disjoint copies of $K_r$, leaving only a small remainder in $\bigcup_{U\in \cP_\ell} G[U]$.

\begin{lemma}\label{lem:iteration}
Let $1/m\ll \alpha \ll \eta \ll \rho\ll 1/k \ll \eps \ll 1/r\leq 1/3$. Let $G$ be a $K_r$-divisible graph on $(V_1, \dots, V_r)$ with $|V_1|=\dots=|V_r|=n$. Let $\cP_1, \dots, \cP_\ell$ be a $(1,k,\hat\delta^\eta_{K_r}+\eps/2, m)$-partition sequence for $G$. For each $1\leq q\leq \ell$, each $1\leq j\leq r$, each $U\in \cP_q$ and each $x\in V(G)$, let $0\leq d_{x,U_j}\leq |U_j|$ be given. Let $\cP_0:=\{V(G)\}$ and, for each $0\leq q \leq \ell$, let $G_q:=G[\cP_q]$. Let $R_1, \dots, R_\ell$ be a sequence of graphs such that $R_q\subseteq G_q-G_{q-1}$ for each $q$. Suppose the following hold for all $1\leq q\leq \ell$, all $1\leq j,j_1, j_2\leq r$ such that $j\neq j_1, j_2$, all $W\in \cP_{q-1}$, all distinct $x,y\in W$ and all $U, U'\in \cP_q[W]$:
\begin{enumerate}[\rm (i)]
\item if $q\geq 2$, $\cP_q[W]$ is a $(1,k,\hat\delta^\eta_{K_r}+\eps)$-partition for $G[W]$;\label{big:item:0'}
\item if $x\in U_j$,
$|d_G(x, U_{j_1})-d_G(x, U_{j_2})|< \alpha |U_j|$;\label{big:item:1'}
	\item $d_{R_q}(x, U_{j})<\rho d_{x, U_j}+\alpha |U_{j}|$;\label{big:item:3'}
	\item $d_{R_q}(\{x,y\}, U_{j})<(\rho^2+\alpha)|U_{j}|;$\label{big:item:6'}
	\item if $x\notin U\cup U'\cup V_{j_1}\cup V_{j_2}$, $|d_{R_q}(x, U_{j_1})-d_{R_q}(x, U'_{j_2})|< 3\alpha |U_{j_1}|$;\label{big:item:4'}
	\item if $x\notin U$, $y\in U$ and $x,y \notin V_j$, then $$d_{G_{q+1}'}(y, N_{R_q}(x, U_j))\geq \rho (1-1/(r-1))d_{x, U_j}+\rho^{5/4}|U_j|$$ where $G_{q+1}':=G_{q+1}-R_{q+1}$ if $q\leq \ell-1$ and $G_{\ell+1}':=G$.\label{big:item:5'}
\end{enumerate}
Then there is a subgraph $H\subseteq \bigcup_{U\in \cP_\ell}G[U]$ such that $G-H$ has a $K_r$-decomposition.
\end{lemma}

\begin{proof}
We will use induction on $\ell$. If $\ell=1$, apply Lemma~\ref{lem:big} (with $\eps/2$, $\cP_1$, $R_1$ and the empty graph playing the roles of $\eps$,  $\cP$, $R$ and $G_0$) to find $H'\subseteq G-G[\cP_1]$ such that $G[\cP_1]\cup H'$ has a $K_r$-decomposition. Letting $H:= G-G[\cP_1]-H'\subseteq \bigcup_{U\in \cP_\ell}G[U]$, shows the result holds for $\ell=1$.

Suppose then that $\ell\geq 2$ and the result holds for all smaller $\ell$. Note that for each $1\leq j\leq r$, each $x\notin V_j$ and each $U\in \cP_1$, $d_{G[\cP_2]-R_2}(x, U_{j})\geq (\hat\delta^\eta_{K_r}+\eps/3)|U_{j}|$, since $R_2$ satisfies \eqref{big:item:3'} and $\cP_1, \dots, \cP_\ell$ is a $(1,k,\hat\delta^\eta_{K_r}+\eps/2,m)$-partition sequence for $G$. So we may apply Lemma~\ref{lem:big} (with $\eps/3$, $\cP_1$, $R_1$, $G$ and $(G-G[\cP_2])\cup R_2$ playing the roles of $\eps$,  $\cP$, $R$, $G$ and $G_0$) to find $H'\subseteq G[\cP_2]-(G[\cP_1]\cup R_2)$ such that $G[\cP_1]\cup H'$ has a $K_r$-decomposition $\cF_1$ and $\Delta(H')\leq 4r\rho n$. Let $G^*:=G-G[\cP_1]-H'=G-\bigcup \cF_1$, so $G^*$ is $K_r$-divisible. Observe that $G^*=\bigcup_{U\in \cP_1}G^*[U]$, so $G^*[U]$ is $K_r$-divisible for each $U\in \cP_1$.

Consider any $U\in \cP_1$. We check that $$G^*[U],\cP_2[U], \dots, \cP_\ell[U], R_2[U], \dots, R_\ell[U]$$ satisfy the conditions of Lemma~\ref{lem:iteration}. Since $\Delta(H')\leq 4r\rho n\leq \eps n/4k^2$, $\cP_2[U]$ is a $(1,k,\hat\delta^\eta_{K_r}+\eps/2)$-partition for $G^*[U]$. 
For any $3\leq q \leq \ell$ and any $W\in \cP_{q-1}$, $G^*[W]=G[W]$ since $H'\subseteq G[\cP_2]$. So \eqref{big:item:0'} holds and $\cP_2[U], \dots, \cP_\ell[U]$ is a $(1,k,\hat\delta^\eta_{K_r}+\eps/2,m)$-partition sequence for $G^*[U]$. For \eqref{big:item:1'}, note that for any $2\leq q \leq \ell$, any $1\leq j\leq r$, any $U'\in \cP_q[U]$ and any $x\in U'$, $d_{G^*}(x, U'_j)=d_G(x, U'_j)$. Conditions \eqref{big:item:3'}--\eqref{big:item:4'} are automatically satisfied.
To see that \eqref{big:item:5'} holds, note that for any $2\leq q\leq \ell$ and any $U'\in \cP_q[U]$, $G^*_{q+1}[U']=G_{q+1}[U']$ since $H'\subseteq G[\cP_2]$.

So we can apply the induction hypothesis to $G^*[U], \cP_2[U], \dots, \cP_\ell[U], R_2[U], \dots, R_\ell[U]$ to obtain a subgraph $H_U\subseteq \bigcup_{U'\in \cP_\ell[U]}G^*[U']$ such that $G^*[U]-H_U$ has a $K_r$-decomposition $\cF_U$. 
Set $H:=\bigcup_{U\in \cP_1}H_U$. Then, $H\subseteq \bigcup_{U\in \cP_\ell} G[U]$ and $G-H$ has a $K_r$-decomposition $\cF_1\cup \bigcup_{U\in \cP_1}\cF_U$.
\end{proof}

We are now ready to prove Theorem~\ref{thm:main}.

\begin{proofof}{Theorem~\ref{thm:main}}
Let $n_0\in \N$ and $\eta>0$ be such that $1/n_0\ll\eta\ll\eps$ and choose additional constants $\eta_1$, $m'$, $\alpha$, $\rho$ and $k$ such that
$$1/n_0 \ll \eta_1\ll 1/m'\ll \alpha\ll \eta\ll \rho\ll 1/k\ll \eps.$$
Let $G$ be any $K_r$-divisible graph on $(V_1, \dots, V_r)$ with $|V_1|=\dots=|V_r|=n \geq n_0$ and $\hat\delta(G)\geq (\hat{\delta}^\eta_{K_r}+\eps)n$.
Apply Lemma~\ref{lem:partseq} to find an $(\alpha, k, \hat\delta^\eta_{K_r}+\eps-\alpha, m)$-partition sequence $\cP_1, \dots, \cP_\ell$ for $G$ where $m'\leq m \leq km'$.
So in particular, by (S\ref{S3}), for each $1\leq q \leq \ell$, all $1\leq j_1, j_2, j_3\leq r$ with $j_1\neq j_2, j_3$, each $U\in \cP_{q}$ and each $x\in U_{j_1}$,
\begin{equation}
|d_G(x, U_{j_2})-d_G(x, U_{j_3})|<\alpha |U_{j_1}|.\label{eq:locbal}
\end{equation}
Let $\cP_0:=\{V(G)\}$ and $G_q:=G[\cP_q]$ for $0\leq q\leq \ell$. Note that $\hat\delta^\eta_{K_r}+\eps-\alpha\geq 1-1/r+\eps$ (with room to spare) by Proposition~\ref{prop:extrem}. So we can apply Corollary~\ref{cor:randoms} to find a sequence of graphs $R_1, \dots, R_{\ell}$ such that $R_q\subseteq G_q-G_{q-1}$ for each $1\leq q\leq \ell$ and the following holds. For all $1\leq q\leq \ell$, all $1\leq j, j'\leq r$, all $W\in \cP_{q-1}$, all distinct $x,y\in W$ and all $U, U'\in \cP_q[W]$,
\begin{align}
\begin{split}
d_{R_q}(x, U_j)<\rho d_{G_q}(x, U_j)+\alpha |U_j|;&\\\label{eq:rqprop}
d_{R_q}(\{x,y\}, U_j)<(\rho^2+\alpha)|U_j|;&\\
|d_{R_q}(x, U_{j})-d_{R_q}(x, U'_{j'})|< 3\alpha |U_j|& \hspace{6pt}\text{ if }x\notin U\cup U'\cup V_j \cup V_{j'};\\
d_{G_{q+1}'}(y, N_{R_q}(x, U_j))\geq \rho (1-1/(r-1))d_{G_{q}}(x, U_j)+\rho^{5/4}|U_j|& \hspace{6pt}\text{ if } x\notin U, y\in U \text{ and } x,y \notin V_j,
\end{split}
\end{align}
where $G_{q+1}':=G_{q+1}-R_{q+1}$ if $q\leq \ell-1$ and $G_{\ell+1}':=G$.

Let $\cH:=\{G[U]: U\in \cP_\ell\}$. Each $H\in \cH$ satisfies $|H|\leq rm$.
Note that
$$\hat{\delta}(G[\cP_1]-R_1) \geq (\hat\delta^\eta_{K_r}+\eps)n-\lceil n/k \rceil -2\rho n>(1-1/(r+1)+\eps/2)n.$$
So we can apply Lemma~\ref{lem:absorbset} (with $\eta_1$, $\alpha$, $rm$ and $G[\cP_1]-R_1$ playing the roles of $\eta$, $\eps$, $b$ and $G$) to find an absorbing set $\cA$ for $\cH$ inside $G[\cP_1]-R_1$ such that $A^*:=\bigcup\cA$ satisfies $\Delta(A^*)\leq \alpha n$.

Let $G^*:=G-A^*$. Note that both $G$ and $A^*$ are $K_r$-divisible, so $G^*$ is $K_r$-divisible. Since $\Delta(A^*)\leq \alpha n$ and $A^*\subseteq G[\cP_1]$, $\cP_1, \dots, \cP_\ell$ is an $(1, k, \hat\delta^\eta_{K_r}+\eps/2,m)$-partition sequence for $G^*$. For each $1\leq q\leq \ell$, each $1\leq j\leq r$, each $U\in \cP_q$ and each $x\in V(G)$, set $d_{x,U_j}:=d_{G_q}(x, U_j)$. Using \eqref{eq:locbal}, \eqref{eq:rqprop} and that $A^*\subseteq G[\cP_1]$, we see that $G^*$, the partition sequence $\cP_1, \dots, \cP_\ell$ and the sequence of graphs $R_1, \dots, R_\ell$ satisfy properties \eqref{big:item:0'}--\eqref{big:item:5'} of Lemma~\ref{lem:iteration} (with $\eps-\alpha$ playing the role of $\eps$).  So we may apply Lemma~\ref{lem:iteration} to find $H\subseteq \bigcup_{U\in \cP_\ell}G^*[U]$ such that $G^*-H$ has a $K_r$-decomposition $\cF_1$.

Note that $H$ is a $K_r$-divisible subgraph of $\bigcup_{U\in \cP_\ell}G[U]$, so for each $U\in \cP_\ell$, $H[U]\subseteq G[U]$ is $K_r$-divisible. Since $\cA$ is an absorbing set for $\cH$, it contains a distinct absorber for each $H[U]$. So $H\cup A^*$ has a $K_r$-decomposition $\cF_2$. Thus
$G=(G^*-H)\cup (H\cup A^*)$ has a $K_r$-decomposition $\cF_1\cup \cF_2$.
\end{proofof}

\section*{Acknowledgements}
We are grateful to Peter Dukes and Richard Montgomery for helpful discussions, and to the referees for a careful reading of the manuscript.

\medskip

{\footnotesize \obeylines \parindent=0pt

Ben Barber
Heilbronn Institute for Mathematical Research
School of Mathematics
University of Bristol
University Walk
Bristol BS8 1TW
UK
}
\begin{flushleft}
{\it{E-mail address}:
\tt{b.a.barber@bristol.ac.uk}}
\end{flushleft}

{\footnotesize \obeylines \parindent=0pt

Daniela K\"{u}hn, Allan Lo, Deryk Osthus, Amelia Taylor 
School of Mathematics
University of Birmingham
Edgbaston
Birmingham
B15 2TT
UK
}
\begin{flushleft}
{\it{E-mail addresses}:
\tt{\{d.kuhn, s.a.lo, d.osthus\}@bham.ac.uk}, a.m.taylor@pgr.bham.ac.uk}
\end{flushleft}

\end{document}